\documentclass[11pt]{amsart}

\usepackage{times,amsfonts,amsmath,amstext,amsbsy,amssymb,
amsopn,amsthm,upref,eucal}
\usepackage[T1]{fontenc}

\newcommand \Ker {{\rm Ker}}
\newcommand \R {{ \mathbb R}}
\newcommand \C {{ \mathbb C}}
\newcommand \Z {{ \mathbb Z}}
\newcommand \N {{ \mathbb N}}
\newcommand \T {{ \mathbb T}}

\newcommand \re {{ \operatorname{Re} }}
\newcommand \im {{ \operatorname{Im} }}

\DeclareMathOperator{\Spec}{Spec}

\newtheorem{theorem}{Theorem}[section]
\newtheorem {lemma} {Lemma}[section]

\newtheorem{corollary}{Corollary}[section]
\newtheorem{remark}{Remark}[section]

\newtheorem{definition}{Definition}[section]

\title[Limit Theorems for Horocycle Flows]%
{Limit Theorems for Horocycle Flows}

\author{Alexander Bufetov} \author{Giovanni Forni}

\address{
Department of Differential Equations\\
Steklov Institute of Mathematics\\
Moscow, Russian Federation\\
Department of Mathematics\\
Rice University\\
Houston, TX USA}
\address{Department  of Mathematics\\
  University of Maryland \\
  College Park, MD USA}

\begin{document}

  \begin{abstract}

The main results of this paper are limit theorems for horocycle flows on compact surfaces
of constant negative curvature. 

One of the main objects of the paper is a special family of horocycle-invariant
finitely additive H{\"o}lder measures on rectifiable arcs. An asymptotic formula for ergodic integrals for horocycle flows
is obtained in terms of the finitely-additive measures, and limit theorems follow as a corollary of the asymptotic formula.

The objects and results of this paper are similar to those in \cite{Fone}, \cite{Ftwo},  \cite{Bufetov1} and \cite{Bufetov2} for translation flows on flat surfaces. The arguments are based on the classification of invariant distributions for horocycle flows established in \cite{FlaFo}.

  \end{abstract}
  \maketitle

\tableofcontents
\section{Introduction}
\subsection{Outline of the main results.}

The aim of this paper is to obtain limit theorems for horocycle flows on compact surfaces of constant negative curvature.

Our limit theorems admit the simplest formulation in the case when the smallest positive eigenvalue 
$\mu_0$ of the Laplace operator on the surface of curvature $-1$ is strictly less than $1/4$ (equivalently, when the spectral decomposition of the space of square-integrable functions on our surface into irreducible unitary representations of the modular group contains representations of the complementary series). 

In this case, the variance of the ergodic integrals  (up to time $T>0$) of a generic smooth function 
grows at the rate $T^{\frac{1+\nu_0}{2}}$, where  $\nu_0:=\sqrt{1-4\mu_0}$, and its ergodic integrals, normalized to have variance $1$, converge in distribution to a nondegenerate compactly supported measure on the real line.

The situation is more complicated for surfaces whose spectral decomposition only contains representations of the principal series (or more generally for functions supported on irreducible
representations of the principal series).

In this case, the variance of  ergodic integrals  (up to time $T>0$) of any smooth function which is
not a coboundary grows at the rate $T^{\frac{1}{2}}$, but its ergodic integrals, normalized to have variance $1$, converge in distribution to an orbit of an infinite-dimensional \emph{quasi-periodic }flow in the space of random variables with compactly supported distributions. The frequencies of this quasi-periodic motion are determined by the eigenvalues larger than $1/4$ of the Laplace-Beltrami operator on the hyperbolic surface. We are not able to determine whether the limit distribution exists in this case; we conjecture that it does not. In fact, the limit distribution will exist for all smooth functions which are not coboundaries if and only if all random variables in each of the invariant subtori of our infinite dimensional torus have exactly the same probability distribution (see \S\S~\ref{rotsym}).

Our argument relies on the classification, due to Flaminio and Forni \cite{FlaFo},
of distributions (in the sense of S.~L.~Sobolev and L.~Schwartz\footnote{The term ``distribution'' is used in two very different senses in our paper: first, probability distributions of random variables and, second,
distributions of Sobolev and Schwartz. For instance, ``limit distributions'' refer to the first meaning,
while ``invariant distributions'' to the second. We hope that our precise meaning is always clear from the context.}) invariant under a given horocycle flow. One of the main objects of our paper is a closely related space of finitely-additive H{\"o}lder measures on rectifiable arcs on our surface, invariant under the complementary horocycle flow. We classify these measures and establish an explicit bijecitve correspondence between them and the subspace of the Flaminio-Forni space given by invariant distributions corresponding to positive eigenvalues of the Casimir operator. This isomorphism yields a natural duality between the spaces of invariant distributions for the two horocycle flows on a surface. We further establish an asymptotic formula for ergodic integrals in terms of the finitely-additive measures.
Our limit theorems are obtained as corollaries of the asymptotic formula. Informally, the limit theorems claim that the normalized ergodic integrals of horocycle flows converge in distribution to the probability distributions of finitely-additive measures of horocycle arcs.

The objects and results of this paper are similar to those in \cite{Fone}, \cite{Ftwo} and especially \cite{Bufetov1}, \cite{Bufetov2}, \cite{Bufetov3} for translation flows on flat surfaces. The methods here are completely different, however, and are based on those in \cite{FlaFo}.

The remainder of this section is organized as follows. In \S\S~\ref{histremarks} we make some brief historical remarks. In \S\S~\ref{Definitions} we establish our notation and
recall the main properties of invariant distributions and basic currents for horocycle flows. In \S\S~\ref{HolderCurrents} we state our main theorems on  finitely additive H\"older measures on rectifiable arcs, invariant with respect to the unstable (stable) horocycle (Theorem~\ref{thm:hatbetaprops}) and on the related additive cocycle for the  stable (unstable) horocycle (Theorem~\ref{thm:cocycleproperties}). We also state several important corollaries of the above mentioned theorems (Corollary~\ref{cor:wuvanishing} and Corollary~\ref{cor:basiccurrents}). In \S\S~\ref{Ergodic} we state our results on the asymptotics of ergodic integrals, in particular we state an approximation theorem for ergodic integrals of sufficiently smooth zero-average functions in terms of additive cocycles (Theorem~\ref{thm:approximation}) and our results on limit distributions of normalized ergodic integrals
(Theorem~\ref{thm:limitdistcompl}  and Theorem~\ref{thm:limit_torus}). We then state our conditional results about the existence of limit distributions for functions supported on irreducible components of the principal series follow (Corollary~\ref{cor:limitdistprincipal1} and Corollary~\ref{cor:limitdistprincipal1}).
In \S\S~\ref{Duality} we introduce currents of dimension $2$ (and degree $1$) associated
to our finitely additive measures on rectifiable arcs. We then state a duality theorem which affirms
that such currents can be written in terms of invariant distributions for the unstable (stable) horocycle flow (Theorem~\ref{thm:duality}). The duality theorem leads to a complete classification of the
class of finitely additive H\"older measures axiomatically defined in \S\S~\ref{HolderCurrents} (see
Definition~\ref{def:hatbetaprops}), in the sense that our construction gives the space of all
finitely additive H\"older measures with the listed properties (Theorem~\ref{thm:classification}).
It also allows us to establish a direct relations between the lifts of  our additive cocyles to the universal cover and $\Gamma$-conformal invariant distributions on the boundary of the Poincar\'e disk
(Theorem~\ref{thm:gammaconformal}).

\subsection{Historical remarks.}
\label{histremarks}
The classical horocycle flow on a compact surface of constant negative curvature is a main example of a unipotent, parabolic flow. Its ergodic theory has been extensively studied. It is known that the flow is minimal \cite{He}, uniquely ergodic~\cite{Fu}, has Lebesgue spectrum and is therefore strongly mixing~\cite{Pa}, in fact mixing of all orders \cite{Ma}, and has zero entropy~\cite{Gu}. 
Its finer ergodic and rigidity properties, as well as the rate of mixing, were investigated by M.~Ratner
is a series of papers \cite{Ra1}, \cite{Ra2}, \cite{Ra3}, \cite{Ra4}. In  joint work with L.~Flaminio \cite{FlaFo}, the second author has proved precise bounds on ergodic integrals of smooth functions.
Those bounds already imply, as proved in \cite{FlaFo}, that all weak limits of probability distributions of normalized ergodic integrals of generic smooth functions have (non degenerate) compact support.

In the case of finite-volume surfaces, the classification of invariant measures is due to Dani \cite{dani}.
The asymptotic behaviour of averages along closed horocycles in the finite-volume case has been studied  by D.~Zagier~\cite{Za}, P.~Sarnak~\cite{Sa}, D.~Hejal \cite{Hj} and more recently in \cite{FlaFo}
and by A.~Str{\"o}mbergsson \cite{St}. The flows on general geometrically finite surfaces have been studied by M.~Burger~\cite{Bu}.

Invariant distributions, and, more generally, eigendistributions for smooth dynamical systems were 
already  considered in 1955 by S.~V.~Fomin \cite{Fomin}, who constructed a full system of
eigendistributions for a linear toral automorphism. 

In the case of horospherical foliations of symmetric spaces $X=G/K$ of non-compact type of connected semi-simple Lie groups $G$ with finite center, invariant distributions are related to  {\it conical distributions} on the space of horocycles introduced in the work of S.~Helgason \cite{Helgason}.

Invariant distributions for the horocycle flow appear in the asymptotics for the equidistribution of long closed horocycle on finite-volume non-compact hyperbolic surfaces in work of P.~Sarnak \cite{Sa}. To the authors' best knowledge this is the first appearance of invariant distributions in the context of 
quantitative equidistribution. Sarnak's work was later generalized to arbitrary horocycle arcs and
to the horocyle  flow also in the compact case in \cite{FlaFo} (see also~\cite{Hj}, ~\cite{St}).

Other similar  (parabolic, uniquely ergodic) systems for which the asymptotics and the limit distributions of ergodic integrals have been studied include translation flows on surfaces of higher genus and interval exchange transformations, substitution dynamical systems and Vershik's automorphisms, and nilflows on homogeneous spaces of the Heisenberg group. The latter are related to the asymptotic behaviour of theta sums. For translation flows and interval exchange transformations, results on the growth of ergodic integrals were proved conditionally in the work of A.~Zorich \cite{Zorich1}, \cite{Zorich2}, \cite{Zorich3} and
M.~Kontsevich \cite{Kontsevich} and later fully proved in \cite{Ftwo} and by A.~Avila and M.~Viana~\cite{AV}. An asymptotic 
formula for ergodic integrals and limit theorems for translation flows were obtained  in~\cite{Bufetov1}, ~\cite{Bufetov2} and~\cite{Bufetov3}. Similar results for suspension flows over Vershik's automorphisms were obtained in \cite{Bufetov1}. 
Limit theorems for theta sums were proved by W.~B.~Jurkat 
and  J.~W.~Van Horne \cite{JvH1}, \cite{JvH2}, by J.~Marklof~\cite{Marklof} and more recently in stronger form by F.~Cellarosi~\cite{Cellarosi}.  Invariant distributions and asymptotics of ergodic integrals for Heisenberg nilflows were studied in \cite{FlaFonil}, which generalizes the asymptotics
for theta sums proved by  H.~Fiedler, W.~B.~Jurkat and O.~K{\"o}rner \cite{FJK} .

\subsection{Definitions and notation.}
\label{Definitions}
Let $\Gamma$ be a co-compact lattice in $PSL(2,\R)$ and let $M:=
\Gamma \backslash D$ be the corresponding hyperbolic surface
obtained as a quotient of the Poncar\'e disk $D$ under standard
action of $\Gamma$ by linear fractional transformations. Since
$PSL(2,\R)$ acts freely and transitively on the unit tangent bundle
of the Poncar\'e disk, the unit tangent bundle $SM$ of the
hyperbolic surface $M$ can be identified with the homogeneous space
$\Gamma \backslash PSL(2,\R)$. Let $\{X, U, V\}$ be the basis of the
Lie algebra $\mathfrak sl (2,\R)$ of $PSL(2,\R)$ given by the
infinitesimal generators of the geodesic flow and of the stable and
unstable horocycle flows. The following commutation relations hold:
\begin{equation}
\label{eq:commutLie} [X,U]=U \,, \quad  [X,V] =-V  \,,  \quad
[U,V]= 2 X \,.
\end{equation}

\smallskip
 Let $\{\hat X, \hat U, \hat V\}$ be the frame of  the cotangent bundle dual to the frame $\{X,U,V\}$
 of the tangent bundle, that is,
 \begin{equation}
 \begin{aligned}
 \hat X (X) &= 1\,, \quad   \hat X (U) = 0 \, \quad  \hat X (V) = 0 \,; \\
  \hat U (X) &= 0\,, \quad   \hat U (U) = 1 \, \quad  \hat U (V) = 0 \,; \\
   \hat V (X) &= 0\,, \quad   \hat V (U) = 0 \, \quad  \hat V (V) = 1\,.
 \end{aligned}
 \end{equation}

Let $\vert \hat X\vert$, $\vert \hat U\vert$ and $\vert \hat V\vert$  denote the $1$-dimensional
measures on $SM$ tranverse to the $2$-dimensional foliations $\{\hat X=0\}$, $\{\hat U=0\}$
and $\{\hat V=0\}$ given respectively by the $1$-forms $\hat X$, $\hat U$ and  $\hat V $.
In other terms, if $\gamma$ is any rectifiable path transverse to the foliation $\{\hat X=0\}$, (
$\{\hat U=0\}$,  $\{\hat V=0\}$), then respectively,
$$
\int_\gamma \vert \hat X\vert  = \vert \int_\gamma \hat X \vert \quad
\left (\int_\gamma \vert \hat U\vert  = \vert \int_\gamma \hat U \vert\,, \quad
\int_\gamma \vert \hat V\vert  = \vert \int_\gamma \hat V \vert  \right).
$$

\smallskip
Let $g_t :=\exp(tX)$, $h^U_t =\exp (tU)$ and $h^V_t =\exp (tV)$ be
the corresponding one-paramer groups. Since $PSL(2,\R)$ acts on $SM$
on the right, the following commutation relations hold for the
flows:
\begin{equation}
\label{eq:commutflows} g_t \circ h^U_s = h^U_{e^{-t} s}\circ g_t
\quad  \text{ and } \quad g_t \circ h^V_s = h^U_{e^{t} s} \circ
g_t\,, \quad \text{ for all }\, s,t \in \R\,.
\end{equation}
Thus  the flows $\{h^U_t\}$ and $\{h^V_t\}$ are respectively the
stable and unstable horocycle flows for the hyperbolic geodesic flow
$\{g_t\}$  on $SM$.

Let $L^2(SM)$ be the Hilbert space of
square-integrable complex-valued functions on $SM$, endowed with its 
usual Hilbert space norm 
$\Vert\cdot\Vert$.
By the theory of unitary representations of the unimodular group (see for instance
\cite{Ba}, \cite{GF} \cite{GN}), the Hilbert space $L^2(SM)$ of
splits as an orthogonal sum
\begin{equation}
\label{eq:L2splitting}
L^2(SM) = \bigoplus_{\mu \in \Spec(\square)} H_\mu
\end{equation}
of irreducible unitary representations of  $PSL(2,\R)$ which are
parametrized by the value $\mu \in \R^+ \cup \{-n^2 +n \vert n\in
\Z^+\}$ of the \emph{Casimir operator}
$$
\square :=  X^2 +X +VU = X^2 -X +UV = X^2 + UV + VU \,.
$$
The Casimir operator is a second order differential operator which
generates the center of the enveloping algebra of the Lie algebra
$\mathfrak sl(2, \R)$, hence its restriction to every irreducible
unitary representation is a scalar multiple of the identity. The
unitary type of irreducible unitary representations of $PSL(2,\R)$
is uniquely determined by the value of the Casimir parameter.
Irreducible unitary representations are divided into three series:
the \emph{principal series} consists of all representations with
Casimir parameter $\mu \geq 1/4$, the \emph{complementary series} of
all representations with Casimir parameter $0<\mu <1/4$ and the
\emph{discrete series} of representations with Casimir parameter
$\mu=-n^2+n$.

{\bf Remark.} In formula~(\ref{eq:L2splitting}) and everywhere below the eigenvalues of the Casimir operator are always understood with multiplicities.

 \medskip
Let us consider the stable horocycle flow $\{h^U_t\}$. Similar
statements hold for the unstable horocycle flow. It was proved in
\cite{FlaFo} that for every Casimir parameter $\mu \in \R^+$ the
space of \emph{invariant distributions} for the horocycle flow which
are non-trivial on the space $C^\infty(H_\mu):= C^{\infty}(SM) \cap
H_\mu$ has dimension equal to $2$ and (if $\mu \neq 1/4$) it is
generated by eigenvector for the action of the geodesic flow.

We recall that an invariant distribution for the stable horocycle  flow
is a distribution $D\in \mathcal D'(SM)$ such that $\mathcal L_U
D=0$ in the sense of distributions. Let $\mathcal I_U(SM) \subset \mathcal D'(SM)$
denote the space of all invariant distributions for the stable horocycle flow.

For every Casimir parameter $\mu \in \R^+$ let $\nu := \sqrt{
1-4\mu}$. We remark that $\nu \in \C$ is purely imaginary if
$\mu\geq 1/4$ (principal series) and $\nu \in (0,1)$ if $0<\mu <1/4$
(complementary series).  By \cite{FlaFo}, Theorem 3.2 and Lemma 3.5,
the space of invariant distributions for the stable horocycle flow
which are non-trivial on $C^\infty(H_\mu)$ has a basis $\{D^+_\mu,
D^-_\mu\}$ such that (in the sense of distributions)
\begin{equation}
\label{eq:eigenv} \mathcal L_X  D^\pm_\mu = - \frac{1\pm \nu}{2}
D^\pm_\mu\,, \quad \text{\rm for all }\, \mu\in
\R^+\setminus\{1/4\}\,,
\end{equation}
while for the for the special case $\mu=1/4$ ($\nu=0$),
\begin{equation}
\label{eq:Jordanb} \mathcal L_X  \begin{pmatrix} D^+_{\mu} \\
D^-_{\mu}\end{pmatrix} = -\frac{1}{2}  \begin{pmatrix} 1 & 0 \\ 1 &
1 \end{pmatrix} \,
\begin{pmatrix}
D^+_{\mu} \\ D^-_{\mu}\end{pmatrix} \,.
\end{equation}

Invariant distributions for the horocycle flow are naturally associated (by a general
construction which holds for any volume preserving flow) with \emph{basic currents} for
the horocycle foliation. A current $C$ of degree $2$ (and dimension $1$) is called
{\it basic }for the orbit foliation of the stable horocycle  flow if and only if
\begin{equation}
 \mathcal L_U C  = \imath_U C =0 \,\, \text{ in }\,\, \mathcal D'(SM)\,.
\end{equation}
(The operators $ \mathcal L_U$ and $\imath_U$ are respectively the
Lie derivative and the contraction with respect to the horocycle
generator $U$ acting on currents according to the standard
definition).  Let $\mathcal B_U(SM)$ denote the space of all basic currents
for the stable horocycle foliation.

For any $s>0$, let $W^s(SM)$ and $\Omega_1^s(SM)$ be respectively the $L^2$ Sobolev spaces
of functions and of $1$-forms on $SM$ and let $W^{-s}(SM)$ and $\Omega_1^{-s}(SM)$ denote
the dual Sobolev spaces of distributions. Let then
$$
\begin{aligned}
\mathcal I^{-s}_U(SM) :=  \mathcal I_U(SM) \cap W^{-s}(SM) \,,  \\
\mathcal B^{-s}_U(SM) :=  \mathcal B_U(SM) \cap \Omega_1^{-s}(SM) \,.
\end{aligned}
$$

Let $\omega$ denote the volume form on $SM$ and let $\eta_U:= \imath_U\omega$ denote the
contraction of the volume form along the stable horocycle. The $2$-form $\eta_U$ is closed
since the horocycle flow is volume preserving. We will show in Lemma~\ref{lemma:Isom} below that  the space  $\mathcal I_U(SM)$
of all invariant distributions and the space $\mathcal B_U(SM)$ are identified
by the isomorphism
\begin{equation}
\label{eq:star5}
D \to   D \eta_U
\end{equation}
which maps the Sobolev space $\mathcal I^{-s}_U(SM)$  isometrically
onto the Sobolev space $\mathcal B^{-s}_U(SM)$, for all $s>0$, hence it maps
$\mathcal I_U(SM)$ onto $\mathcal B_U(SM)$.

It was proved in \cite{FlaFo}, \S 3.2, that the invariant distributions $D^\pm_\mu$ have
Sobolev orders equal to $(1\pm \re\,\nu)/2$ (that is, $D^\pm_\mu\in W^{-s}(SM)$ for
all $s>(1\pm \re\,\nu)/2$). It was later proved by S.~Cosentino in~\cite{Co} that $D^\pm_\mu$
are in fact H\"older of the same orders (that is, they can be written as first derivatives of
H\"older continuous functions of  exponent $1\mp \re\,\nu)/2$, except for the distribution $D^-_{1/4}$
which can be written as a first derivative of a H\"older continuous function of any exponent
$\alpha < 1/2$).

\smallskip
\emph{Notation}: The Lie derivative $\mathcal L_W$ of a distribution
or a current with respect to a smooth vector field  $W$ is defined
in the standard weak sense (based on the formula of integration by
parts). For consistence, the action of a smooth flow $\{\phi_t\}$ on
a current $C$ is defined by pull-back as follows:
$$
(\phi^\ast_t C) (\lambda) = C( \phi_{-t}^\ast \lambda) \,, \quad
\text{ for any smooth form } \lambda\,.
$$
In particular, with the above convention the following identity
holds. Let $W$ be the infinitesimal generator of the smooth flow
$\{\phi_t\}$. For all $t\in \R$,
$$
\frac{d}{d t} \phi^\ast_t C   =   \phi_t^\ast \mathcal L_W C\,.
$$

\subsection{H{\"o}lder currents and H{\"o}lder cocycles.}
\label{HolderCurrents}

One of the main objects of this paper is a  space of
finitely-additive H{\"o}lder measures defined on the semi-ring of all rectifiable arcs in $SM$.
These measures are invariant under the {\it unstable} horocycle flow $h_t^V$ and will be seen to govern the  asymptotics of  ergodic integrals for the {\it stable} horocycle flow $h_t^U$.

\begin{definition}  \label{def:hatbetaprops} Let $\hat {\mathfrak  B}_V(SM)$  be the space of all functionals $\hat \beta$ which to every rectifiable arc $\gamma\subset SM$ assign a complex number
$\hat \beta(\gamma)\in {\mathbb C}$ so that the following holds:
\begin{enumerate}
\item (Additive property)   For any decomposition $\gamma= \gamma_1 +\gamma_2$ into subarcs,
\begin{equation*}
\hat \beta(\gamma)= \hat \beta(\gamma_1) + \hat \beta(\gamma_2) \,;
\end{equation*}
\item (Weak unstable vanishing) For all $\gamma$ tangent to the weak unstable
foliation,
$$
\hat \beta (\gamma) =0\,.
$$
\item (Unstable horocycle invariance)  For all $t \in \R$,
\begin{equation*}
\hat \beta(h^V_{t}\gamma ) =  \beta(\gamma)
\,.
\end{equation*}
\item (H\"older property) There exists an exponent  $\alpha \in (0,1)$ and  a constant $C>0$  such that
for all rectifiable arcs $\gamma$
satisfying
$$
\int_{\gamma}
\vert \hat U \vert\leq 1
$$
we have
\begin{equation*}
\vert  \hat \beta(\gamma) \vert \leq C \left( 1+
\int_{\gamma} \vert \hat X\vert + \int_{\gamma} \vert \hat U \vert
\int_{\gamma} \vert \hat V \vert \right) \left(  \int_{\gamma}
\vert \hat U \vert\right)^{\alpha}.
\end{equation*}
\end{enumerate}
\end{definition}

The space $\hat {\mathfrak  B}_V(SM)$ contains a sequence of special
elements $$\{\hat \beta_{\mu}^{\pm} \vert \mu \in \Spec(\square) \cap \R^+\}$$ described in the following Theorem. We prove below that the above set is in fact a basis of $\hat {\mathfrak  B}_V(SM)$ endowed with a natural Sobolev-type Hilbert space structure.

\begin{theorem}  \label{thm:hatbetaprops}
For any positive Casimir parameter $\mu >0$  there exist two independent (normalized) finitely-additive measures
$\hat \beta^\pm_\mu$ such that the following holds.

 For all rectifiable arcs $\gamma$ in $SM$ the following properties hold:
\begin{enumerate}
\item (Additive property)   For any decomposition $\gamma= \gamma_1 +\gamma_2$ into subarcs,
\begin{equation*}
\hat\beta^{\pm}_\mu(\gamma)= \hat \beta^{\pm}_\mu(\gamma_1) +
\hat\beta^{\pm}_\mu(\gamma_2) \,;
\end{equation*}
\item (Geodesic scaling)  For all $t \in \R$ and for $\mu\neq 1/4$,
\begin{equation*}
\hat\beta^{\pm}_\mu (g_{-t}\gamma ) =   \exp( \frac{1\mp \nu}{2} t)
\hat\beta^{\pm}_\mu(\gamma) \,,
\end{equation*}
while for $\mu =1/4$ ($\nu=0$),
\begin{equation*}
\begin{pmatrix} \hat\beta^+_{1/4}(g_{-t}\gamma ) \\ \hat\beta^-_{1/4} (g_{-t}\gamma ) \end{pmatrix}  =
\exp( \frac{t}{2}) \begin{pmatrix} 1 & -\frac{t}{2} \\ 0 & 1 \end{pmatrix}
\begin{pmatrix}\hat\beta^+_{1/4}(\gamma ) \\ \hat\beta^-_{1/4} (\gamma ) \end{pmatrix} \,;
\end{equation*}
\item (Unstable horocycle invariance)  For all $t \in \R$,
\begin{equation*}
\hat\beta^{\pm}_\mu (h^V_{t}\gamma ) =  \hat \beta^{\pm}_\mu(\gamma)\,.
\end{equation*}
\item (H\"older property) There exists a constant $C>0$ such that,
for all rectifiable arc $\gamma \subset SM$,  for all $\mu \neq 1/4$,
\begin{equation}
\label{eq:Hoelder}
\vert  \hat \beta^{\pm}_\mu(\gamma) \vert \leq C \left( 1+
\int_{\gamma} \vert \hat X\vert + \int_{\gamma} \vert \hat U \vert
\int_{\gamma} \vert \hat V \vert \right) \left( \int_{\gamma}
\vert \hat U \vert\right)^{\frac{ 1\mp \re\,\nu}{2}}
\end{equation}
and, for $\mu =1/4$ ($\nu=0$),
\begin{equation}
\label{eq:Hoelderbis}
\begin{aligned}
\vert  \hat \beta^+_{1/4}(\gamma) \vert &\leq C  \left( 1+
\int_{\gamma} \vert \hat X\vert + \int_{\gamma} \vert \hat U \vert
\int_{\gamma} \vert \hat V \vert \right) \left( \int_{\gamma}
\vert \hat U \vert\right)^{\frac{ 1}{2}+} \,, \\
\vert  \hat \beta^-_{1/4}(\gamma) \vert &\leq C \left( 1+
\int_{\gamma} \vert \hat X\vert + \int_{\gamma} \vert \hat U \vert
\int_{\gamma} \vert \hat V \vert \right) \left( \int_{\gamma}
\vert \hat U \vert\right)^{\frac{ 1}{2}} \,.
\end{aligned}
\end{equation}
\end{enumerate}
\end{theorem}
\emph{Notation}: In the above formulas~(\ref{eq:Hoelder}) and~(\ref{eq:Hoelderbis}), the symbols
 $\vert \hat X\vert$, $\vert \hat U \vert$ and $\vert \hat V \vert$ stand for the transverse
 measures given by the forms $\hat X$, $\hat U$ and $\hat V$ respectively, and, for any $L>0$, we set
$$
L^{\frac{ 1}{2}+} =  L^{\frac{ 1}{2}} (1 +\vert \log L\vert).
$$

Recall that by definition the {\it weak unstable} manifolds of the geodesic flow 
are the $2$-dimensional manifolds tangent to the integrable distribution 
$\{X,V\}$ in the tangent bundle of $SM$. If follows immediately from the  H\"older
property that the restrictions of the finitely-additive measures $\hat
\beta^\pm_\mu$ to the weak unstable manifolds of the geodesic flow
vanish. 

\begin{corollary}
\label{cor:wuvanishing} For all Casimir parameters $\mu \in \R^+$
and for any rectifiable arc  $\gamma_{wu}$ contained in a (single)
weak unstable manifold of the geodesic flow, we have
$$
\hat \beta^\pm_\mu( \gamma_{wu}) =0\,.
$$
In particular, all the finitely-additive measures $\hat
\beta^\pm_\mu$ belong to the space $\hat {\mathfrak  B}_V(SM)$.
\end{corollary}

For any $s>0$,  let $\mathcal B_+^{-s}(SM)$ be the (closed) subspace of basic currents for the stable horocycle foliation
supported on irreducible unitary representations of the principal and complementary series and let $\mathcal B^{-s}_+:
\Omega_1^{-s}(SM) \to \mathcal B_+^{-s}(SM)$  be the orthogonal projection. By the Sobolev embedding theorem, for
any $s>3/2$, any rectifiable arc $\gamma$ can be seen as a current in the dual Sobolev space $ \Omega_1^{-s}(SM)$.

\begin{corollary}
\label{cor:basiccurrents}

For any $r>9/2$, for any $s>r+1$ and for any rectifiable arc $\gamma\subset SM$, the limit
$$
\hat B(\gamma):= \lim_{t\to +\infty} (g^\ast_{-t}  \circ \mathcal B^{-r}_+ \circ g^\ast_t)(\gamma)  \in
\Omega_1^{-s}(SM)
$$
exists and is equal to a uniquely determined basic current for the stable horocycle foliation. In fact, there exists a basis  $\{ B^\pm_\mu�\} \subset \mathcal B_U(SM)$ of eigenvectors  for the action of the geodesic flow on the space of basic currents such that
$$
\hat B(\gamma)=  \sum_{\mu \in \Spec(\square) \cap \R^+}
 \hat \beta^+_\mu(\gamma) B^+_\mu + \hat \beta^-_\mu(\gamma) B^-_\mu \,.
$$
For all rectifiable arcs $\gamma$ in $SM$ the following properties hold:
\begin{enumerate}
\item (Additive property)   For any decomposition $\gamma= \gamma_1 +\gamma_2$ into subarcs,
\begin{equation*}
\hat B(\gamma)= \hat B(\gamma_1) + \hat B(\gamma_2) \,;
\end{equation*}
\item (Weak unstable vanishing) For all $\gamma$ tangent to the weak unstable
foliation,
$$
\hat B (\gamma) =0\,.
$$
\item (Unstable horocycle invariance)  For all $t \in \R$,
\begin{equation*}
\hat B (h^V_{t}\gamma ) = \hat B(\gamma)\,.
\end{equation*}
\item (H\"older property) There exist exponents $\alpha^\pm \in (0,1)$ and a constant $C>0$
such that, for all rectifiable arc $\gamma \subset SM$, we have
\begin{equation*}
\Vert \hat B(\gamma)\Vert_{-s} \leq C \left( 1+
\int_{\gamma} \vert \hat X\vert + \int_{\gamma} \vert \hat U \vert
\int_{\gamma} \vert \hat V \vert \right)\max_{ \alpha\in \{\alpha^+, \alpha^-\}} \left( \int_{\gamma}
\vert \hat U \vert\right)^{\alpha}
\end{equation*}
\end{enumerate}
\end{corollary}

\begin{remark} It is unclear to the authors whether the dependence of the current
$\hat B(\gamma)\in \Omega'_1(SM)$ on the rectifiable arc $\gamma\subset SM$ is
continuous with respect to a natural topology (for instance the Hausdorff topology)
on the space of bounded rectifiable arcs (with common endpoints).
\end{remark}

For any sufficiently smooth $1$-form $\lambda\in \Omega_1(SM)$, let $\hat \beta_\lambda$ be the finitely additive functional defined on rectifiable arcs $\gamma\subset SM$ as
\begin{equation}
\label{eq:betalambda}
\hat \beta_\lambda (\gamma) := <\hat B(\gamma), \lambda> \,.
\end{equation}
It follows from Corollary~\ref{cor:basiccurrents} that $\hat \beta_\lambda \in\hat {\mathfrak B}_V(SM)$. In particular, for any sufficiently
smooth complex-valued function $f$ on $SM$, let  $\hat \beta_f \in \hat {\mathfrak B}_V(SM)$ be finitely additive functional
$\hat \beta_{f\hat U}$, that is, for any rectifiable arc $\gamma\subset SM$,
\begin{equation}
\label{eq:betaf}
\hat \beta_f (\gamma) = \hat \beta_{f \hat U}(\gamma)=  <\hat B(\gamma), f\hat U>\,.
\end{equation}
By Corollary~\ref{cor:basiccurrents} and the identification between basic currents and invariant distributions
given by (\ref{eq:star5}),
the finitely-additive measure $\hat \beta_f $ has the expansion:
\begin{equation}
\label{eq:expansion1}
\hat \beta_f= \sum_{\mu \in \Spec(\square) \cap \R^+}
 D^+_\mu(f)  \hat \beta^+_\mu  + D^-_\mu(f)  \hat \beta^-_\mu  \,.
\end{equation}

\begin{remark}
The formula (\ref{eq:expansion1}) yields a {\it duality} between the spaces of $V$-invariant
distributions and $V$-basic currents. We describe this duality in detail in \S\S~\ref{Duality}.
\end{remark}

By restriction of  the finitely additive measures  $\hat \beta\in \hat {\mathfrak B}_V(SM)$ to horocycle arcs,
we obtain finitely-additive H{\"o}lder cocycles $\beta$ for the stable horocycle flow $\{h^U_t\}$.

 For any $(x,T) \in SM\times \R$ Let $\gamma_U(x,T)$ denote the oriented horocycle arc
$$
\gamma_U(x,T) := \{ h^U_t(x) \vert t\in [0,T] \} \,.
$$
For every Casimir parameter $\mu >0$ the cocycles $\beta^\pm_\mu$
are defined as follows:
\begin{equation}
\label{eq:additivecocycles_mu} \beta^{\pm}_\mu (x,T) := \hat
\beta^{\pm}_\mu [\gamma_U(x,T)] \,, \quad \text{ for all } (x,T)\in
SM\times \R \,.
\end{equation}
For any sufficiently smooth complex-valued function $f$  on $SM$, the cocycle
$\beta_f$ is similarly defined by the formula
\begin{equation}
\label{eq:additivecocycles_f} \beta_f(x,T) := \hat
\beta_f [\gamma_U(x,T)] \,, \quad \text{ for all } (x,T)\in
SM\times \R \,.
\end{equation}
By construction and by formula~(\ref{eq:expansion1}), the following expansion formula holds:
\begin{equation}
\label{eq:expansion2}
 \beta_f(x,T) = \sum_{\mu \in \Spec(\square) \cap \R^+}
 D^+_\mu(f)  \beta^+_\mu  + D^-_\mu(f) \beta^-_\mu  \,.
\end{equation}
Thus for every Casimir parameter $\mu\in \R^+$ we obtain a pair of (linearly independent)
additive H\"older cocycles $\beta^\pm_\mu: SM \times \R \to \C$ for the stable horocycle flow.
Such  cocycles  have the following properties.

\begin{theorem}  \label{thm:cocycleproperties}
For any Casimir parameter $\mu \in \R^+$ the following holds.
\begin{enumerate}
\item (Cocycle property)   For all $x\in SM$ and for all $S$, $T\in \R$:
\begin{equation*}
\beta^{\pm}_\mu(x,S+T)=  \beta^{\pm}_\mu(x,S) +
\beta^{\pm}_\mu(h^U_S x,T) \,;
\end{equation*}
\item (Geodesic scaling)  For $\mu \neq  1/4$, for all $x\in SM$, for all $t$, $T\in \R$,
\begin{equation*}
\beta^{\pm}_\mu (g_{-t}x, Te^t ) =   \exp( \frac{1\mp \nu}{2} t)
\beta^{\pm}_\mu(x,T) \,,
\end{equation*}
and for $\mu= 1/4$ ($\nu=0$), for all $x\in SM$, for all $t$, $T\in \R$,
\begin{equation*}
\begin{pmatrix} \beta^+_{1/4} (g_{-t}x, Te^t ) \\ \beta^-_{1/4}(g_{-t}x, Te^t )\end{pmatrix}
=   \exp( \frac{t}{2}) \begin{pmatrix} 1 & -\frac{t}{2} \\ 0 & 1 \end{pmatrix}
\begin{pmatrix} \beta^+_{1/4} (x, T) \\ \beta^-_{1/4} (x, T )\end{pmatrix} \,;
\end{equation*}
\item (H\"older property)  For all $\mu\neq 1/4$, there exists a constant $C_\mu>0$
such that, for all $(x,T) \in SM \times \R^+$,
\begin{equation*}
\vert  \beta^{\pm}_\mu(x,T) \vert  \leq   C_\mu \, \vert
T\vert^{\frac{1\mp \re\,\nu}{2}}  \,,
\end{equation*}
while for $\mu=1/4$ ($\nu=0$), there exists a constant $C>0$ such that
\begin{equation*}
\begin{aligned}
\vert  \beta^+_{1/4}(x,T) \vert  &\leq   C \, \vert
T\vert^{\frac{1}{2}+}  \,, \\
 \vert  \beta^-_{1/4}(x,T) \vert  &\leq   C  \, \vert
T\vert^{\frac{1}{2}} \,.
\end{aligned}
\end{equation*}
\item (Orthogonality)  For any $T\in \R$, the bounded function
$\beta^\pm_\mu(\cdot,T)$ belongs to  the irreducible component $H_\mu \subset L^2(SM)$.
\end{enumerate}
\end{theorem}

\subsection{H{\"o}lder cocycles and ergodic integrals.}
\label{Ergodic}

The asymptotics and limit distributions  of ergodic integrals of smooth functions  is controlled
by the additive H\"older cocycles for the horocycle flow introduced above. More precisely, the following approximation theorem holds.
\begin{theorem}
\label{thm:approximation} For any $s>11/2$ there exists a constant $C_s>0$ such that for every rectifiable curve $\gamma \subset SM$  and for all $1$-forms $\lambda\in \Omega_1^s(SM)$ supported
on irreducible components of the principal and complementary series, we have
$$
\vert \int_{\gamma}\lambda -  \hat B_\lambda(\gamma) \vert  \leq  C_s \Vert \lambda \Vert_s
(1 + \int_\gamma \vert \hat X\vert + \int_\gamma \vert \hat V\vert)\,.
$$
In particular, for all functions $f\in W^s(SM)$ supported on irreducible components of the principal and complementary series, we have
$$
\vert \int_{\gamma} f \hat U -  \hat \beta_f(\gamma) \vert  \leq  C_s \Vert f \Vert_s
(1 + \int_\gamma \vert \hat X\vert + \int_\gamma \vert \hat V\vert)\,.
$$
\end{theorem}

By the results of \cite{FlaFo} it is possible to derive a logarithmic upper bound in the uniform norm for the the ergodic integrals along horocycle orbits of functions supported on irreducible components of the discrete series. Theorem~\ref{thm:approximation}  therefore implies the following:
\begin{corollary}
\label{cor:approximation}
 For any $s>11/2$ there exists a constant $C_s>0$ such that for all zero-average functions $f\in W^s(SM)$ and all $(x,T)\in SM\times \R^+$ we have
$$
\vert \int_0^T f \circ h^U_t (x) dt  -  \beta_f(x,T) \vert  \leq  C_s \Vert f \Vert_s
(1+ \log^+ \vert T\vert)\,.
$$
\end{corollary}

\begin{remark} 
\label{rk:nonzero}
Corollary \ref{cor:approximation} and the {\it lower} bounds proved in \cite{FlaFo} on the $L^2$ 
norm of ergodic integrals imply, in particular, that the cocycles  $\beta^{\pm}_{\mu}(x,T)$ do not vanish identically as a function of $x\in SM$, for any $T\neq 0$. Indeed, to see this, it suffices to
apply Corollary \ref{cor:approximation} to any  function $f\in C^\infty(SM)$ such that $\beta_f=\beta_{\mu}^{\pm}$. Observe also that, by the Ergodic Theorem, if $f$ has zero average on $SM$, then, for any fixed $T>0$, the function $\beta_f(x,T)$ also has zero average on $SM$.
\end{remark}

From Corollary~\ref{cor:approximation} we derive the following limit theorems. Recall that for 
any function $f \in L^2(SM)$ the symbol $\Vert f \Vert$ stands for its $L^2$-norm.
For any \emph{zero-average }real-valued function $f \in L^2(SM)$, for all $t>0$ and $T \in \R$, let
$\mathfrak M_t(f)$ be  the probability distribution on the real line of the random variable on $SM$ defined by the formula
\begin{equation}
\label{eq:Ef}
\mathcal E_t(f,T):= \frac{\int\limits_0^{Te^t} f\circ h_t^U(\cdot) \,dt }{ \Vert
 \int\limits_0^{Te^t} f\circ h_t^U(\cdot) \,dt \Vert}\,.
\end{equation}
We are interested in the asymptotic behaviour (as $t\to +\infty$) of the probability distributions 
$\mathfrak M_t(f,T)$ for $T\in [0,1]$. 

Let $f$ be a  smooth function with non-zero orthogonal projection 
onto irreducible components of the complementary series. Let
$$
f = \sum_{\mu\in  \Spec(\square)} f_\mu
$$
denote the decomposition of $f$ with respect to a splitting  of the space $L^2(SM)$
into irreducible components. Let
$$
\mu_f := \min \{ \mu \in \Spec(\square)\setminus\{0\} \vert  f_\mu \not =0 \}
$$
let $H_1, \dots, H_k\subset L^2(SM)$ be all the irreducible components of Casimir parameters
$\mu_1= \dots= \mu_k=\mu_f$. Let $\{D^\pm_1, \dots, D^\pm_k\}$ denote the basis of distributional eigenvectors of the geodesic flow of the space of invariant distributions for the horocycle flow supported on $\mathcal D'(H_1)\oplus \dots\oplus \mathcal D'(H_k)$. Let $\{\beta^\pm_1, \dots, \beta^\pm_k\}$ be the corresponding cocycles for the horocycle flow. For every $T\in \R$, let $P_{cp}(f,T)$ denote
the probability distribution on the real line of the random variable on $SM$
$$
 \frac{ \sum_{i=1}^k D^-_i(f) \beta^-_i(\cdot,T)} {\left ( \sum_{i=1}^k \vert D^-_i(f)\vert^2
  \Vert \beta^-_i\Vert^2   \right)^{1/2} }\,.
$$
By Remark \ref{rk:nonzero} and the orthogonality of cocycles, 
the above function is bounded, non-constant and has zero average on $SM$.  The probability measure
 $P_{cp}(f,T)$ is therefore non-atomic and has compact  support on the real line. 

Let $d_{LP}$ denote the L{\'e}vy-Prohorov metric on the space of probability measures on the real line.
We recall that on any separable metric space, hence, in particular, on the real line, 
the L{\'e}vy-Prohorov metric induces the weak$^*$ topology on the space of probability measures 
(see, e.g. \cite{Billingsley}).

\begin{theorem}
\label{thm:limitdistcompl}
There exists a constant $\alpha>0$ depending only on the surface $M$ such that the following holds.
For any  $s>11/2$ there exists a constant $C_s>0$ depending only on $s$  such that the following holds.  Let
$f \in W^s(SM)$ be any real-valued function of zero average such that the Casimir parameter
$\mu_f \in (0,1/4)$  and  $(D^-_1(f), \dots, D^-_k(f)) \not=(0, \dots, 0)$. Let $\nu_f := \sqrt{ 1-\mu_f} \in
(0,1)$.   Then
\begin{enumerate}
\item
For all $T\in [0,1]$ and all $t>0$, we have
 \begin{equation}
 \left\vert   \frac{\Vert \int_0^{Te^t} f\circ h^U_\tau (x)d\tau\Vert}{ e^{ \frac{1+\nu_f}{2} t}
  \left ( \sum_{i=1}^k \vert D^-_i(f)\vert^2  \Vert \beta^-_i\Vert^2   \right)^{1/2} } \,-\, 1  \right\vert
  \leq C_s \Vert f\Vert_s e^{-\alpha t}\,.
 \end{equation}
\item
For any $T\in [0,1]$  we have the convergence in distribution
 \begin{equation}
\mathfrak M_t(f,T)\rightarrow P_{cp}(f,T)\  {\rm as} \ t\to\infty
 \end{equation}
with the following estimate that holds for all $t>0$ uniformly in $T\in [0,1]$:
 \begin{equation}
\label{eq:dlp:compl}
d_{LP}\left( \mathfrak M_t(f,T), P_{cp}(f,T)\right) \leq C_s \Vert f\Vert_s e^{-\alpha t}\,.
 \end{equation}
\end{enumerate}
\end{theorem}

\begin{remark} The estimate in formula~(\ref{eq:dlp:compl}), uniform in $T\in \R$ over any compact 
interval, implies, in particular, that Theorem \ref{thm:limitdistcompl} can be strengthened to a {\it functional} limit theorem: the convergence in distribution holds in the space of measures on the space $C[0,1]$ as well, similarly to the limit theorems of \cite{Bufetov1}, \cite{Bufetov2}, \cite{Bufetov3}.
\end{remark}

Now we prove that, for sufficiently smooth functions supported on irreducible components
of the principal series, normalized ergodic integrals converge in distribution
on $SM$ to a quasi-periodic motion on an infinite-dimensional torus.

Let $\{\mu_n\}$ be the sequence of Casimir parameter in the interval $(1/4, +\infty)$ (listed with multiplicities). For all $n\in \N$, let $\upsilon_n:= \sqrt{4\mu_n-1} \in \R^+$.
The isotypical components of the decomposition of $L^2(SM)$ into irreducible representations, being eigenspaces of the Casimir operator, are closed under complex conjugation. It follows that there exists an orthogonal decomposition of $L^2(SM)$  into irreducible components each closed under complex conjugation. Let $\{D^\pm_{\mu_n}\}$ denote the corresponding sequence of  horocycle invariant distributions, and let $\{\beta^\pm_{\mu_n} \}$ the  sequence of  additive H\"older  cocycles described in Theorem~\ref{thm:cocycleproperties}. By the characterization of the distributions  $\{D^\pm_{\mu_n}\}$ as distributional eigenvectors of the geodesic flow and by the construction of the cocycles  $\{\beta^\pm_{\mu_n} \}$, it follows that, for all $n\in \N$, 
\begin{equation}
\label{eq:conjugation}
D^-_{\mu_n} = \overline{ D^+_{\mu_n}} \quad \text{ and } \quad \beta_{\mu_n}^- 
=\overline{\beta^+_{\mu_n}} \,.
\end{equation}
For any  $s>11/2$, let $f \in W^s(SM)$ be a real-valued function  supported on irreducible components of the principal series.  By definition and by formula~(\ref{eq:conjugation}), the cocycle  $\beta_f:SM\times \R \to \C$ is also real-valued, and from (\ref{eq:expansion1}) we have
$$
\beta_f(x,T) = \re [ \sum_{n\in \N} D^+_{\mu_n} (f) \beta^+_{\mu_n} (x,T)]\,, \quad \text{ for all }(x,T)\in SM\times \R\,.
$$
Let $\T^\infty:=(\R/ 2\pi\Z)^\infty$ be the infinite-dimensional torus endowed with the product topology.
For any real-valued function $f\in W^s(SM)$ supported on irreducible unitary components of the principal series and for all $ \theta \in \T^\infty$, let
$$
\beta(f,\theta,x,T) :=\re[  \sum_{n\in \N}  D^+_{\mu_n}(f) e^{i\theta_n} \beta^+_{\mu_n}(x,T) ]\,,
\quad \text{ for all } (x,T) \in SM\times \R\,.
$$
For $\theta \in\T^\infty$ and $T>0$, let  $P_{pr} (f,\theta,T)$ be the probability distribution of the random variable given by the formula
\begin{equation}
\label{eq:limit_torus}
 \frac{\beta(f,\theta,\cdot,T)}{ \Vert \beta(f,\theta,\cdot,T)\Vert}\,, \quad \text{for all } x\in SM\,.
\end{equation}
Since the random variables $\beta(f,\theta,\cdot,T)$ on $SM$ are non-constant and
bounded, the probability distributions $P (f,\theta,T)$ are non-atomic compactly supported
measures on the real line  (uniformly with respect to $T\in [0,1]$).

Our main result on the asymptotics of distributions of normalized ergodic integrals for real-valued functions supported on the principal series is the following

\begin{theorem}
\label{thm:limit_torus}
For any  $s>11/2$ there exists a constant $C_s>0$ such that the following holds
for any real-valued function $f\in W^s(SM)$ supported on the irreducible components of the principal series such that $\{D^+_{\mu_n}(f)\} \not =0$ in $\ell^1(\N,\C)$.
\begin{enumerate}
\item
For all $T\in [0,1]$ and all $t>0$, we have
 \begin{equation}
 \left\vert   \frac{\Vert \int_0^{Te^t} f\circ h^U_\tau (x)d\tau\Vert}{ e^{ \frac{t}{2}}
\Vert \beta(f,\frac{\upsilon t }{2},\cdot,T)\Vert} \,-\, 1  \right\vert
  \leq C_s \Vert f\Vert_s e^{-\frac{t}{2}}\,.
 \end{equation}
\item
For all $T\in [0,1]$ and all $t>0$, we have
$$
d_{LP}\left( \mathcal E_t(f,T), P_{pr} (f, \frac{\upsilon t}{2},T) \right)   \leq  C_s \Vert f\Vert_s
e^{-\frac{t}{2}}\,.
$$
\end{enumerate}
\end{theorem}

The above theorem implies that for real-valued functions supported on the principal series
limit distributions exist along sequence of time such that the orbit of the toral translation
of frequency $\upsilon/2 \in \R^\infty$ on the infinite torus $\T^\infty$ converges. We 
conjecture that the limit does not exist otherwise. Below from Theorem~\ref{thm:limit_torus} 
we derive some restrictions on limit distributions. 

\begin{definition}
Let $H \subset L^2(SM)$ be a $PSL(2,\R)$-invariant subspace which is a direct
sum of finitely many irreducible components of the principal series, that is,
$$
H = \sum_{i=1}^n H_i\,.
$$
The subspace $H$ is called  \emph{Casimir simple} if all the corresponding Casimir
parameters $\{\nu_1, \dots, \nu_n\}$ are distinct. The subspace $H$ is called \emph{Casimir
irrational} if the  Casimir parameters $\{\nu_1, \dots, \nu_n\}$ are rationally independent.
\end{definition}

We derive the following conditional uniqueness result for the principal series.

\begin{corollary}
\label{cor:limitdistprincipal1}
Let $H\subset L^2(SM)$ be any Casimir simple
 $PSL(2,\R)$-invariant
subspace. If the limit distribution of the family of random variables
$$
 \frac{ \int_0^T   f\circ h^U_t dt}{ \Vert \int_0^T   f\circ h^U_t dt \Vert }
$$
exists for any given $f\in C^\infty(H)$  which is not a coboundary, then the limit
distribution is unique in the sense that it does not depend on the function.
\end{corollary}

Finally, we derive restrictions on the joint probability distribution of the  cocycle functions
in case limit distributions exist for all function supported on a Casimir irrational subspace.

For any irreducible representation $H_\mu\subset L^2(SM)$ of the principal series, we have
constructed  H\"older cocycle functions $\beta:= \beta^+_\mu$ and $\bar \beta=\beta^-_\mu :SM\to \C$. For any $PSL(2,\R)$-invariant subspace $H \subset L^2(SM)$ supported on finitely many irreducible components of the principal series, let
$$
\beta_H := (\beta_1, \dots, \beta_n): SM \to \C^n
$$
be the corresponding vector-valued cocycle function.

Given any function $\beta:SM\to \C$, let $T_\beta: \C \to \C$ be the affine transformation
defined as follows. Let $R_\beta$ be the rotation by the angle $\theta_\beta\in [0, 2\pi)$
such that
$$
 e^{ 2i\theta_\beta} \int_{SM}  \beta^2 d\text{vol}  \in \R^+\cup\{0\}\,.
$$
For any pair $(A,B)$ of positive real numbers, let $T_{A,B}:\C \to \C$ the affine map
$$
T_{A,B} (x,y) := (x/A,y/B) \,, \quad \text{ \rm for all } (x,y)\in  \R^2 \equiv \C\,.
$$
Let $(A_\beta,B_\beta)$ be the positive real numbers given by the formulas:
$$
\begin{cases}
A^2_\beta= (\Vert \beta\Vert^2 + \vert \int_{SM} \beta^2 \omega \vert)/2\,, \\
B^2_\beta= (\Vert \beta\Vert^2 - \vert \int_{SM} \beta^2\omega \vert)/2\,.
\end{cases}
$$
It is proved in Section~\ref{innerproducts} that $\int_{SM} \beta^2 \omega \not=0$,
hence $A_\beta^2 > B_\beta^2$. Let then $T_\beta$ be given by the formula
$$
T_\beta :=  T_{A_\beta,B_\beta} \circ R_\beta \,.
$$
Given any function $\beta:=(\beta_1, \dots, \beta_n) \to \C^n$, let $T_\beta$
be the product affine map
$$
T_\beta = T_{\beta_1} \times \dots T_{\beta_n} :\C^n \to \C^n \,.
$$

\begin{corollary}
\label{cor:limitdistprincipal2}
 Let $H\subset L^2(SM)$ be any Casimir irrational $PSL(2,\R)$-invariant
subspace. The limit distribution of the family of random variables
$$
 \frac{ \int_0^T   f\circ h^U_t dt}{ \Vert \int_0^T   f\circ h^U_t dt \Vert }
$$
exists for all $f\in C^\infty(H)$ which is not a coboundary if and only if the function
$T_{\beta_H} \circ \beta_H:SM \to \C^n$ has a rotationally invariant probability distribution.
\end{corollary}

 \subsection{Duality theorems.}
 \label{Duality}
The formalism of finitely-aditive measures allows us to establish a {\it duality} between the spaces of
distributions invariant under the stable and the unstable horocycle flows, respectively; more precisely, between
the subspaces of invariant distributions corresponding to the positive eigenvalues of the Casimir operator.

Finitely additive $V$-invariant $1$-dimensional H\"older measures on rectifiable arcs induce by integration
currents of dimension $2$ (and degree $1$). In fact, let $\hat \beta\in {\hat {\mathfrak  B}}_V(SM)$.
Given a smooth $2$-form $\eta$, using the H{\"o}lder property and the (finite) additivity of $\hat\beta$, one can
define the integral
$$
\int\limits_{SM} \hat\beta \otimes\eta
$$
as the limit of Riemann sums. The correspondence
$$
\eta\to\int\limits_{SM} \hat\beta \otimes\eta
$$
now yields a current on $SM$ of dimension $2$ (and degree $1$), which, slightly abusing notation, we denote by the same symbol $\hat\beta$. The current $\hat\beta$ defined above in fact extends
to continuous forms and, by the Sobolev embedding theorem, to  forms in Sobolev spaces.

Our next aim is to describe the currents $\hat\beta^\pm_\mu$ in terms of
distributions invariant under the unstable horocycle flow $\{h^V_t\}$.

Given a  Casimir parameter $\mu >0$, consider the finitely-additive measure
\begin{equation}
\label{dpmu}
\hat D^\pm_\mu ={\hat X}\otimes \hat \beta^\pm_\mu \otimes \hat V.
\end{equation}

Since, for any $f\in C^{\infty}(SM)$, the integral of $f$ with respect to  the measure
$\hat D^\pm_\mu$ can be defined as the limit of Riemann sums, the measure $\hat D^\pm_\mu$
yields a distribution (in the sense of S.L. Sobolev and L. Schwartz)  on $C^{\infty}(SM)$; slightly
abusing notation, we denote the distribution by the same symbol $\hat D^\pm_\mu$.

\begin{theorem}\label{product-duality}
For every Casimir parameter $\mu \in \R^+$, the distributions $\hat D^\pm_\mu$ given by
 (\ref{dpmu})  are $V$-invariant. For $\mu\neq 1/4$, they satisfy the identities
 $$
 \mathcal L_X  \hat D^\pm_{\mu}= \frac{1\pm \nu}{2} \hat D^\pm_{\mu}\,,
 $$
 while for $\mu=1/4$ ($\nu=0$), they satisfy the identity
 $$
  \mathcal L_X  \begin{pmatrix}  \hat  D^+_{1/4} \\  \hat  D^-_{1/4}  \end{pmatrix} =
 \frac{1}{2} \begin{pmatrix} 1 & 1 \\ 0 & 1 \end{pmatrix}
  \begin{pmatrix}   \hat  D^+_{1/4} \\  \hat  D^-_{1/4}  \end{pmatrix} \,.
 $$
\end{theorem}

Theorem \ref{product-duality} can be equivalently reformulated as follows.

Given a distribution $D$ acting on $C^{\infty}(SM)$, let $D\wedge \hat U$ denote the current of
degree $1$ (and dimension $2$) defined as the exterior product of the the distribution $D$,
identified to a current of degree $0$ (and dimension $3$) via the normalized volume form $\omega$, times the smooth $1$-form $\hat U$ on $SM$; that is, the current given by the following formula:
for any smooth $2$-form $\eta$ on $SM$,
$$
(D\wedge \hat U)(\eta):=D\left(\frac{\hat U\wedge \eta}{\omega}\right).
$$

\begin{theorem}
\label{thm:duality} For every Casimir parameter $\mu \in \R^+$,
there exist $V$-invariant distributions $\hat D^\pm_{\mu} \in
\mathcal D'(H_\mu)$  such that
$$
\hat \beta_\mu^\pm = \hat D^\pm_{\mu} \wedge \hat U\,.
$$
For all $\mu\neq 1/4$, the distributions $\hat D^\pm_{\mu}$ are
eigenvectors of the geodesic flow, that is, they satisfy the identitities
$$\mathcal L_X  \hat D^\pm_{\mu}= \frac{1\pm \nu}{2} \hat D^\pm_{\mu}\,,
$$
while for $\mu=1/4$ ($\nu=0$) they are generalized eigenvectors, that is,
 \begin{equation}
 \label{eq:Jordanbbis}
  \mathcal L_X  \begin{pmatrix}  \hat  D^+_{1/4} \\  \hat  D^-_{1/4}  \end{pmatrix} =
 \frac{1}{2} \begin{pmatrix} 1 & 1 \\ 0 & 1 \end{pmatrix}
  \begin{pmatrix}   \hat  D^+_{1/4} \\  \hat  D^-_{1/4}  \end{pmatrix} \,.
\end{equation}
\end{theorem}

The duality theorem  (Theorem \ref{thm:duality}) leads to the classification theorem
stated below. Let $\hat {\mathfrak  B}_V(SM)$ be the space of  all finitely additive
$1$-dimensional H\"older measures introduced in Definition~\ref{def:hatbetaprops}. For any $s>0$, let $\Omega_2^{-s}(SM)$ be the Sobolev space of currents of dimension $2$ (and degree $1$) defined as the dual space of the Sobolev space $\Omega_2^s(SM)$ of $2$-forms on $SM$.  By the Sobolev embedding theorem, the space $\hat {\mathfrak  B}_V(SM)$ embeds as closed subspace, denoted as
$\hat {\mathfrak  B}^{-s}_V(SM)$, into $\Omega_2^{-s}(SM)$.

\begin{theorem}
\label{thm:classification}
For all $s> 3/2$, the Hilbert space $\hat {\mathfrak  B}^{-s}_V(SM)$ is spanned by the system of finitely-additive measures $\{\hat \beta_{\mu}^{\pm} \vert \mu \in \Spec(\square) \cap \R^+\}$.
\end{theorem}

The duality theorem (Theorem \ref{thm:duality}) also leads to a direct bijective correspondence between
the lift of the finitely-additive measures $\hat\beta^\pm_\mu$ to $PSL(2,\R)$ (denoted below by
the same symbol) and the $\Gamma$-invariant conformal distributions on the boundary of the
Poincar\'e disk studied by S.~Cosentino in \cite{Co}.

\begin{theorem}
\label{thm:gammaconformal}
 For any Casimir parameter $\mu \in \R^+\setminus \{1/4\}$, there exist on the boundary of the Poincar\'e disk $\Gamma$-invariant conformal distributions $\phi^\pm_\mu$  of exponents $ (1\mp \nu)/2$ such that the following identities hold on $PSL(2,\R)$:
$$
\hat \beta^\pm_\mu \otimes dt   = \phi^\pm_\mu  \otimes
e^{-(\frac{1\mp \nu}{2}) t} dt  \,.
$$
For $\mu=1/4$ ($\nu=0$), on the boundary of the Poincar\'e disk  there exist  a $\Gamma$-invariant conformal distribution $\phi_{1/4}$  of exponent $1/2$ and a distribution $\phi'_{1/4}$ of order $1/2^+$ (in the H\"older sense) such that
\begin{equation*}
\begin{aligned}
\hat \beta^+_{1/4} \otimes dt   &= \phi'_{1/4}  \otimes
e^{-\frac{t}{2}} dt +  \phi_{1/4}  \otimes
  \frac{t}{2} e^{-\frac{t}{2} } dt \,, \\
  \hat \beta^-_{1/4} \otimes dt   &= \phi_{1/4}  \otimes
e^{-\frac{t}{2} } dt   \,.
\end{aligned}
\end{equation*}
\end{theorem}
\begin{remark} For $\mu =1/4$, the space of all
$\Gamma$-invariant conformal distributions  of exponent
$1/2$ is $1$-dimensional in each irreducible component.
\end{remark}

\subsection{Organization of the paper.}

The paper is organized as follows. In Section~\ref{sec:BasicCurrents} we construct finitely
additive measures of rectifiable arcs and prove our main results about them (in particular
Theorem~\ref{thm:hatbetaprops}, Corollary~\ref{cor:wuvanishing} and Corollary~\ref{cor:basiccurrents}
up to a technical estimate (Lemma~\ref{lemma:remainderextension}) which will be proved in
\S\S~\ref{keyestimates}. In Section~\ref{sec:Cocycles} we prove our results on additive cocycles
for the horocycle flow (Theorem~\ref{thm:cocycleproperties})  and  the approximation theorem for ergodic integrals (Theorem~\ref{thm:approximation}). From the approximation theorem, we
then derive our results on limit distributions (Theorem~\ref{thm:limitdistcompl}, Theorem~\ref{thm:limit_torus}, Corollary~\ref{cor:limitdistprincipal1} and Corollary~\ref{cor:limitdistprincipal2}).
Section~\ref{sec:Duality} is devoted to the proof of the duality theorem (Theorem~\ref{thm:duality}),
of the classification theorem (Theorem~\ref{thm:classification}) and to the relations with
$\Gamma$-invariant conformal distributions (Theorem~\ref{thm:gammaconformal}).
 In Section~\ref{sec:appendix} we collect several technical auxiliary results. In \S\S~\ref{keyestimates} we prove the above-mentioned estimate we need in the construction of finitely additive measures
 and additive cocycles (Lemma~\ref{lemma:remainderextension}). In \S\S~\ref{innerproducts} and
 \S\S~\ref{rotsym} we prove the technical lemmas needed in the proof of our conditional theorems on
 existence of limit distributions for functions supported on irreducible components of the principal
 series.

\subsection{ Acknowledgements}.
We are deeply grateful to Yakov Sinai for suggesting the problem to us and for his friendly encouragement.
We are deeply grateful to Artur Avila for pointing out the orthogonality
property of additive cocycles stated in Theorem~\ref{thm:cocycleproperties}.

A.I. B. is an Alfred P. Sloan Research Fellow.
During work on this project, he
was supported in part by Grant MK-4893.2010.1 of the President of the Russian Federation,
by the Programme on Mathematical Control Theory of the Presidium of the Russian Academy of Sciences,
by the Programme 2.1.1/5328 of the Russian Ministry of Education and Research,
by the Edgar Odell Lovett Fund at Rice University, by
the National Science Foundation under grant DMS~0604386, and by
the RFBR-CNRS grant 10-01-93115.

G.F. was supported by the National Science Foundation grant DMS~0800673.

\section{Basic Currents and Finitely-Additive Measures on Rectifiable Arcs.}
\label{sec:BasicCurrents}

In this section we prove Theorem~\ref{thm:hatbetaprops}  up to a
technical estimate which will be proved in the  \S\S~\ref{keyestimates}.
We then derive Corollary~\ref{cor:wuvanishing} and Corollary~\ref{cor:basiccurrents}.

 \subsection{Basic currents.} By definition,  the volume form $\omega$ on $SM$ can be written as
 $$
 \omega = \hat X \wedge \hat U \wedge \hat V\,.
 $$
The contractions $\eta_X:=\imath_X \omega$, $\eta_U:=\imath_U
\omega$ and $\eta_V:=\imath_V \omega$ are closed $2$-forms which can
be written as follows:
$$
\eta_X = \hat U \wedge \hat V\,, \quad \eta_U= - \hat X \wedge \hat
V\, , \quad \eta_V = \hat X \wedge \hat U\,.
$$

We recall that a distribution $D \in \mathcal D'(SM)$ (in the sense
of S.L. Sobolev and L. Schwartz) is called $U$-{\it invariant }(or
invariant under the stable horocycle flow $\{h^U_t\}$) iff $\mathcal
L_U D=0$ in $\mathcal D'(SM)$. A current $C$ of degree $2$ (and
dimension $1$) is called {\it basic }for the orbit foliation of the
stable horocycle  flow if and only if
\begin{equation}
\label{eq:basic}
 \mathcal L_U C  = \imath_U C =0 \,\, \text{ in }\,\, \mathcal D'(SM)\,.
\end{equation}
(The operators $ \mathcal L_U$ and $\imath_U$ are respectively the
Lie derivative and the contraction with respect to the horocycle
generator $U$ acting on currents according to the standard
definition).

Let $\mathcal I_U (SM)$ denote the space of all $U$-invariant
distributions and $\mathcal B_U(SM)$ denote the space  of all \emph{basic
currents }of degree $2$ (and dimension $1$) for the orbit foliation
of the stable horocycle flow.

 For every $s\geq 0$, let $W^s(SM)$ be the standard Sobolev space on the compact manifolds
 $SM$ and let $\Omega_1^s(SM)$ be the Sobolev space of all $1$-forms on $SM$ defined as follows:
 $$
  \lambda: = \lambda_X \hat X + \lambda_U \hat U  + \lambda_V \hat V \in
 \Omega_1^s(SM) \Leftrightarrow (\lambda_X, \lambda_U, \lambda_V) \in [W^s(SM)]^3  \,.
 $$
Let $W^{-s}(SM)$ and $\Omega_1^{-s}(SM)$ denote the Sobolev spaces
dual of the (Hilbert) spaces $W^s(SM)$ and $\Omega_1^s(SM)$
respectively. The space $W^{-s}(SM)$ which can be viewed either as
currents of degree $3$ and dimension $0$ (linear functionals on
functions) or as currents of degree $0$ and dimension $3$ (linear
functional on $3$-forms). A standard $SL(2,\R)$-invariant
identification between functions and $3$-forms is in fact given by
the volume form on $SM$.  The space  $\Omega_1^{-s}(SM)$ is a  space
of  currents of degree $2$ and dimension $1$ (linear functionals on
$1$-forms). Let  $\mathcal B_U^{-s}(SM) \subset \Omega_1^{-s}(SM)$
denote the subspace of basic currents for the orbit foliation of the
stable horocycle flow, that is, of currents satisfying the
identities (\ref{eq:basic}). It is a standard fact, easy to prove,
that  the space $\mathcal B_U^{-s}(SM)$ of basic currents is
isomorphic to the space $\mathcal I_U^{-s} (SM)$ of $U$-invariant
distributions:

\begin{lemma}
\label{lemma:Isom} For any $s\in {\mathbb R}$ the correspondence
$D\to D \eta_U$ defines an isomorphism from the space $\mathcal
I_U^{-s} (SM)$ of invariant distributions for the stable horocycle
flow onto the space $\mathcal B_U^{-s}(SM)$ of basic currents for
its orbit foliation.
\end{lemma}
\begin{proof} Let $D$ be any $U$-invariant distribution.
It follows that
$C:=D \eta_U$ is closed. By definition $\imath_U C=0$. It follows
that
$$
 \mathcal L_U C = \imath_U dC + d  \imath_U C =0\,,
$$
hence $C$ is a basic current for the stable horocycle foliation. Conversely, let $C$ be any basic
current for the stable horocycle foliation and let $D: = C\wedge \hat U$.
We claim that $D$ is $U$-invariant. Since $\mathcal L_U C= \imath_U C=0$,
a computation yields
$$
 \mathcal L_U D = \mathcal L_U C \wedge \hat U +  C \wedge \mathcal L_U \hat U
 =  C \wedge \imath_U (d\hat U)= \imath_UC \wedge d \hat U=0\,.
$$
Finally, since $\imath_U C=0$, it follows that
$$
D\eta_U \equiv \imath_U D = \imath_U (C\wedge \hat U) = C  \,,
$$
hence the map $D \to D \eta_U$ is a bijection of the space of all invariant distributions
 onto the space of all basic currents  with inverse given by the map $C \to C\wedge \hat U$.
 It follows from the definition of the Sobolev spaces of currents that the above maps  are
 isomorphisms between the dual Sobolev spaces $\mathcal I_U^{-s} (SM)$ and
 $\mathcal B_U^{-s}(SM)$.
 \end{proof}

\subsection {Geodesic scaling of basic currents.}

Let $H_\mu \subset L^2(SM)$ be any non-trivial irreducible component  with Casimir
parameter $\mu \in \R\setminus\{0\}$. Let $W^{-s}(H_\mu)$ and  $\Omega_1^{-s}(H_\mu)$ denote
the associated Sobolev spaces of distributions and, respectively, currents of dimension
$1$ (and degree $2$). The subspaces $W^{-s}(H_\mu)$ and  $\Omega_1^{-s}(H_\mu)$ are  $SL(2,\R)$-invariant irreducible components of the decomposition of the dual Sobolev spaces $W^{-s}(SM)$ and 
$\Omega_1^{-s}(SM)$ respectively. 
Let $\mathcal B^{-s}_U(H_\mu)$ denote the associated $SL(2,\R)$-invariant  irreducible component of the space $\mathcal B^{-s}_U(SM)$ of basic currents for the stable horocycle foliation, 
that is, for all $\mu\in \R\setminus\{0\}$,
$$
\mathcal B^{-s}(H_\mu) := \mathcal B^{-s}(SM) \cap  \Omega_1^{-s}(H_\mu)\,.
$$
The following result describes the  (infinitesimal) action of the geodesic flow on the space
$\mathcal B_U^{-s}(H_\mu)$ of basic currents for all Casimir parameters $\mu \in \R\setminus\{0\}$.

\begin{lemma}
\label{lemma:EVcurrents} For any $s>1$ and $\mu \in \R^+ \setminus
\{1/4\}$, the space $\mathcal B^{-s}_U(H_\mu)$ has complex dimension
$2$ and has a basis $\{B^+_\mu, B^-_\mu\}$ of eigenvectors for the
action of the geodesic flow. In fact, the following formulas hold:
\begin{equation}
\label{eq:BEV} \mathcal L_X  B^{\pm}_\mu =   \frac{1\mp \nu}{2}
B^{\pm}_\mu \,.
\end{equation}
For $\mu=1/4$ ($\nu=0$), the space $\mathcal B^{-s}_U(H_\mu)$ has complex
dimension $2$ and has a basis $\{B^+_\mu, B^-_\mu\}$ of generalized
eigenvectors for the action of  the geodesic flow. The following
formula holds:
\begin{equation}
\label{eq:BGEV} \mathcal L_X \begin{pmatrix} B^+_{1/4} \\ B^-_{1/4}
\end{pmatrix} = \frac{1}{2}  \begin{pmatrix}  1 & 0 \\ -1 & 1
\end{pmatrix}
 \begin{pmatrix} B^+_{1/4} \\ B^-_{1/4} \end{pmatrix}\,.
\end{equation}
For $\mu =-n^2+n<0$ ($\nu=2n-1$) and $s>n$, the space $\mathcal B^{-s}_U(H_\mu)$ has complex
dimension $1$ and has a basis $\{B_\mu\}$, containing a single eigenvector for the action of  the geodesic flow. In fact, the following formula holds:
\begin{equation}
\label{eq:BEVdiscrete} \mathcal L_X  B_\mu =   \frac{1- \nu}{2}
B_\mu =  (1-n) B_\mu\,.
\end{equation}
\end{lemma}
\begin{proof} It follows by Lemma \ref{lemma:Isom} that the space $\mathcal
B^{-s}_U(H_\mu)$ is isomorphic to the space $\mathcal I_U^{-s}
(H_\mu) := \mathcal I_U(SM)\cap W^{-s}(H_\mu)$ of invariant
distributions. By \cite{FlaFo}, Theorem 3.2,  for any $\mu \in \R^+
\setminus \{1/4\}$, the space $\mathcal I_U^{-s} (H_\mu)$ has
complex dimension $2$ and has a basis $\{D^+_\mu, D^-_\mu\}$ of
eigenvectors of the geodesic flow, in the sense that the following
formulas hold:
$$
\mathcal L_X  D^{\pm}_\mu =   -\frac{1\pm \nu}{2}  D^{\pm}_\mu \,.
$$
Let $B^\pm_\mu := D^\pm_\mu \eta_U$. Since $U$ is the generator of
the stable horocycle flow, we have the following equality  of
currents:
\begin{equation}
\label{eq:etaUder} \mathcal L_X \eta_U =  \eta_U \,.
\end{equation}
The statement for the case $\mu \neq 1/4$ then follows since
$$
\mathcal L_X B^\pm_\mu = (\mathcal L_X  D^{\pm}_\mu +
D^\pm_\mu)\eta_U =  \frac{1\mp \nu}{2} B^\pm_\mu\,.
$$
In the case $\mu =1/4$, by \cite{FlaFo}, Lemma 3.5, the space
$\mathcal I_U^{-s} (H_\mu)$ has complex dimension $2$ and has a
basis $\{D^+_\mu, D^-_\mu\}$ of generalized eigenvectors of the
geodesic flow, in the sense that the following formulas hold:
$$
\mathcal L_X \begin{pmatrix} D^+_\mu \\ D^-_\mu \end{pmatrix} =
-\frac{1}{2}  \begin{pmatrix}  1 & 0 \\ 1 & 1 \end{pmatrix}
 \begin{pmatrix} D^+_\mu \\ D^-_\mu \end{pmatrix}\,.
$$
Let $B^\pm_\mu := D^\pm_\mu \eta_U$. Formula (\ref{eq:BGEV}) then
follows by Leibniz rule from the above formula and formula
(\ref{eq:etaUder}). In fact,
$$
\mathcal L_X \begin{pmatrix} B^+_\mu \\ B^-_\mu \end{pmatrix} =
\left [ -\frac{1}{2}  \begin{pmatrix}  1 & 0 \\ 1 & 1 \end{pmatrix}
+  \begin{pmatrix}  1 & 0 \\ 0 & 1 \end{pmatrix} \right]
\begin{pmatrix} B^+_\mu \\ B^-_\mu \end{pmatrix}
$$
Similarly, formula (\ref{eq:BEVdiscrete}) follows from \cite{FlaFo}, Lemma 3.5.
In fact, for any Casimir parameter $\mu=-n^2+n<0$, the space $\mathcal I_U^{-s} (H_\mu)$
has complex dimension $1$ and has a basis $\{D_\mu\}$ containing a single eigenvector
of the geodesic flow such that
$$
\mathcal L_X  D_\mu =   -\frac{1+ \nu}{2}  D_\mu= -n D_\mu \,.
$$
>From the above formula it follows that
$$
\mathcal L_X B_\mu = (\mathcal L_X  D_\mu +
D_\mu)\eta_U =  \frac{1- \nu}{2} B_\mu = (1-n) B_\mu\,.
$$
The Lemma is proved.
\end{proof}

\subsection{Orthogonal projections on basic currents.}

 For any $s>1/2$ we have the orthogonal direct sum decomposition
$$
\Omega_1^{-s}(SM)={\mathcal  B}^{-s}_U(SM)\oplus^{\perp} {\mathcal
B}^{-s}_U(SM)^{\perp}
$$
Let $\mathcal B^{-s}: \Omega_1^{-s}(SM) \to \mathcal  B^{-s}_U(SM)$
denote the orthogonal projection onto the subspace of basic currents
and $\mathcal R^{-s}: \Omega_1^{-s}(SM) \to \mathcal  B^{-s}_U
(SM)^\perp$ denote the orthogonal projection onto its orthogonal
complement .

Let $\Pi^{-s}_\mu: \Omega_1^{-s}(SM) \to
\Omega_1^{-s}(H_\mu)$ be the orthogonal projection. We remark that the
projections $\Pi^{-s}_\mu$ commute with the action of $SL(2,\R)$,
hence in particular with the action of the geodesic flow, on the
Sobolev space $\Omega_1^{-s}(SM)$.

Let $\mathcal B_\mu^{-s} :=  \Pi^{-s}_\mu \circ \mathcal B^{-s}=
\mathcal B^{-s}\circ  \Pi^{-s}_\mu$ be the orthogonal projection
onto the subspace $\mathcal B^{-s}_U(H_\mu) \subset \Omega_1^{-s}
(H_\mu)$.  Let $\mathcal R_\mu^{-s} :=  \Pi^{-s}_\mu \circ \mathcal
R^{-s}=  \mathcal R^{-s}\circ \Pi^{-s}_\mu $ be the complementary projection on the
space $\mathcal B^{-s}_U(SM)^\perp \cap  \Omega_1^{-s}(H_\mu)$.  We remark that the projections
$\mathcal B_\mu^{-s}$ and $\mathcal R_\mu^{-s}$ do not necessarily commute with the action of the geodesic flow. However, the range  of the projection $\mathcal B_\mu^{-s}$, which is the space
$\mathcal B^{-s}_U(H_\mu)$ of basic currents, is invariant under the action of the geodesic flow.

 By the Sobolev embedding theorem,  any rectifiable arc $\gamma$ can be viewed as a current of dimension $1$ (and degree $2$) in $\Omega_1^{-s}(SM)$ for any $s> 3/2$. For all non-trivial irreducible unitary representations of Casimir parameter $\mu \in \R$, let $\mathcal B^{-s}_\mu(\gamma)
\in \Omega_1^{-s}(H_\mu)$ denote the projection $\Pi^{-s}_\mu\circ \mathcal B^{-s}(\gamma)$ of the current $\mathcal B^{-s}_\mu(\gamma)$ onto the irreducible subspace  $\Omega_1^{-s}(H_\mu) \subset \Omega_1^{-s}(SM)$. We then write
\begin{equation}
\label{eq:basiccurrentproj}
\mathcal B_\mu^{-s}(\gamma):= \begin{cases}   \hat \alpha^+_{\mu,-s} (\gamma) B_\mu^+ +
  \hat \alpha^-_{\mu,-s} (\gamma) B_\mu^- \,, &\quad \text{ for } \mu \in \Spec(\square) \cap \R^+\,;\cr
 \hat \alpha_{\mu,-s} (\gamma) B_\mu\,,  &\quad \text{ for } \mu \in \Spec(\square) \cap \R^-   \,.
 \end{cases}
\end{equation}
In other words,  the complex numbers $\hat \alpha^\pm_{\mu,-s} (\gamma)$, $\hat\alpha_\mu(\gamma)$  are the components of the current $\gamma$ in the direction of the basic currents $B_\mu^\pm$,
$B_\mu$ (that is, by definition, the coefficients of the currents $B_\mu^\pm$, $B_\mu$ in the orthogonal projection of the current $\gamma$ onto the closed subspace of all basic currents).

We recall that the subspace $\mathcal B_U^{-s}(H_\mu) \subset \Omega_1^{-s}(H_\mu)$ is
trivial for all Casimir parameters $\mu=-n^2+n$ (discrete series) whenever $s\leq n\in \Z^+$. In this
case the component $\hat \alpha_{\mu,-s} $ is defined as zero.

\begin{lemma}
\label{lemma:distorsion}
For every $\mu_0>1/4$ and for every $s>1/2$ the system $\{B^+_\mu, B^-_\mu\}$ has uniformly bounded distorsion in $\Omega_1^{-s}(SM)$ for all Casimir parameters $\mu\geq \mu_0$, that is, there exists a constant $C_s(\mu_0)>0$ such that, for all $\mu\geq \mu_0$,
$$
 \sup_{ (\alpha^+, \alpha^-)\in \R^2\setminus\{0\} } \, \frac{  \Vert \alpha^+  B^+_\mu \Vert_{-s} + \Vert \alpha^- B^-_\mu \Vert_{-s} } {  \Vert \alpha^+ B^+_\mu +  \alpha^- B^-_\mu\Vert_{-s}}  \leq   C_s(\mu_0)\,.
$$
\end{lemma}
\begin{proof}
As observed in the proof of Lemma 5.1 in \cite{FlaFo}, if the Casimir parameter $\mu\geq \mu_0 >1/4$, for any $s>1/2$, the {\it distorsion} in $W^{-s}(SM)$ of the system of distributions $\{D^+_\mu, D^-_\mu\}$ stays uniformly bounded above (in other terms, the angle in $W^{-s}(SM)$ between $D^+_\mu$ and $D^-_\mu$ stays uniformly bounded below), and in fact this bound is also uniform with respect to
$s>1$. By Lemma \ref{lemma:Isom}, the map $D \to D \eta_U$ defines an isomorphism from the space
${\mathcal I}^{-s}_U(SM)$ of invariant distribution onto the space ${\mathcal B}^{-s}_U(SM)$ of basic currents. It follows that for all Casimir parameters $\mu>0$ the distorsion  in $\Omega_1^{-s}(SM)$ of the system of basic current $\{B^+_\mu, B^-_\mu\}$ is equal to the distorsion in $W^{-s}(SM)$ of the system  of invariant distributions $\{D^+_\mu, D^-_\mu\}$, in particular the distorsion of the system  $\{B^+_\mu, B^-_\mu\}$ in $\Omega_1^{-s}(SM)$  is uniformly bounded above for all $\mu\geq \mu_0>1/4$.
\end{proof}
\subsection{The construction of the
finitely-additive measures.}

The core of our argument is
the following construction of finitely-additive measures on rectifiable arcs.
\begin{theorem}
\label{thm:hatbeta} For any rectifiable arc $\gamma\subset SM$  the following holds. For any Casimir parameter $\mu\in \R^+\setminus\{1/4\}$ the following limits exist and do not depend  on $s>9/2$:
\begin{equation}
\label{eq:hatbeta} \hat \beta^\pm_\mu (\gamma) := \lim_{t\to
+\infty}
 \frac{ \hat\alpha^\pm_{\mu,-s} (g_t^\ast \gamma)} {\exp (\frac{1\mp \nu}{2} t)}\,;
\end{equation}
For $\mu =1/4$ the  limits below exist  and do not depend on $s>9/2$:
\begin{equation}
\label{eq:hatbetabis}
\begin{aligned}
 \hat \beta^+_{1/4} (\gamma) &:= \lim_{t\to +\infty}
  \frac{( \hat\alpha^+_{1/4,-s} + \frac{t}{2} \hat\alpha^-_{1/4,-s})
 (g_t^\ast \gamma) }  { \exp (\frac{t}{2} )} \,, \\
   \hat \beta^-_{1/4} (\gamma) &:=  \lim_{t\to +\infty}
 \frac{ \hat\alpha^-_{1/4,-s} (g_t^\ast \gamma)} {  \exp (\frac{t}{2} )}\,.
 \end{aligned}
 \end{equation}
The convergence in the limits (\ref{eq:hatbeta}), (\ref{eq:hatbetabis}) is exponential in the following
precise sense. For all $t>0$, let us introduce the rescaled weak unstable length
\begin{equation}
\label{eq:weakunstablelength}
\vert \gamma \vert_{XV,t} := \int_\gamma \vert \hat X\vert + e^{-t}  \int_\gamma \vert \hat V\vert \,.
\end{equation}
There exists a constant $C_s>0$ such that, for any $\mu \neq 1/4$,
\begin{equation}
\label{eq:hatbetaspeed}
\left|\hat \beta^\pm_\mu (\gamma)-
 \frac{ \hat\alpha^\pm_{\mu,-s} (g_t^\ast \gamma)} {\exp (\frac{1\mp \nu}{2}  t)}\right| \leq
 \frac{C_s}{ \Vert B^\pm_\mu\Vert_{-s}} \frac{1 +\vert \gamma\vert_{XV,t}}{\exp (\frac{1\mp \nu}{2}t)}\, ,
 \end{equation}
while for $\mu=1/4$,
\begin{equation}
\label{eq:hatbetaspeedbis}
\begin{aligned}
&\left| \hat \beta^+_{1/4} (\gamma)  -
  \frac{ (\hat\alpha^+_{1/4,-s}+ \frac{t}{2} \hat\alpha^-_{1/4,-s})
 (g_t^\ast \gamma) }  { \exp (\frac{t}{2} )}  \right|  \leq \frac{C_s(1+t)}{ \Vert B^\pm_\mu\Vert_{-s}}
 \frac{ (1 +\vert \gamma\vert_{XV,t})}{ \exp(\frac{t}{2})} \,, \\
&\left|  \hat \beta^-_{1/4} (\gamma) -
 \frac{ \hat\alpha^-_{1/4,-s} (g_t^\ast \gamma)} {\exp (\frac{t}{2})} \right|  \leq
 \frac{C_s(1+t)}{ \Vert B^\pm_\mu\Vert_{-s}}
 \frac{ (1 +\vert \gamma\vert_{XV,t})}{ \exp(\frac{t}{2})}  \,.
  \end{aligned}
\end{equation}
For all Casimir parameters $\mu\in \R^+$, the following bound holds:
\begin{equation}
\label{eq:hatbetabound} \vert \hat \beta^\pm_\mu (\gamma) \vert \leq
\frac{C_s}{ \Vert B^\pm_\mu\Vert_{-s}} (1 +\int_\gamma \vert \hat X\vert +
 \int_\gamma \vert \hat U\vert + \int_\gamma \vert \hat V\vert)\,.
\end{equation}
\end{theorem}

\begin{proof}[Proof of Theorem~\ref{thm:hatbeta}] The argument is a refinement of
the method of \cite{FlaFo}, \S 5.3. The main technical improvement
consists in replacing the difference equations of \cite{FlaFo} by
ordinary differential equations. Here we also work in the more
general setting of currents instead of distributions.

For any $s>3/2$, consider the decomposition of the current
$\gamma\in \Omega_1^{-s}(SM)$ given by integration along a rectifiable
arc. By definition we have:
\begin{equation}
\label{eq:gammasplit} \gamma = \mathcal B^{-s}(\gamma)  +  \mathcal
R^{-s}(\gamma) \,.
\end{equation}
We are interested in the evolution of this decomposition under the
action of the geodesic flow. By the group property of the geodesic
flow $\{g_t\}$ and by  formula (\ref{eq:gammasplit}), for any $t$,
$\tau \in \R$ we obtain:
\begin{equation}
\label{eq:gammaidone}
\begin{aligned}
g_{t+\tau}^\ast\gamma &= g_\tau^\ast \mathcal B^{-s}(g_t^\ast \gamma)  +
g_\tau^\ast \mathcal R^{-s}(g_t^\ast \gamma)  =  \mathcal
B^{-s}(g_{t+\tau}^\ast \gamma)  + \mathcal R^{-s}(g_{t+\tau}^\ast
\gamma)\,.
\end{aligned}
\end{equation}
By  projection of  (\ref{eq:gammaidone}) under $\mathcal B^{-s}_\mu:
\Omega_1^{-s}(SM) \to \mathcal B^{-s}_U(H_\mu)$ we therefore have
\begin{equation}
\label{eq:gammaidtwo}
  \mathcal  B_\mu^{-s}(g_{t+\tau}^\ast \gamma)  =
 g_\tau^\ast \mathcal   B_\mu^{-s}(g_t^\ast \gamma)  +
 \mathcal  B_\mu^{-s} g_\tau^\ast  \mathcal  R_\mu^{-s}(g_t^\ast \gamma)   \,.
\end{equation}

We would like to differentiate the above identity
(\ref{eq:gammaidtwo}) with respect to the parameter $\tau\in \R$.
That is made possible by the following technical result whose proof
we postpone until \S\S~\ref{keyestimates}:

\begin{lemma}
\label{lemma:remainderextension}   For any $s\geq r >7/2$, for any rectifiable arc $\gamma$ in $SM$
and for any irreducible component $H_\mu\subset L^2(SM)$ of Casimir parameter $\mu\in \R\setminus\{0\}$, the current $\mathcal  R_\mu^{-s}(\gamma)  \in \Omega_1^{-s}(H_\mu)$ has a unique continuous extension $\mathcal R_\mu^{-s,-r}(\gamma)   \in \Omega_1^{-r}(H_\mu)$ and the following uniform bound holds. There exists a constant $C_{r,s}>0$ such that
$$
\Vert  \mathcal  R_\mu^{-s,-r}(\gamma) \Vert_{-r}  \leq C_{s,r}
(1+ \int_\gamma \vert \hat X \vert + \int_\gamma \vert \hat V\vert
)\,.
$$
\end{lemma}

Now, assuming Lemma~\ref{lemma:remainderextension}, we conclude the
proof of Theorem \ref{thm:hatbeta}.

 Let $s>9/2$.  By Lemma~\ref{lemma:remainderextension}, the current
 ${\tilde {\mathcal  R}}_\mu^{-s}( \gamma) :=  \mathcal
 R_\mu^{-s, -(s-1)}(\gamma)$ is well-defined, and
 the following limit exists in the Hilbert space $\Omega_1^{-s}(H_{\mu})$:
 $$
 \lim_{\tau \to 0}
 \frac{ g_\tau^\ast  \mathcal  R_\mu^{-s}(g_t^\ast \gamma) - \mathcal  R_\mu^{-s}(g_t^\ast \gamma) }{\tau}
 = \mathcal L_X {\tilde  {\mathcal  R}}_\mu^{-s} (g_t^\ast\gamma) \in  \Omega_1^{-s}(H_\mu)\,.
 $$
Thus, by differentiating (\ref{eq:gammaidtwo}) with respect to
$\tau$ at $\tau=0$, we obtain, for all $t\in \R$,
\begin{equation}
\label{eq:diffid} \frac d{dt} \mathcal   B_\mu^{-s}(g_t^\ast\gamma)
=  \mathcal L_X \mathcal B_\mu^{-s}(g_t^\ast\gamma) + \mathcal
B_\mu ^{-s} \mathcal L_X {\tilde {\mathcal
R}}_\mu^{-s}(g_t^\ast\gamma) \,.
\end{equation}
We  now write the above differential equation in coordinates. We write
\begin{equation}
\label{eq:Bcoord}
\begin{aligned}
\mathcal B^{-s}_\mu (g_t^\ast\gamma)&=\hat \alpha_{\mu,-s}^+(t) B^+_\mu
+\hat \alpha_{\mu,-s}^{-}(t) B^-_\mu \,, \\
\mathcal B^{-s}_\mu\mathcal L_X {\tilde {\mathcal
R}}_\mu^{-s}(g_t^\ast\gamma)
 &=\rho_{\mu,-s}^+(t) B^+_\mu \,+\, \rho_{\mu,-s}^-(t) B^-_\mu.
\end{aligned}
\end{equation}

If $\mu\neq  1/4$,  by Lemma \ref{lemma:EVcurrents} and equation
(\ref{eq:diffid}) we obtain the following formulas:
\begin{equation}
\label{eq:alpha} \frac d{dt} \hat\alpha_{\mu,-s}^\pm = \frac{1\mp
\nu}{2}  \hat\alpha_{\mu,-s}^\pm + \rho_{\mu,-s}^\pm\,  ;
\end{equation}
If $\mu=1/4$, we obtain the following formulas:
\begin{equation}
\label{eq:alphabis} \frac d{dt}  \begin{pmatrix} \hat\alpha_{\mu,-s}^+
\\  \hat\alpha_{\mu,-s}^-  \end{pmatrix} = \frac{1}{2} \begin{pmatrix} 1
&-1 \\ 0 & 1 \end{pmatrix}
 \begin{pmatrix} \hat\alpha_{\mu,-s}^+   \\  \hat\alpha_{\mu,-s}^-  \end{pmatrix}
 +    \begin{pmatrix} \rho_{\mu,-s}^+ \\  \rho_{\mu,-s}^-    \end{pmatrix}   \,.
\end{equation}
By writing down solutions of the above O.D.E.'s we conclude that,
for $\mu\neq 1/4$,
\begin{equation}
\label{eq:hatalpha}
\frac{\hat\alpha_{\mu,-s}^\pm(t)}{ \exp( \frac{1\mp \nu}{2} t)}   = \hat\alpha_{\mu,-s}^\pm(0) +
 \int_0^t \rho_{\mu,-s}^\pm(\tau) e^{- \frac{1\mp \nu}{2} \tau }\,d\tau
\end{equation}
while for $\mu=1/4$, after some elementary calculations,
\begin{equation}
\label{eq:hatalphabis}
\begin{aligned}
\frac{\hat\alpha_{1/4,-s}^+ (t) + \frac{t}{2} \hat\alpha_{1/4,-s}^- (t)}{\exp(\frac{t}{2})} \,= &
\,\hat\alpha_{1/4,-s}^+ (0) \\ + \,&\int_0^t [ \rho_{1/4,-s}^+(\tau) +\frac{\tau}{2}  \rho_{1/4,-s}^-(\tau)]
e^{-\tau/2}\, d\tau\,; \\
\frac{\hat\alpha_{1/4,-s}^- (t)}{\exp(\frac{t}{2})} \,=  \,\hat\alpha_{1/4,-s}^- (0)\, + \,&
\int_0^t  \rho_{1/4,-s}^-(\tau) e^{-\tau/2} \, d\tau \,.
 \end{aligned}
\end{equation}

We conclude the argument by proving that the integrals in formulas
(\ref{eq:hatalpha}) and  (\ref{eq:hatalphabis}) are absolutely
convergent (as $t\to +\infty$) and are absolutely and uniformly bounded in terms
of the transverse lengths of the rectifiable arc $\gamma$ in $SM$.

Since $M$ is a compact hyperbolic surface, the Casimir spectrum of the standard unitary representation
of $SL(2,\R)$ on $L^2(SM)$ is discrete. Thus, by the distorsion Lemma~\ref{lemma:distorsion} and by  formula (\ref{eq:Bcoord}), for all $s>9/2$ there exists a constant $C_s>0$ such that,
 for all $\mu \in \R^+$, the following estimate holds:
\begin{equation}
\label{eq:rhobound} \vert \rho_{\mu,-s}^\pm (t) \vert  \leq \frac{C_s}{  \Vert B^\pm_\mu\Vert_{-s}} \,
\Vert \mathcal L_X {\tilde {\mathcal  R}}_\mu^{-s}
(g_t^\ast\gamma)\Vert_{-s} \leq \frac{C_s}{  \Vert B^\pm_\mu\Vert_{-s}} \, \Vert {\tilde {\mathcal
R}}_\mu^{-s} (g_t^\ast\gamma)\Vert_{-s+1}\,,
\end{equation}
hence by Lemma \ref{lemma:remainderextension} there exists a
constant $C'_s>0$ such that
\begin{equation}
\label{eq:rhoboundbis} \vert \rho_{\mu,-s}^\pm (t) \vert \leq
\frac{C'_s}{  \Vert B^\pm_\mu\Vert_{-s}} \, (1 + \int_\gamma \vert \hat X \vert  +e^{-t} \int_\gamma
\vert \hat V \vert )\,.
\end{equation}
The above bound (\ref{eq:rhoboundbis}) immediately implies that
the integrals in formulas (\ref{eq:hatalpha}) and (\ref{eq:hatalphabis}) are
absolutely and uniformly bounded and convergent (as $t\to +\infty$),
hence the limits in the left hand side of both formulas exist.
The bound (\ref{eq:rhoboundbis}) also implies that such limits are independent
of $s >9/2$. In fact, by the distorsion Lemma \ref{lemma:distorsion} and by Lemma \ref{lemma:remainderextension}, for any $ s \geq r >9/2$ there are constants $C_s$, $C_{s,r} >0$  such that, for any rectifiable arc $\gamma$ and for all Casimir parameters $\mu>0$,
\begin{equation}
\begin{aligned}
&\Vert B^\pm_\mu\Vert_{-s} \, \vert \hat\alpha^\pm_{\mu,-s}(t) -
\hat\alpha^\pm_{\mu,-r}(t) \vert  \leq  C_s
 \Vert \mathcal B^{-s}_\mu(g_t^\ast\gamma) - \mathcal B^{-r}_\mu(g_t^\ast\gamma)\Vert_{-s} \\
 &=C_s \Vert \mathcal R^{-s}_\mu(g_t^\ast\gamma) - \mathcal
R^{-r}_\mu(g_t^\ast\gamma)\Vert_{-s}  \leq C_{s,r} (1 +
\int_\gamma \vert \hat X\vert + e^{-t}  \int_\gamma \vert \hat V
\vert )\,.
\end{aligned}
\end{equation}
Thus the limits  $\hat\beta^\pm_\mu(\gamma)$ of formulas~(\ref{eq:hatbeta})
and~(\ref{eq:hatbetabis}) exist and are well-defined and the speed of convergence is
 correctly given by the estimates in formulas~(\ref{eq:hatbetaspeed})
and~(\ref{eq:hatbetaspeedbis}).

Finally, the bound in formula (\ref{eq:hatbetabound}) follows from
the estimate (\ref{eq:rhoboundbis}) and the following bound. By the distortion Lemma
\ref{lemma:distorsion} and by the Sobolev embdedding theorem, there exist constants $C''_s$,
$C'''_s>0$ such that, for all Casimir parameters $\mu>0$,
$$
\vert \hat\alpha_{\mu,-s}^\pm(0) \vert \leq \frac{C''_s} {\Vert B^\pm_\mu\Vert_{-s}}\, \Vert \gamma
\Vert_{-s} \leq  \frac{C'''_s} {\Vert B^\pm_\mu\Vert_{-s}}\,  (1+ \int_\gamma \vert \hat X\vert +
\int_\gamma \vert \hat U\vert + \int_\gamma \vert \hat V\vert)\,.
$$
The proof of Theorem \ref{thm:hatbeta} is therefore complete.
\end{proof}

\subsection{Proof of the main properties (Theorem \ref{thm:hatbetaprops}).}
The proof of the theorem requires a stronger
estimate on the current $\hat \beta_\mu$ than the one given above in
Theorem  \ref{thm:hatbeta}. The following result is a crucial step
in that direction as well as in the proof of the invariance under
the unstable horocycle.

For any rectifiable arc $\gamma$, let $\Gamma_{ws} (\gamma)$ be the
set of all rectifiable arcs obtained projecting the arc $\gamma$
under the unstable horocycle holonomy on any leaf of the weak stable
foliation of the geodesic flow. The weak stable foliation of the
geodesic flow is the $2$-dimensional foliation  tangent to the
integrable distribution $\{X,U\}$ in the tangent bundle of $SM$.

\begin{lemma}
\label{lemma:wsproj}
 For any $\mu\in \R^+$, for any rectifiable arc $\gamma$ and any $\gamma_{ws} \in \Gamma_{ws} (\gamma)$,
$$
\hat\beta^\pm_\mu(\gamma) = \hat\beta^\pm_\mu(\gamma_{ws})\,.
$$
\end{lemma}
\begin{proof} For any $\gamma_{ws} \in \Gamma_{ws}(\gamma)$ and let
$D(\gamma_{ws}, \gamma)$ be the surface spanned by the trajectories
of the unstable horocycle flow projecting $\gamma$ onto
$\gamma_{ws}$. The surface $D(\gamma_{ws}, \gamma)$ is  the union of
all unstable horocycle arcs $I$  such that the boundary of $I$ is
contained in $\gamma_{ws} \cup \gamma$ and the interior of $I$ is
disjoint from $\gamma_{ws} \cup \gamma$.  The surface
$D(\gamma_{ws}, \gamma)$ defines by integration a current of
dimension $2$ (and degree $1$).  Let $g_{-t}(\gamma_{ws})$ and
$g_{-t}(\gamma)$ be the rectifiable arcs which are direct images of
$\gamma_{ws}$ and $\gamma$
 under the diffeomorphism $g_{-t}:SM\to SM$ respectively. By definition the arcs $g_{-t}(\gamma_{ws})$ and $g_{-t}(\gamma)$ are respectively the support of the currents $g_t^\ast \gamma_{ws}$ and
 $g_t^\ast \gamma$. By definition we have the following identity of currents
$$
g_t^\ast D(\gamma_{ws},\gamma)= D(g_{-t}(\gamma_{ws}),
g_{-t}(\gamma)) \,.
$$
Since the current $\partial D(\gamma_{ws},\gamma) - (\gamma
-\gamma_{ws})$ is composed of two arcs of orbits of the unstable
horocycle flow, it follows that
\begin{equation}
\label{eq:boundary}
 \partial[ g_t^\ast D(\gamma_{ws},\gamma)] - (g_t ^\ast\gamma -g_t ^\ast\gamma_{ws}) = g_t^\ast
 [\partial D(\gamma_{ws},\gamma) - (\gamma -\gamma_{ws})]  \to 0\,.
\end{equation}
\end{proof}

\begin{lemma}
\label{lemma:boundedarea}
The area of $g_{-t} D(\gamma_{ws},\gamma)$ is uniformly bounded for
all $t>0$.
\end{lemma}

\begin{proof}
For $p\in\gamma$, let $\tau(p) >0$ be the length of the unstable horocycle arc lying in $D:=D(\gamma_{ws}, \gamma)$.
By construction the function $\tau:\gamma \to \R^+$ is continuous, hence $\tau_\gamma
:= \sup \{ \tau(p) \vert p\in \gamma\} <+\infty$. We write
\begin{equation}
\label{eq:D}
D=\bigcup\limits_{p\in \gamma} \bigcup\limits_{\tau\in [0, \tau(p)]} h^V_\tau (p),
\end{equation}
whence, letting $dl$ be the length parameter on $SM$,  for the area of $D$
we may write
\begin{equation}
\label{eq:AreaD}
\text{ \rm Area}(D)=\int\limits_{\gamma} \tau dl\,.
\end{equation}
Since,  by formula~(\ref{eq:D}), for any $t\in \R$,
$$
g_{-t}D=\bigcup\limits_{p\in g_{-t}\gamma} \bigcup\limits_{\tau\in [0, e^{-t}\tau(p)]} h^V_\tau (p),
$$
and, since $\text{\rm Length} (g_{-t}\gamma) \leq e^t \text{\rm Length} (\gamma)$, by formula
(\ref{eq:AreaD}) we have
\begin{equation}
\text{\rm Area}(g_{-t}D)=\int\limits_{g_{-t}\gamma} e^{-t}\tau dl  \leq   \tau_\gamma e^{-t}
\text{\rm Length} (g_{-t}\gamma) \leq  \tau_\gamma \text{\rm Length} (\gamma) \,,
\end{equation}
thus the lemma is proved.
\end{proof}

\medskip
It follows from Lemma \ref{lemma:boundedarea}  and formula (\ref{eq:boundary}) that for any $s>7/2$,
$$
 \sup_{t>0} \Vert g_t^\ast\gamma -g_t^\ast \gamma_{ws} \Vert_{-s}\,  < \,+\infty\,,
$$
hence by continuity of orthogonal projections
$$
 \sup_{t>0}  \vert \hat\alpha^\pm_{\mu,-s}(\gamma) - \hat\alpha^\pm_{\mu,-s}(\gamma_{ws}) \vert
 \,<\, +\infty\,.
$$
The statement of the Lemma now follows immediately from the
definition of the currents $\hat \beta^\pm_\mu$ in the statement of
Theorem \ref{thm:hatbeta}.

We return to the proof of Theorem~\ref{thm:hatbetaprops}.

\emph{Additivity}. It follows from the definition of $\hat\beta^\pm_\mu$ in
the statement of Theorem \ref{thm:hatbeta} and from the linearity of
projections.

\emph{Geodesic Scaling}. It follows immediately from the definitions in
Theorem~\ref{thm:hatbeta} and from the group property of the
geodesic flow. In fact, for any $\mu \neq 1/4$,
\begin{equation}
\begin{aligned}
\hat\beta^\pm_\mu(g_{-t}\gamma)&= \lim_{\tau \to +\infty}
\frac{ \hat\alpha^\pm_{\mu,-s} (g_{t+\tau}^\ast\gamma)} {\exp( \frac{1\mp \nu}{2} \tau)} \\
&=  e^{ \frac{1\mp \nu}{2} t}\,  \lim_{\tau \to +\infty} \frac{
\hat\alpha^\pm_{\mu,-s} (g_{t+\tau}^\ast \gamma)} {\exp( \frac{1\mp
\nu}{2}(t+ \tau) )} = e^{ \frac{1\mp \nu}{2} t }\,
\hat\beta^\pm_\mu(\gamma)\,.
\end{aligned}
\end{equation}
For $\mu=1/4$,  the geodesic scaling properties of the current
$\hat\beta^-_{1/4}$ can be proved as above (as for $\mu\neq 1/4$),
while for the current $\hat\beta^+_{1/4}$, the following holds:
\begin{equation}
\begin{aligned}
\hat\beta^+_{1/4}(g_{-t}\gamma)&= \lim_{\tau \to +\infty}
\frac{(\hat\alpha^+_{1/4,-s} + \frac{\tau}{2} \hat\alpha^-_{1/4,-s})(g_{t+\tau}^\ast\gamma)} {\exp(
\frac{\tau}{2})} \\
&= e^{\frac{t}{2}}  \lim_{\tau \to +\infty}  \frac{(\hat\alpha^+_{1/4,-s} + \frac{\tau}{2}
\hat\alpha^-_{1/4,-s})(g_{\tau}^\ast\gamma)} {\exp(\frac{\tau-t}{2})} \\
&=  e^{\frac{t}{2}} \left( \hat\beta^+_{1/4}(\gamma)- \frac{t}{2} \hat\beta^-_{1/4}(\gamma)\right).
\end{aligned}
\end{equation}

\emph{Unstable Horocycle Invariance}. It follows from Lemma
\ref{lemma:wsproj}. In fact, for any rectifiable arc $\gamma$ and
for any $t>0$, the arcs $\gamma$ and $h^V_t(\gamma)$ have common
weak stable projections. In other terms, the identity
$\Gamma_{ws}(\gamma)= \Gamma_{ws}(h^V_t\gamma)$ holds by definition.
Let then $\gamma_{ws} \in \Gamma_{ws}(\gamma)=
\Gamma_{ws}(h^V_t\gamma)$. By Lemma \ref{lemma:wsproj} we have:
$$
\hat\beta_\mu^\pm(h^V_t\gamma)= \hat\beta_\mu^\pm(\gamma_{ws}) =
\beta_\mu^\pm(\gamma)\,.
$$
\emph{H\"older property}.  Let $\gamma_{ws}$ be any rectifiable arc
contained in a weak stable manifold of the geodesic flow and let
$$
t = \log ( \int_{\gamma_{ws}} \vert \hat U \vert) \,.
$$

By construction the transverse lengths of the rectifiable arc
$\gamma_{ws}(t):=g_t(\gamma_{ws})$ satisfy the following
properties:
$$
\int_{\gamma_{ws}(t)} \vert \hat X \vert  = \int_{\gamma_{ws}} \vert
\hat X \vert \quad \text{ and } \quad \int_{\gamma_{ws}(t)} \vert
\hat U \vert=1\,.
$$
Thus by Theorem \ref{thm:hatbeta} and by the geodesic scaling
properties of the finitely-additive measures $\hat\beta^\pm$, the following bounds hold:
for all $s>9/2$ there exists a constant $C_s>0$ such that,
for $\mu\neq 1/4$, the following bound holds:
\begin{equation}
\label{eq:scalingbound}
\begin{aligned}
\vert \hat \beta^\pm_\mu (\gamma_{ws}) \vert &= e^{\frac{ 1\mp
\re\,\nu}{2}t} \vert \hat \beta^\pm_\mu (\gamma_{ws}(t)) \vert  \\
&\leq \frac{C_s}{ \Vert B^\pm_\mu\Vert_{-s}} (1+ \int_{\gamma_{ws}} \vert \hat X \vert )
(\int_{\gamma_{ws}} \vert \hat U \vert)^{\frac{ 1\mp
\re\,\nu}{2}} \,;
\end{aligned}
\end{equation}
for $\mu = 1/4$, the following bounds hold:
\begin{equation}
\label{eq:scalingboundbis}
\begin{aligned}
\vert \hat \beta^+_{1/4}(\gamma_{ws}) \vert &= e^{\frac{t}{2}}
\vert  \left(\hat \beta^+_{1/4}  - \frac{t}{2}
 \hat \beta^-_{1/4}\right)(\gamma_{ws}(t))\vert  \\ & \leq  \frac{C_s}{ \Vert B^+_{1/4}\Vert_{-s}}
 (1+ \int_{\gamma_{ws}} \vert \hat X \vert )
  (\int_{\gamma_{ws}} \vert \hat U \vert)^{\frac{1}{2}+}\,; \\
 \vert \hat \beta^-_{1/4}(\gamma_{ws}) \vert &= e^{\frac{t}{2}}
\vert  \hat \beta^-_{1/4} (\gamma_{ws}(t))\vert  \\ & \leq  \frac{C_s}{ \Vert B^-_{1/4}\Vert_{-s}}
(1+ \int_{\gamma_{ws}} \vert \hat X \vert )
 (\int_{\gamma_{ws}} \vert \hat U \vert)^{\frac{1}{2}} \,.
\end{aligned}
\end{equation}
We recall that we adopt the notation
$$
L^{\frac{1}{2}+} =L^{\frac{1}{2}} (1+\vert \log L\vert) \,, \quad \text{ \rm for all }\, L>0\,.
$$

Let now $\gamma$ be any rectifiable arc. By the $SL(2,\R)$
commutation relations, there exists a rectifiable arc $\gamma_{ws}
\in \Gamma_{ws}(\gamma)$ such that
\begin{equation}
\label{eq:projbound} \int_{\gamma_{ws}} \vert \hat X\vert  \leq
\int_{\gamma} \vert \hat X\vert + \int_{\gamma} \vert \hat U \vert
\int_{\gamma} \vert \hat V \vert \quad \text{and} \quad
\int_{\gamma_{ws}} \vert \hat U\vert = \int_{\gamma} \vert \hat
U\vert \,.
\end{equation}
It follows from estimates (\ref{eq:scalingbound}) and
(\ref{eq:projbound}) that, for $\mu\neq 1/4$,
\begin{equation*}
\vert \hat \beta^\pm_\mu (\gamma_{ws}) \vert
\leq  \frac{C_s}{ \Vert B^\pm_\mu\Vert_{-s}} \left( 1+ \int_{\gamma} \vert \hat X\vert + \int_{\gamma}
\vert \hat U \vert   \int_{\gamma} \vert \hat V \vert \right)
\left(\int_{\gamma} \vert \hat U \vert\right)^{\frac{ 1\mp
\re\,\nu}{2}}\,,
\end{equation*}
while for $\mu=1/4$ ($\nu =0$), by estimates (\ref{eq:scalingboundbis}) and
(\ref{eq:projbound}),
\begin{equation*}
\begin{aligned}
\vert  \hat \beta^+_{1/4}(\gamma_{ws}) \vert &\leq  \frac{C_s}{ \Vert B^+_{1/4}\Vert_{-s}}   \left( 1+
\int_{\gamma} \vert \hat X\vert + \int_{\gamma} \vert \hat U \vert
\int_{\gamma} \vert \hat V \vert \right) \left( \int_{\gamma}
\vert \hat U \vert\right)^{\frac{ 1}{2}+} \,, \\
\vert  \hat \beta^-_{1/4}(\gamma_{ws}) \vert &\leq  \frac{C_s}{ \Vert B^-_{1/4}\Vert_{-s}}  \left( 1+
\int_{\gamma} \vert \hat X\vert + \int_{\gamma} \vert \hat U \vert
\int_{\gamma} \vert \hat V \vert \right) \left( \int_{\gamma}
\vert \hat U \vert\right)^{\frac{ 1}{2}} \,.
\end{aligned}
\end{equation*}
By Lemma  \ref{lemma:wsproj} the above bounds
immdediately implies the H\"older property stated in Theorem
\ref{thm:hatbetaprops}, which is therefore completely proved.

\subsection{Proof of weak unstable vanishing (Corollary \ref{cor:wuvanishing}).}

On one hand, by the geodesic scaling property of Theorem
\ref{thm:hatbetaprops}, for any rectifiable arc $\gamma$ in $SM$
$$
\hat\beta^\pm_\mu (\gamma) =  \exp (- \frac{1\mp \nu}{2} t )
\beta^\pm_\mu(g_{-t}\gamma)\,;
$$
on the other hand, for any rectifiable arc $\gamma_{wu}$ contained
in a weak unstable manifold for the geodesic flow, by the H\"older
property of Theorem  \ref{thm:hatbetaprops},
$$
\vert\hat\beta^\pm_\mu(g_{-t}\gamma_{wu})\vert  \leq C_\mu \left( 1+
\int_{\gamma} \vert \hat X\vert \right)\,.
$$
It follows immediately that $\hat \beta^\pm_\mu (\gamma_{wu})=0$ as
stated.

\subsection{Proof of existence of dynamical projections (Corollary \ref{cor:basiccurrents}).}
For any $r>0$,  let $\mathcal B_+^{-r}(SM) \subset \Omega_1^{-r}(SM)$ be the closed subspace of basic currents for the stable horocycle foliation, supported on irreducible unitary representations of the principal and complementary series, and let $\mathcal B^{-r}_+:
\Omega_1^{-r}(SM) \to \mathcal B_+^{-r}(SM)$  be the orthogonal projection.

By definition (see formula~(\ref{eq:basiccurrentproj})), the orthogonal projection $\mathcal B_+^{-r} (g^\ast _t \gamma)$ is given by the formula:
$$
\mathcal B_+^{-r} (g^\ast _t \gamma) = \sum_{ \mu \in \Spec(\square) \cap \R^+}
\hat \alpha^+_{\mu,-r} (g^\ast_t\gamma) B^+_\mu  +  \hat \alpha^-_{\mu,-r} (g^\ast_t\gamma) B^-_\mu\,.
$$
By Lemma~\ref{lemma:EVcurrents}  we have that, for Casimir parameters $\mu \in
 \R^+\setminus\{1/4\}$,
 $$
 g^\ast_{-t} (B^\pm_\mu) =  \exp (- \frac{1\mp \nu}{2} t )  B^\pm_\mu\,,
 $$
while for $\mu=1/4$,
$$
 g^\ast_{-t} \begin{pmatrix} B^+_{1/4} \\ B^-_{1/4}
\end{pmatrix} = \exp(- \frac{t}{2})   \begin{pmatrix}  1 & 0 \\ \frac{t}{2} & 1
\end{pmatrix}
 \begin{pmatrix} B^+_{1/4} \\ B^-_{1/4} \end{pmatrix}\,.
$$
It follows then from Theorem~\ref{thm:hatbeta} that the series
\begin{equation}
\label{eq:Bseries}
g^\ast_{-t} \mathcal B_+^{-r} (g^\ast _t \gamma)= \sum_{ \mu \in \Spec(\square) \cap \R^+}
\hat \alpha^+_{\mu,-r} (g^\ast_t\gamma) g^\ast _{-t} B^+_\mu  +  \hat \alpha^-_{\mu,-r} (g^\ast_t\gamma) g^\ast_{-t} B^-_\mu
\end{equation}
converges in the distributional sense as $t\to +\infty$ to the series
\begin{equation}
\label{eq:hatB}
\hat B(\gamma)=\sum_{ \mu \in \Spec(\square) \cap \R^+}  \hat \beta^+_\mu (\gamma) B^+_\mu +
\hat \beta^-_\mu (\gamma) B^-_\mu\,.
\end{equation}
In fact,  by Theorem~\ref{thm:hatbeta}  there exists a constant $C>0$ such that, for all $s \geq r >9/2$,
for all Casimir parameter $\mu>0$ and for all $t\in \R^+$ the following bound holds :
$$
 \vert \hat \alpha^\pm_{\mu,-r} (g^\ast_t\gamma) \vert \Vert   g^\ast _{-t} B^\pm_\mu   \Vert_{-s}
 \leq C  \frac{\Vert  B^\pm_\mu  \Vert_{-s}}{ \Vert  B^\pm_\mu  \Vert_{-r}}\,.
$$
The dual Sobolev norms in the above estimate can be compared as follows: since the
distributions $D^\pm_\mu \in \mathcal D'(H_\mu)$ for all $\mu>0$, for any $\sigma\in \R^+$,
$$
(1+\mu)^{\frac{\sigma}{2}} \,  \Vert  D^\pm_\mu \Vert_{-s} = \Vert (I + \square)^\frac{\sigma}{2}
 D^\pm_\mu \Vert_{-s} \leq  \Vert  D^\pm_\mu \Vert_{-s+\sigma} \,.
$$
By the Weyl asymptotics for the Laplace-Beltrami operator on a compact  hyperbolic surface, for any
$\sigma >1$,
$$
\sum_{\mu \in \Spec(\square) \cap \R^+}   \left(\frac{1}{1+\mu}\right)^\sigma\, < \, +\infty\,,
$$
hence for every $s>r+1>11/2$ the series in formula~(\ref{eq:Bseries}) is absolutely uniformly
convergent to the current $\hat B$ defined in formula~(\ref{eq:hatB}) in the Sobolev space of currents
$\Omega_1^{-s}(SM)$. Finally, by the uniform convergence of the series in formula~(\ref{eq:hatB}),  all the properties of the current $\hat B(\gamma)$ stated in the corollary (additive property, weak unstable
vanishing,  unstable horocycle invariance ad H\"older property) follow from the corresponding properties for the finitely additive measures  $\hat\beta^\pm_\mu(\gamma)$ stated in Theorem~\ref{thm:hatbetaprops} and Corollary~\ref{cor:wuvanishing}.

\section{Additive Cocycles and Limit Distributions.}
\label{sec:Cocycles}

In this section we prove our results on additive H\"older cocycles for the horocycle flow
(Theorem~\ref{thm:cocycleproperties}). We then derive the approximation theorem
(Theorem~\ref{thm:approximation}) and our results on limit distributions of ergodic integrals
of the horocycle flow.

\subsection{Proof of the cocycle theorem (Theorem~\ref{thm:cocycleproperties}).}

Let us recall that the functions $\beta^\pm_\mu:SM\times \R\to \C$ are defined in terms
of the finitely additive measures $\hat\beta^\pm_\mu$ on rectifiable arcs.
For any $(x,T)\in SM\times \R$, let $\gamma_U(x,T)$ be the oriented horocycle arc
$$
\gamma_U(x,T):= \{ h^U_t (x) \vert   t\in [0,T]\}\,.
$$
For every Casimir parameter $\mu>0$, let
$$
\hat\beta^\pm_\mu(x,T) := \hat \beta^\pm_\mu[\gamma_U(x,T)] \,.
$$
By Corollary \ref{cor:approximation}, we derive the following approximation results for
functions supported on a single irreducible component. For functions supported on the complementary
series we have
\begin{corollary}
\label{cor:maincompl} For any $\mu\in (0, 1/4)$, there exists $\varepsilon_\mu>0$ such that the
following holds. Let $f\in W^s(H_\mu)$  ($s>11/2$) be any function such
that  $D_{\mu}^-(f)\neq 0$. Then
$$
\max_{T\in [0,1]} \big|\frac {1}{ D_{\mu}^-(f)\exp (\frac{1+\nu}{2}
t)} \int_0^{T e^t} f\circ h^U_\tau (x)d\tau -\beta^-_\mu(g_{t}
x,T)\big|=O(\exp(-\varepsilon_\mu t)).
$$

Let $f\in W^s(H_\mu)$  ($s>11/2$) be any function such that
$D_{\mu}^-(f)= 0$, but $D_{\mu}^+(f)\neq 0$. Then
$$
\max_{T\in [0,1]} \big|\frac {1}{ D_{\mu}^+(f)\exp (\frac{1-\nu}{2}
t)} \int_0^{T e^t} f\circ h^U_\tau
(x)d\tau-\beta^+_\mu(g_{t}x,T)\big|=O(\exp(-\varepsilon_\mu t)).
$$
\end{corollary}
For functions  supported on the principal
series we have

\begin{corollary}
\label{cor:mainprincip} For all $\mu>1/4$, there exists $\varepsilon_\mu>0$
such that the following holds. For any function $f\in W^s(H_\mu)$ ($s>11/2$),
$$
\begin{aligned}
\max_{T\in [0,1]} \big| ( \frac {1}{ \exp(\frac{t}{2})}
&\int_0^{Te^t} f\circ h^U_\tau(x)d\tau  \,-\,  \beta^+_\mu (g_{t}
x,T) D_{\mu}^+  (f) \exp( -\frac{\nu t}{2})     \\ - & \beta^-_\mu
(g_{t} x,T) D_{\mu}^-  (f) \exp(\frac{\nu t}{2})
   \big|=O(\exp(-\varepsilon_\mu t))\,.
\end{aligned}
$$
For $\mu =1/4$, there exists $\varepsilon>0$ such that, for any $f\in W^s(H_{1/4})$ ($s>11/2$),
$$
\begin{aligned}
\max_{T\in [0,1]} \big| ( \frac {1}{ \exp(\frac{t}{2})}
\int_0^{Te^t} & f\circ h^U_\tau(x)d\tau \,-\,   \beta^+_{1/4} (g_{t}
x,T) D_{1/4}^+  (f)   \\  -  & \beta^-_{1/4} (g_{t} x,T)
[D_{1/4}^-  (f) -\frac{t}{2}  D_{1/4}^+  (f)]
\big|=O(\exp(-\varepsilon t))\,.
 \end{aligned}
$$
\end{corollary}

We proceed with the proof of Theorem~\ref{thm:cocycleproperties}.

\emph{Cocycle property}. It follows from the additivity property of the measures
$\hat\beta^\pm_\mu$ and from the cocycle properties of horocycle
arcs:
$$
\gamma(x, S+T)= \gamma(x,S) \cup \gamma(h^U_Sx, T) \,, \quad \text{
for all } (x,S,T)\in SM\times \R^2\,.
$$
\emph{Geodesic scaling}. It follows from the geodesic scaling property of
the measures $\hat\beta^\pm_\mu$, since by the commutation relation
(\ref{eq:commutflows}),
$$
g_{-t} \gamma(x,T) =  \gamma(g_{-t} x, Te^t) \, , \quad \text{ for
all } (x,T,t) \in SM\times \R^2\,.
$$
\emph{H\"older property}. It follows from the H\"older property of the
currents $\hat\beta^\pm_\mu$. In fact, if $\gamma:=\gamma(x,T)$ is a
stable horocycle arc, then
$$
\vert \int_\gamma \hat U \vert=\vert T \vert \,, \quad \int_\gamma
\hat X= \int_\gamma \hat V =0\,.
$$
\emph{Orthogonality}. Take $T\in {\mathbb R}$.  By Corollaries~\ref{cor:maincompl} and~\ref{cor:mainprincip}, the function $\beta_{\mu}^{\pm}(\cdot,T)$ is the uniform limit of normalized ergodic
integrals, i.e., continuous functions lying in the space $H_{\mu}$, and so the function $\beta_{\mu}^{\pm}(\cdot,T)$ must itself belong to $H_{\mu}\subset L^2(SM)$.

\subsection{Proof of the approximation theorem (Theorem~\ref{thm:approximation}).}

For all rectifiable arcs $\gamma \subset SM$ and for all Casimir parameters $\mu>0$, let
$\hat B_\mu(\gamma) \in \mathcal B^{-s}_U(SM)$ be the basic current defined as follows:
$$
\hat B_\mu(\gamma) := \hat \beta^+_\mu(\gamma) B^+_\mu + \hat \beta^-_\mu(\gamma) B^-_\mu\,.
$$
It follows from the bounds  (\ref{eq:hatbetaspeed}) and (\ref{eq:hatbetaspeedbis}) in Theorem~\ref{thm:hatbeta} that, for any $r>9/2$ there exists a constant $C_r>0$ such that, for all $t \geq 0$,
\begin{equation}
\Vert (\mathcal B^{-r}_\mu \circ g^\ast_t) (\gamma) -  (\hat B_\mu \circ g^\ast_t)(\gamma) \Vert_{-r}
\leq \,C_r (1+ \int_\gamma \vert \hat X\vert + e^{-t}\int_\gamma \vert \hat V\vert)\,.
\end{equation}
By Lemma~\ref{lemma:remainderextension} and from the splitting formula~(\ref{eq:gammasplit})
there exists a constant $C'_r>0$ such that, for all $t\geq 0$,
\begin{equation}
 \Vert (\mathcal B^{-r}_\mu \circ g^\ast_t) (\gamma) - g^\ast_t(\gamma) \Vert_{-r} \leq
 C'_r(1 + \int_\gamma \vert \hat X\vert + e^{-t}\int_\gamma \vert \hat V\vert)\,.
 \end{equation}

Let $\gamma_+$ denote the projection of the current $\gamma$ onto the components of the principal
and complementary series. By orthogonality and by  the Weyl asymptotics for hyperbolic surfaces,
for any $s>r+1$, there exists $C_s>0$ such that
\begin{equation}
\label{eq:errorplus}
\Vert g^\ast_t(\gamma_+) -  (\hat B\circ g^\ast_t) (\gamma) \Vert_{-s}
\leq \,C_s  (1 + \int_\gamma \vert \hat X\vert + e^{-t} \int_\gamma \vert \hat V\vert)\,.
\end{equation}

Taking $t=0$, we obtain Theorem~\ref{thm:approximation}. In fact, for any $1$-form $\lambda\in \Omega_1^s(SM)$ ($s>11/2$) supported on the irreducible components of the principal and complementary series we have
\begin{equation}
 \vert \int_\gamma \lambda  - \hat B_\lambda(\gamma) \vert =  \vert
 <\gamma_+ - \hat B(\gamma),\lambda>\vert \leq C_s \Vert \lambda\Vert_s  (1 + \int_\gamma \vert \hat X\vert + \int_\gamma \vert \hat V\vert)   \,.
\end{equation}

The proof is complete.

\subsection{Proof of the limit theorems: complementary series (Theorem \ref{thm:limitdistcompl}).}

We now assume that our hyperbolic surface admits complementary series, that is,
the spectrum of the Laplace operator has eigenvalue in the open interval
$(0,1/4)$. Let $s>11/2$ and  consider smooth functions with non-trivial projection
on the complementary series components.
Let $\mu_f \in (0,1)$ be the smallest
Casimir parameter appearing (non-trivially) in the decomposition of a zero-average
function $f\in C^\infty(SM)$. Let $\nu_f:= \sqrt{1-4\mu_f}$.
 Let $H_1, \dots, H_k\subset L^2(SM)$ be the collection
of all irreducible components of Casimir parameters $\mu_1= \dots=\mu_k=\mu_f$
and let $\{D^\pm_1, \dots, D^\pm_k\}$ be the basis of eigenvectors of the geodesic flow
of the space of invariant distributions supported on  $W^s(H_1\oplus\dots\oplus H_k)$
and let $\beta^\pm, \dots, \beta^\pm_k: SM\times \R \to \C$ be the corresponding cocycles
for the horocycle flow.

The main step in the proof is the following approximation Lemma which immediately follows from
the approximation theorem (Theorem~\ref{thm:approximation}).

\begin{lemma}
\label{lemma:leadingterm}
There exists $\alpha>0$ such that the following holds. For every
$s>11/2$ there exists a constant $C_s>0$ such that, for every function
$f\in W^s(SM)$ of zero average, for all $(x,T)\in SM\times \R$ and $t>0$,
$$
\vert \frac{1}{ e^{ \frac{1+\nu_f}{2} t} } \int_0^{Te^t} f \circ h^U_\tau (x) d\tau  -  \sum_{i=1}^k
D^-_i(f) \beta^-_i(g_t x,T) \vert
 \leq  C_s \Vert f \Vert_s e^{-\alpha t}\,.
$$
\end{lemma}

\smallskip
The cocycles $\beta^-_1,\dots, \beta^-_k$ (in fact, all cocycles $\beta^\pm_\mu)$ have zero average but
are not identically zero on $SM$. It follows that, for all $i\in \{1, \dots, k\}$, we have
$$
\Vert \beta^-_i (\cdot, T)\Vert \not =0 \,, \quad \text{ \rm for all } T\in \R\setminus\{0\}\,.
$$
By the orthogonality property of Theorem~\ref{thm:cocycleproperties}, the random variables
$$\beta^-_1(\cdot, T), \dots, \beta^-_k(\cdot, T)$$
are orthogonal (uncorrelated). By Lemma~\ref{lemma:leadingterm}, for any $s>11/2$ and for any 
function $f\in W^s(SM)$ of zero average such that $(D^-_1(f), \dots, D^-_k(f)) \not= (0, \dots, 0)$, we have
\begin{equation}
\left|
\frac{ \Vert \int\limits_0^{Te^t} f\circ h_t^U(x) \,dt  \Vert }{e^{\frac{1+\nu_f}{2} t} \left( \sum_{i=1}^k
\vert D^-_i (f)\vert^2  \Vert \beta^-_i(\cdot, T)\Vert^2 \right)^{1/2} }\, - \,1\right|
\le C_s e^{-\alpha t}\,.
\end{equation}
Finally, by Lemma~\ref{lemma:leadingterm} and by definition of the L{\'e}vy-Prohorov metric \cite{Billingsley},
for all $T\in [0,1]$ and all $t>0$, we have
\begin{equation}
d_{LP}\left( \mathfrak M_t(f,T), P_{cp}(f,T)\right) \leq C_s \Vert f\Vert_s e^{-\alpha t}\,.
 \end{equation}

The Theorem is proved.

\subsection{Proof of the limit theorems: principal series (Theorem~\ref{thm:limit_torus}).}
We turn next to limit theorems for functions supported on the principal series.
We prove our main theorem (Theorem~\ref{thm:limit_torus}) on the asymptotics
of probability distributions of normalized ergodic integrals and derive our conditional
results on the uniqueness of the limit distributions (Corollary~\ref{cor:limitdistprincipal1}
and Corollary~\ref{cor:limitdistprincipal2}).

\smallskip
 Let us recall that, by construction, for any Casimir parameter $\mu>1/4$,
$$
D^-_{\mu} =  \overline{ D^+_\mu} \quad \text{ and }\quad  \beta^-_\mu= \overline{\beta^+_{\mu}}\,.
$$
It follows that for any real-valued function $f \in W^s(H_\mu)$,
$$
\beta_f(x,T) = \re [ D^+_\mu(f) \beta^+_\mu(x,T)]\,, \quad \text{ for all }(x,T)\in SM\times \R\,.
$$
Let $\{\mu_n\}$ be the sequence of Casimir parameter in the interval $(1/4, +\infty)$ (listed with multiplicities),  let $\{D^\pm_{\mu_n}\}$ denote the sequence of \emph{normalized} horocycle invariant distributions and let  $\{\beta^\pm_{\mu_n} \}$ denote the corresponding sequence of  additive H\"older  cocycles.  For any sequence $\textbf{z}\in  \ell^1(\N, \C)$, let $\beta_{\textbf{z}}:SM\times \R \to \R$
be the  H\"older  additive cocycle for the horocycle flow defined  as follows:
 \begin{equation}
 \label{eq:betazeta}
 \beta_{\textbf{z}}:= \re[ \sum_{n\in \N}  z_n \beta^+_{\mu_n}] = \sum_{n\in \N}  (z_n \beta^+_{\mu_n} +
 \bar{z}_n \beta^-_{\mu_n})\,.
 \end{equation}
 It follows from Theorem~\ref{thm:cocycleproperties}, in particular from the uniform bound on
 additive cocycles given in the H\"older property, that the series in formula~(\ref{eq:betazeta}) is convergent for any $\textbf{z}\in \ell^1(\N, \C)$, hence the additive cocycle $\beta_{\textbf{z}}$ is
 well-defined. By the orthogonality property of the system  $\{\beta^+_{\mu_n}\}$ of additive cocycles,
 it follows that, for any~$\textbf{z} \in \ell^1(\N, \C)\setminus\{0\}$, the zero-average function
 $\beta_{\textbf{z}}(\cdot,T)$ is non-constant, hence
 $$
 \Vert \beta_{\textbf{z}} (\cdot, T) \Vert \not= 0\,, \quad \text{ \rm for all } T>0\,.
 $$
Let $s>11/2$. For any real-valued function $f \in W^s(SM)$ supported on irreducible components
of the principal series, we have
$$
\beta_f(x,T) = \re [ \sum_{n\in \N} D^+_{\mu_n}(f) \beta^+_{\mu_n} (x,T)]\,, \quad 
\text{ for all }(x,T)\in SM\times \R\,.
$$
Theorem~\ref{thm:limit_torus} follows from the following lemma that can
in turn be derived from  the approximation theorem  (Theorem~\ref{thm:approximation}).
For all $n\in \N$, let $\upsilon_n:= \sqrt{4\mu_n-1} \in \R^+$.

\begin{lemma}
\label{lemma:approxprincipal}
For every $s>11/2$ there exists $C_s>0$ such that, for any real-valued function
$f\in W^s(H_\mu)$ supported on irreducible components of the principal series, for all
$(x,T)\in SM\times \R$ and $t>0$, we have
$$
\vert \frac{1}{e^{ \frac{t}{2} } } \int_0^{Te^t} f \circ h^U_\tau (x) d\tau  -
 \re [ \sum_{n\in \N} D^+_{\mu_n} (f) e^{\frac{ i\upsilon_n t}{2}}  \beta^+_{\mu_n} (g_t x,T)] \vert
 \leq  C_s \Vert f \Vert_s e^{-\frac{ t}{2}}\,.
$$
\end{lemma}

By Theorem~\ref{thm:limit_torus} for real-valued functions supported on the principal series
limit distributions exist along time sequences  such that the orbit of the toral translation
of frequency $\upsilon/2 \in \R^\infty$ on the infinite torus $\T^\infty$ converges. Conjecturally
the limit does not exist otherwise. However, we are not able to prove that the limit distribution
does not exist for any function and any time sequence. Nevertheless, as a straightforward
consequence of Theorem~\ref{thm:limit_torus}, we derive the following result.

\begin{corollary}
\label{cor:limit_existence}
Let $f \in C^\infty(SM)$ be any real-valued function supported
on the irreducible components of the principal series. If the family probability distributions
$\mathfrak M_t(f,T)$ has a (unique) limit as $t\to +\infty$ for some $T\in [0,1]$, then
for all $T\in [0,1]$ the family of probability distributions $P_{pr} (f, \cdot ,T)$ is constant
on any minimal set of the linear flow of frequency $\upsilon/2 \in \R^\infty$ on
the infinite torus $\T^\infty$.
\end{corollary}

Finally, Corollary~\ref{cor:limitdistprincipal1} and Corollary~\ref{cor:limitdistprincipal2} follow from the above Corollary~\ref{cor:limit_existence} and, respectively, from Lemma~\ref{lemma:uniqueness}
and Lemma~\ref{lemma:genconstdisthighDstrong} in Section~\ref{rotsym}.

By Theorem~\ref{lemma:innerproducts}, proved in Section~\ref{innerproducts}, for every $n\in \N$ there exists $\theta_n^\ast \in \T$ such that the cocycle $\beta_n:= e^{i\theta_n^\ast} \beta^+_{\mu_n}$ has the property that, for all $T\in \R$,
\begin{equation}
\label{eq:realscalar}
 < \beta_n(\cdot,T),  \overline{\beta_n}(\cdot, T) >= e^{2i\theta_n^\ast} <\beta^+_{\mu_n}(\cdot,T),
 \beta^-_{\mu_n}(\cdot, T)>\, \in\, \R^+ \,.
\end{equation}
Let us assume that there exists a real-valued function $f\in W^s(SM)$ supported on finitely many irreducible components $H_1, \dots, H_m$ of the principal series such that, for some $T>0$, the probability distribution  $P_{pr}(f,\cdot,T)$ is constant on a $d$-dimensional subtorus $\T^d$ of the infinite dimensional torus $\T^\infty$. If the Casimir spectrum is simple on $H_1 \oplus \dots \oplus H_n$,
then there exist distinct integral vectors $v^{(1)}, \dots, v^{(m)} \in \R^d$ such that the probability distributions of the random variables
\begin{equation}
\label{eq:betacconstdistr}
\frac{ \re[ \sum_{k=1}^m \vert D^+_k(f)\vert e^{ i <v^{(k)}, \theta>} \beta_k(\cdot,T)]}
{\Vert  \re[ \sum_{k=1}^m \vert D^+_k(f)\vert e^{ i <v^{(k)}, \theta>} \beta_k(\cdot,T)]\Vert} \,.
\end{equation}
does not depend on $\theta \in \T^d$. By formula~(\ref{eq:realscalar}), we have
\begin{equation}
\begin{aligned}
A^2_k:= &\Vert \re \, \beta_k(\cdot,T)\Vert^2 > B^2_k:=  \Vert \im \,\beta_k(\cdot,T)\Vert^2 \,; \\
 &<\re\, \beta_k(\cdot,T), \im\, \beta_k(\cdot,T)>=0 \,.
\end{aligned}
\end{equation}
Thus by the orthogonality property, a calculation yields
\begin{equation}
\begin{aligned}
\Vert & \re[ \sum_{k=1}^m \vert D^+_k(f)\vert e^{ i <v^{(k)}, \theta>} \beta_k(\cdot,T)]\Vert^2_0 = \\
&\sum_{k=1}^m  \vert D^+_k(f)\vert^2 ( A_k^2 \cos^2 \!<v^{(k)}, \theta> +  B_k^2 \sin^2\!<v^{(k)}, \theta>)\,,
\end{aligned}
\end{equation}
hence Lemma~\ref{lemma:uniqueness} and Lemma~\ref{lemma:genconstdisthighDstrong} in Section~\ref{rotsym} do apply to the family of probability distributions in formula~(\ref{eq:betacconstdistr}). Corollary~\ref{cor:limitdistprincipal1} and Corollary~\ref{cor:limitdistprincipal2} follow.

\section{Duality and Classification.}
\label{sec:Duality}

In this section we prove the duality theorem (Theorem \ref{thm:duality}) and we derive
the classification theorem for finitely-additive measures on rectifiable arcs (Theorem \ref{thm:classification}). We conclude with a short discussion of the relations between finitely additive measures on rectifiable arcs and  the induced cocycles for the horocycle flow on one hand, and
invariant conformal distributions on the boundary of the Poincar\'e disk (see \cite{Co}).

\subsection{Proof of the duality theorems (Theorem \ref{product-duality} and Theorem \ref{thm:duality}).}
Recall that any $1$-dimensional, finitely-additive measures can be naturally interpreted
as currents of dimension $2$ (and degree $1$).
\begin{lemma}
\label{lemma:hatbetaformula}
For any $1$-dimensional, finitely-additive measure $\hat \beta\in\hat {\mathfrak  B}_V(SM)$ there
exists a $V$-invariant distribution $\hat D_\beta \in \mathcal I_V(SM)$ such that we have the following identity
of currents:
$$
\hat \beta = \hat D_\beta \wedge \hat U\,.
$$
\end{lemma}
\begin{proof}   Let $\hat\beta \in \mathcal B_V(SM)$. We recall that the $2$-dimensional current $\hat \beta$
is defined by Riemann integration as follows: for all smooth $2$-form $\eta\in \Omega_2^\infty(SM)$,
\begin{equation}
\label{eq:hatbetadef}
 <\hat \beta, \eta >:= \int\limits_{SM} \hat\beta \otimes\eta \,.
\end{equation}
By the weak unstable vanishing property (property $(2)$ in Definition~\ref{def:hatbetaprops}), the current
$\hat\beta$ has zero contraction on the vector fields $X$, $V$. In fact, for any smooth
$3$-form $\omega$ on $SM$, the restrictions of the $2$-forms $\imath_X\omega$ and
$\imath_V\omega$ to any leaf of the weak unstable foliation (tangent to the integrable
distribution $\{X,V\}$) vanish. It follows that, for all smooth $3$-form $\omega$ on $SM$,
$$
\begin{aligned}
<\imath_X \hat \beta, \omega> &=   \int\limits_{SM} \hat\beta \otimes \imath_X \omega =0 \,;\\
<\imath_V \hat \beta, \omega> &=   \int\limits_{SM} \hat\beta \otimes \imath_V \omega =0 \,.
\end{aligned}
$$
Since the dual forms $\{\hat X, \hat U, \hat V\}$ yield a frame of the cotangent bundle,
it follows from the identities $\imath_X \hat\beta= \imath_V \hat \beta=0$, that there
exists a distribution $\hat D_\beta \in \mathcal D'(SM)$ such that
\begin{equation}
\label{eq:D_beta}
\hat \beta = \hat D_\beta \wedge \hat U\,.
\end{equation}
In fact, any current of dimension $2$ and degree $1$ can be written as a linear
combination of the $1$-forms $\{\hat X, \hat U, \hat V\}$ with distributional coefficients.

We claim that the distribution $\hat D_\beta $ is $V$-invariant.
By the property of unstable horocycle invariance (property $(3)$  in Definition~\ref{def:hatbetaprops}) the current
$\hat \beta =\hat D_\beta \wedge \hat U$ is invariant under the unstable horocycle flow $\{h^V_t\}$, hence
\begin{equation}
\label{eq:Vderivative} 0=\mathcal L_V (\hat  D_\beta \wedge  \hat U) =
(\mathcal L_V \hat D_\beta) \wedge  \hat U +
 \hat D_\beta \wedge ( \mathcal L_V \hat U)\,.
\end{equation}
A straightforward calculation yields $\mathcal L_V \hat U=0$. Indeed, first write
$$
\mathcal L_V \hat U = \imath_V d\hat U + d  \imath_V \hat U =
\imath_V d\hat U \,.
$$
By a standard formula, for any pair of vector fields $W_1$, $W_2$, we have
$$
d\hat U (W_1, W_2) = W_1 \hat U (W_2) - W_2 \hat U (W_1)  - \hat U(
[W_1, W_2])\,.
$$
Recall that $U$, being the infinitesimal generator of the stable horocycle, satisfies the commutation
relation $[X,U]=U$ (see (\ref{eq:commutLie})), whence
$$
 d\hat U (X,U) = 1 \quad \text  {\rm and } \quad d \hat U (X,V) = d\hat U (U,V) =0\,.
$$
We have derived the  identity
\begin{equation}
\label{eq:dhatU} d\hat U   = - \hat X \wedge \hat U \,,
\end{equation}
which implies that $\mathcal L_V \hat U=\imath_V d\hat U=0$,  as
stated.

Formula (\ref{eq:Vderivative}) then implies that $\mathcal L_V \hat
D_\beta \wedge \hat U=0$, hence $\mathcal L_V \hat D_\beta=0$, that
is, the distribution $\hat D_\beta$ are $V$-invariant.
\end{proof}

We now complete the proof of the duality theorem (Theorem \ref{thm:duality}).
By Theorem \ref{thm:hatbetaprops}  and by Corollary \ref{cor:wuvanishing},
the measures $\hat \beta^\pm_\mu$ belong to the space $\hat {\mathfrak  B}_V(SM)$.
By Lemma \ref{lemma:hatbetaformula}, there exists $V$-invariant
distributions $\hat D^\pm_\mu \in \mathcal I_V(SM)$ such that
$$
\hat \beta^\pm_\mu = \hat D^\pm_\mu \wedge \hat U\,.
$$
\smallskip
Finally we prove that the $V$-invariant distributions $\hat
D^\pm_\mu$ are eigenvectors for the action of the geodesic flow. An
immediate computation yields
$$
g_t ^\ast( \hat U) = e^{-t}  \hat U \,, \quad \text{ for all } t\in
\R\,.
$$
By Theorem \ref{thm:hatbetaprops}, for $\mu\neq 1/4$, we have the following identity of currents
$$
g_t^\ast ( \hat \beta^\pm_\mu) = \exp(- \frac{1\mp \nu}{2} t)
\hat\beta^{\pm}_\mu \,, \quad
 \text{ for all } t\in \R\,,
$$
while for $\mu=1/4$ ($\nu=0$) and all $t\in \R$, we have
$$
g_t^\ast \begin{pmatrix}   \hat \beta^+_{1/4} \\ \hat \beta^-_{1/4} \end{pmatrix} =
\exp(-\frac{t}{2}) \begin{pmatrix} 1 & \frac{t}{2} \\ 0 & 1\end{pmatrix}
  \begin{pmatrix}   \hat \beta^+_{1/4} \\ \hat \beta^-_{1/4} \end{pmatrix} \,.
  $$
It follows that, for $\mu\neq 1/4$,
$$
g_t^\ast \hat D^\pm_\mu  = \exp( \frac{1\pm \nu}{2} t) \hat
D^\pm_\mu  \,, \quad \text{ for all } t\in \R\,,
$$
and finally, for $\mu= 1/4$ and all $t\in \R$,
$$
g_t^\ast  \begin{pmatrix}  \hat D^+_{1/4}  \\ \hat D^-_{1/4} \end{pmatrix} =
\exp(\frac{t}{2}) \begin{pmatrix} 1 & \frac{t}{2} \\ 0 & 1\end{pmatrix}
  \begin{pmatrix}   \hat D^+_{1/4}   \\ \hat D^-_{1/4} \end{pmatrix} \,,
$$
The proof of the duality theorem is  complete.
\medskip
\begin{remark}
The currents $\hat \beta^\pm_\mu$ are not  closed!  In fact, for all $\mu\neq 1/4$ we have
\begin{equation}
\label{eq:diffbeta} d\hat \beta^\pm_\mu =   \frac{(1\mp \nu)}{2}
\hat \beta^\pm_\mu \wedge   \hat X
\end{equation}
and, for $\mu= 1/4$ ($\nu=0$), we have
\begin{equation}
\label{eq:diffbetabis}
d \begin{pmatrix} \hat \beta^+_{1/4} \\   \hat \beta^-_{1/4} \end{pmatrix}    =
\frac{1}{2}   \begin{pmatrix}    1 & -1  \\  0 & 1 \end{pmatrix}
\begin{pmatrix} \hat \beta^+_{1/4} \\   \hat \beta^-_{1/4} \end{pmatrix}  \wedge \hat X\,.
\end{equation}
\end{remark}
\begin{proof} By (\ref{eq:dhatU}), since $\hat D^\pm_\mu$
are $V$-invariant distributions, for $\mu\neq 1/4$ we have
\begin{equation}
\begin{aligned}
d (\hat D^\pm_\mu \wedge \hat U) &= d\hat D^\pm_\mu  \wedge \hat U +
\hat D^\pm_\mu \wedge d\hat U  \\
&= (\mathcal L_X \hat D^\pm_\mu - \hat D^\pm_\mu)  \hat X \wedge
\hat U =  \frac{(1\mp \nu)}{2} ( \hat  D^\pm_\mu\hat U) \wedge   \hat X\,,
 \end{aligned}
\end{equation}
which is precisely formula (\ref{eq:diffbeta}). Similarly, for $\mu=1/4$,
\begin{equation}
\begin{aligned}
d (\hat D^+_{1/4} \wedge \hat U) &= (\mathcal L_X \hat D^+_{1/4} - \hat D^+_{1/4})
 \hat X \wedge \hat U= (\frac{1}{2} D^+_{1/4}\hat U - \frac{1}{2} D^-_{1/4}\hat U) \wedge \hat X\,,  \\
d (\hat D^-_{1/4} \wedge \hat U) & =  (\mathcal L_X \hat D^-_{1/4} - \hat D^-_{1/4})  \hat X \wedge
\hat U =  \frac{1}{2} ( \hat  D^-_{1/4}\hat U) \wedge   \hat X\,,
 \end{aligned}
\end{equation}
which yields formula (\ref{eq:diffbetabis}).
\end{proof}

\subsection{Proof of the classification theorem (Theorem \ref{thm:classification}).}

By Lemma~\ref{lemma:hatbetaformula}, for any $\hat \beta\in \hat {\mathfrak  B}_V(SM)$
there exists a $V$-invariant distribution $\hat D_\beta \in \mathcal I_V(SM)$ such that the following identity
holds in the sense of currents:
$$
\hat \beta = \hat D_\beta \wedge \hat U\,.
$$
 By the H\"older property (property $(4)$ in Definition~\ref{def:hatbetaprops}),  there exists $\alpha>0$ and a constant $C>0$ such that
for any rectifiable arc $\gamma$ of length not exceeding $1$ we have
$$
\vert \hat \beta ( g_{t} \gamma)  \vert  \leq  Ce^{-\alpha t}   \,, \quad \text {\rm for all } \, t\in \R \,.
$$
It follows by the above formulas that  for all $s>3/2$ there exists $C_s>0$ such that
\begin{equation}
\label{eq:D_betascalingbound}
\Vert  g_t^\ast \hat D_\beta \Vert_{-s}   \leq    C_s e^{(1-\alpha) t} \,, \quad \text{ \rm for all } t\in \R\,.
\end{equation}
The results of \cite{FlaFo} yield a complete classification of all $U$-invariant  and, equivalently,
of all $V$-invariant distributions by constructing a basis of generalized distributional eigenvectors for the action of geodesic flow on
$\mathcal I_U(SM)$ and $\mathcal I_V(SM)$.
For any $V$-invariant distribution $D_\mu$ supported on an irreducible subrepresentation of the discrete series of Casimir parameter
$\mu =-n^2+n$ ($n\in \Z^+$), the action of the geodesic flow is given by the formula (see \cite{FlaFo})
$$
 g_t^\ast \hat D_\mu =   e^{n t}   \hat D_\mu  \,, \quad \text {\rm for all } \, t\in \R \,,
$$
which is not compatible with the bound in formula (\ref{eq:D_betascalingbound}) for $t>0$ large.
It follows that the distribution $\hat D_\beta\in \mathcal I_V(SM)$ is supported on irreducible unitary
subrepresentations of the principal and complementary series.

By the duality theorem, for any $s>3/2$, there is a bounded (in fact, isometric) linear map $\mathcal I$
from the span $\hat {\mathcal B}^s_+(SM)$ of the system
$$
\{ \hat \beta_{\mu}^{\pm} \vert \mu \in \Spec(\square) \cap \R^+\} \subset \Omega_2^{-s}(SM)
$$
into the Sobolev space  $W^{-s}(SM)$. The map $\mathcal I$ is defined as follows. For a current
$\hat \beta^\pm_\mu\in \hat {\mathcal B}^s_+(SM)$ introduce a current
$$
\mathcal I (\hat \beta^\pm_\mu)  =  \hat X \wedge \hat\beta^\pm_\mu \wedge \hat V\,.
$$
It is immediate from the definitions of the Sobolev norms that the map $\mathcal I$ is isometric
on the system $\{ \hat \beta_{\mu}^{\pm} \vert \mu \in \Spec(\square) \cap \R^+\}$ with respect to the
Sobolev norms on the space $\Omega_2^{-s}(SM)$ of $2$-dimensional current  into the space $W^{-s}(SM)$
of distributions, hence it  can be extended by linearity and continuity to an isometry defined on the space
$\hat {\mathcal B}_+^s(SM)$.  We claim that the range of the isometry $\mathcal I$ coincides with the space
$\mathcal I^s_{V,+}(SM)$ of all $V$-invariant distributions supported on the principal and complementary series.
In fact, the map $\mathcal I $ is injective, hence by a dimension argument it is also surjective from the finite dimensional
space $\hat {\mathcal B}^s_\mu(SM):= \hat {\mathcal B}_+^s(SM) \cap \Omega^{-s}_2(H_\mu)$ onto the space
$\mathcal I^s_{V,\mu}(SM):= \mathcal I_V(SM)\cap W^{-s}(H_\mu)$, for any \emph{fixed} Casimir parameter
$\mu \in \Spec(\square)\cap \R^+$.

We can thus conclude that, since for every functional $\hat \beta \in {\mathcal B}_V(SM)$, the distribution
 $\hat D_\beta \in \mathcal I^s_{V,+}(SM)$ and since the isometry $\mathcal I$ maps $\hat
 {\mathcal B}_+^s(SM)$  onto  $\mathcal I^s_{V,+}(SM)$, the image ${\mathcal B}^s_V(SM)$ of the
 space ${\mathcal B}_V(SM)$ in $\Omega^{-s}(SM)$  coincide with the span $\hat
 {\mathcal B}_+^s(SM)$ of the system $\{ \hat \beta_{\mu}^{\pm} \vert \mu \in \Spec(\square) \cap \R^+\}$,  as stated.

\subsection{$\Gamma$-invariant conformal distributions.}

We now describe the relation between the finitely
additive measures $\hat \beta^\pm_\mu$ (lifted to the Lobachevsky plane)
and the $\Gamma$-invariant conformal distributions
on the boundary of the Poincar\'e disk.

\smallskip
Let $\mathcal L$ denote the Lebesgue measure on the circle $S^1$
seen as the boundary of the Poincar\'e disk $D$. Given a Fuchsian
group $\Gamma$, let
$$
\rho_g= \frac{ d g^{-1}\mathcal L}{d \mathcal L} \, , \quad g\in
\Gamma
$$
be the Radon-Nikodym multiplicative coycle for the action of
$\Gamma$ on $(S^1, \mathcal L)$.

 Given any complex number $\sigma\in \C$, following S.Cosentino \cite{Co} we let
 $^\Gamma \mathcal D'_\sigma (S^1)$ be the space of $\Gamma$-invariant conformal distributions with exponent $\sigma$, that is, those distributions $\phi\in \mathcal D'(S^1)$ such
 that
 $$
 g \phi =  \rho_g^{-\sigma} \cdot \phi \,, \quad \text{ for any } g\in \Gamma.
  $$
As explained in \cite{Co}, there is a natural (linear)
identification between $\Gamma$-invariant conformal distributions of
exponent $\sigma\in \C$ and invariant distributions for the
stable/unstable horocycle on the quotient $\Gamma \backslash
PSL(2,\R)$ which are eigenvectors of the geodesic flow of eigenvalue
$\sigma-1$.

In fact, let $h^U$, $h^V$ denote the stable, resp. unstable
horocycle  subgroups of $PSL(2,R)$. To any $U$-invariant
[$V$-invariant] distribution $D_U$  [$D_V$] on the space $\Gamma
\backslash PSL(2,\R)$ there correspond $\Gamma$-invariant
distribution $ \tilde {D}_U$ [$ \tilde {D}_V$] on the space of
stable [unstable] horocycles $PSL(2,\R) \slash h^U  $  [$PSL(2,\R)
\slash h^V $],  (by the the so-called $KAN$ decomposition of
$PSL(2,\R)$ such space can be identified with the space $KA$). If
$D_U$ [$D_V$] is also an eigenfunction of the geodesic flow of
eigenvalue $\sigma-1$, by the natural identification $KA \equiv S^1
\times \R$ one can write
$$
\tilde {D}_U = \phi_U \otimes  e^{\sigma t } dt    \quad \text{ and
} \quad \tilde {D}_V= \phi_V \otimes  e^{-\sigma t } dt
 $$
 (the parameter $t\in \R$ denotes the geodesic arc-length) for some distributions $\phi_U$,  $\phi_V \in \mathcal D'(S^1)$. This decomposition follows from the fact that the Lebesgue measure  is the only translation invariant distribution on $\R$ up to constant factors. It can be checked that  $\phi_U$ and $\phi_V\in \,^\Gamma\mathcal D'_\sigma (S^1)$ since $\tilde{D}_U$ and $\tilde{D}_V$ are $\Gamma$-invariant.

 Conversely, given $\phi\in \,^\Gamma \mathcal D'_\sigma (S^1)$ one can check that
 $\phi \otimes  e^{\sigma t } dt$ is a $\Gamma$-invariant distribution on $KA \equiv S^1 \times \R$,
 hence a  $\Gamma$-invariant distribution on the space of stable [unstable] horocycles,
 $PSL(2,\R)/h^U $ [$PSL(2,\R)\slash h^V $] (which can both be identified to $KA$).
 It follows that the distribution
 \begin{equation}
 \label{eq:conicdist}
 \phi \otimes e^{\sigma t } dt \otimes dh^U \quad [\phi \otimes e^{-\sigma t } dt \otimes dh^V] \,,
 \end{equation}
appropriately defined on $PSL(2,\R)$, is $\Gamma$-invariant,
$h^U$-invariant [$h^V$-invariant]  and projects to a $h^U$-invariant
[$h^V$-invariant] distribution $D^\phi_U$ [$D^\phi_V$ ] on the
manifold $\Gamma \backslash PSL(2,\R)$. By Fubini theorem for
distributions, it follows that
$$
g_t ^\ast(D^\phi_U) = e^{(\sigma-1)t} D^\phi_U\,, \quad \text{ and }
\quad g_t^\ast (D^\phi_V) = e^{(1-\sigma)t} D^\phi_V\,,
$$
that is, $D^\phi_U$ [$D^\phi_V$]  is an  eigenvector for the
geodesic flow of eigenvalue $\sigma-1$ [$1-\sigma$].

\smallskip
Cosentino (see \cite{Co}, Prop. 1.2) proves the following result:

\begin{lemma}
\label{lemma:invdistconfdist} If $\sigma(1-\sigma) \in \R^+
\setminus \{1/4\}$ is a Casimir parameter, the map
$$
\phi \to D^\phi_U:=  \phi \otimes e^{\sigma t } dt \otimes dh^U
\quad [ \phi \to D^\phi_V:=  \phi \otimes e^{-\sigma
t } dt \otimes dh^V]
$$
defines  a linear isomorphism of the space $^\Gamma \mathcal D'_\sigma
(S^1)$ of $\Gamma$-invariant conformal distributions of exponent
$\sigma\in \C$ onto the spaces $\mathcal I^\sigma_U$ [$\mathcal
I^\sigma_V$] of $U$-invariant [$V$-invariant]
distributions, which are eigenvectors of eigenvalue  $\sigma-1$
[$1-\sigma$] with respect to the action of the geodesic flow.
\end{lemma}

Cosentino then derives a regularity result for invariant
distribution from a theorem of J.~P.~Otal on the Poisson-Helgason
transform.

\begin{theorem}(\cite{Co}, Cor. 1.4) If $\sigma(1-\sigma) \in \R^+\setminus\{1/4\}$ is a Casimir
parameter, the space $^\Gamma \mathcal D'_\sigma (S^1) \subset
C^{\re(\sigma)-1}(S^1)$, hence $\mathcal I^\sigma \subset
C^{\re(\sigma)-1}(\Gamma\backslash PSL(2,\R))$.
\end{theorem}

(Let $C^\alpha(S^1)$ be the space of all H\"older functions of
exponent $\alpha \in (0, 1)$. Let $C^{\alpha-1}(S^1)$ denote the
space of all distributions in $\mathcal D'(S^1)$ which are locally
derivatives of functions in $C^{\alpha}(S^1)$).

\smallskip
In the exceptional case $\mu=1/4$, the above construction yields a
$1$-dimensional subspace of the $2$-dimensional space of
$U$-invariant [$V$-invariant] distribution, that is, the subspace of
distributional eigenvectors for the geodesic flow. Cosentino
\cite{Co} proves that a second independent distribution can be
constructed as follows.

\begin{lemma} For any $\Gamma$-invariant conformal
distribution $\phi \in ^\Gamma \mathcal D'_{1/2} (S^1)$, there
exists a distribution $\phi'  \in C^{\alpha-1}(S^1)$ for any $\alpha
<1/2$ such that the distribution
\begin{equation}
\begin{aligned}
D^{\phi,-}_U &:= \left( \phi'  \otimes e^{\frac{t}{2}} dt  - \phi
\otimes \frac{t}{2}  e^{\frac{t}{2}} \right)
\otimes dh_U\,,  \\
[D^{\phi, +}_V &:= \left( \phi'  \otimes  e^{-\frac{t}{2}}  dt  + \phi
\otimes \frac{t}{2} e^{-\frac{t}{2}}  \right)  \otimes dh_V] \,.
\end{aligned}
\end{equation}
is $\Gamma$-invariant and $U$-invariant [$V$-invariant] on
$PSL(2,\R)$, hence it projects to a $U$-invariant [$V$-invariant]
distribution on the manifold $\Gamma \backslash PSL(2,\R)$.
\end{lemma}

The distribution $\phi' $ is constructed in \cite{Co} as the inverse
Poisson-Helgason transform (the boundary value) of the function on
the Poincar\'e disk $D$ given by the pairing of the distribution
$\phi \in ^\Gamma \mathcal D'_{1/2} (S^1) $ with the function on
$S^1$ defined as $P(z, \cdot)^{1/2} \log P(z, \cdot)$ in terms of
the Poisson kernel $P$ on $D\times S^1$.

\medskip
Let  $\phi\in \,^\Gamma \mathcal D'_{1/2} (S^1)$ be a
$\Gamma$-invariant conformal distribution of exponent $1/2$ and let
$D^{\phi,+}_U \in \mathcal I^{1/2}_U$ [ $D^{\phi,-}_V\in \mathcal
I^{1/2}_V$] be  the $U$-invariant [$V$-invariant] distribution
defined as in formula (\ref{eq:conicdist}).  A direct calculations
shows that formula~(\ref{eq:Jordanb}) holds for  the distributional
vector $(D^+_{1/4}, D^-_{1/4}):= (D^{\phi,+}_U, D^{\phi,-}_U)$
[formula~(\ref{eq:Jordanbbis}) holds for the distributional vector
$(\hat D^+_{1/4}, \hat D^-_{1/4}):= (D^{\phi,+}_V, D^{\phi,-}_V)$].

\medskip
\subsection{Proof of the correspondence (Theorem \ref{thm:gammaconformal}).} By Theorem
\ref{thm:duality} there exist $V$-invariant distributions $\hat
D^\pm_\mu$, which are eigenvectors of eigenvalues $-(1\pm \nu)/2$
for the geodesic flow, such that the following identities hold on
$SM$, hence on $PSL(2,\R)$:
$$
dt \wedge \hat \beta^\pm_\mu \wedge dh^V  =  \hat D^\pm_\mu \,.
$$
(We remark that  the distributions $\hat D^\pm_\mu$ are here
identified with currents of degree $3$ and dimension $0$, acting on
functions, and the $1$-dimensional finitely-additive measures $\beta^\pm_\mu$ to
currents of degree $1$ and dimension $2$, acting on $2$-forms).

By Lemma \ref{lemma:invdistconfdist}, it follows that, for $\mu\neq 1/4$,
there exist $\Gamma$-invariant conformal distributions $\phi^\pm_\mu$
of exponent $\sigma^\pm:= 1-(1\pm \nu)/2= (1\mp \nu)/2$ such that
$$
\hat D^\pm_\mu= dt \otimes \hat \beta^\pm_\mu \otimes dh^V=dt \wedge \hat
\beta^\pm_\mu \wedge dh^V = \phi^\pm_\mu \otimes e^{- \frac{1\mp
\nu}{2} t } dt \otimes dh^V\,;
$$
for $\mu=1/4$ there exist a $\Gamma$-invariant conformal distribution
$\phi$ of exponent $1/2$ and a distribution $\phi' \in C^{\alpha-1}(S^1)$,
for all $\alpha<1/2$,  such that
\begin{equation*}
\begin{aligned}
\hat D^+_{1/4}=dt \otimes \hat \beta^+_{1/4} \otimes dh^V&= \phi' \otimes e^{-\frac{t}{2}} dt \otimes dh^V
+  \phi \otimes \frac{t}{2} e^{- \frac{t}{2} } dt \otimes dh^V\,, \\
\hat D^-_{1/4}=dt \otimes \hat \beta^-_{1/4} \otimes dh^V&=
\phi \otimes  e^{- \frac{t}{2} } dt \otimes dh^V\,.
\end{aligned}
\end{equation*}

The statement of the theorem follows immediately.

\section{Proofs of Technical Lemmas.}
\label{sec:appendix}
\subsection{Outline of the section.}
In \S\S~\ref{keyestimates} we prove
the key estimate on coboundaries stated in Lemma~\ref{lemma:remainderextension}. In
\S\S~\ref{innerproducts} we compute the $L^2$  inner products of the additive
cocycles coming from a single irreducible components of the principal series and establish a
non-vanishing result (Lemma~\ref{lemma:innerproducts}). In \S\S~\ref{rotsym}  we prove a couple of  results on rotationally invariant measures (up to affine transformations) on complex Euclidean spaces (Lemma~\ref{lemma:genconstdisthighD} and Lemma~\ref{lemma:uniqueness} ). These results hold
for the probability distributions of additive cocycles coming from irreducible components of the principal series by the above-mentioned non-vanishing result (Lemma~\ref{lemma:innerproducts} of \S\S~\ref{innerproducts}) and are motivated by the conditional results on the existence of limit distributions (Corollary~\ref{cor:limitdistprincipal1} and Corollary~\ref{cor:limitdistprincipal2} in Section \ref{sec:Cocycles}).

\subsection{Estimates on coboundaries.}
\label{keyestimates}

We prove below the key Lemma~\ref{lemma:remainderextension} which is part of the proof
of Theorem~\ref{thm:hatbeta} in Section~\ref{sec:BasicCurrents}. The subspace  $\Ker ({\mathcal  I}^{-s}_U(H_\mu))$ is closed in $W^s(H_\mu)$ and we introduce the orthogonal projection
$$
P_\mu^{s}: W^s(H_\mu)\to \Ker({\mathcal  I}^{-s}_U(H_\mu))\,.
$$

As above let $\nu=\sqrt{ 1-4\mu}$. Let $s_{\mu}:= (1+\vert \re
{\nu}\vert)/2$, that is, $s_{\mu}=1/2(1+\nu)$ for $0<\mu<1/4$;
$s_{\mu}=1/2$ for $\mu\geq 1/4$ and $s_{\mu}=n$ for $\mu=-n^2+n$.

\begin{lemma}
\label{lemma:projbound} For any non-trivial irreducible unitary representation of
Casimir parameter $\mu \in \R\setminus\{0\}$ and for any $s\geq r >s_{\mu}$ there exists a
constant $C_{r,s}(\mu)>0$ such that the following holds. For any $f\in W^s(H_\mu)$ we have
$$
\Vert P_\mu^{s} f \Vert_r \leq C_{r,s}(\mu)\Vert f\Vert_r.
$$
For any $\mu_0>1/4$ and for all $s\geq r >1/2$, there exists a constant $C_{r,s}(\mu_0)>0$
such that,  for all $\mu\geq \mu_0$ and for all $f\in W^s(H_\mu)$,
$$
\Vert P_\mu^{s} f \Vert_r \leq C_{r,s}(\mu_0) \Vert f\Vert_r.
$$
\end{lemma}

\begin{proof} Consider first the case of the principal and the
complementary series. Let $\chi_\mu^{\pm}(s)\in ( \Ker({\mathcal
I}^{-s}_U(H_\mu)) ^s)^\perp \subset W^s(H_\mu)$ be functions such
that
\begin{equation}
\label{eq:dualbasis}
\begin{aligned}
D^+_\mu(\chi^+_\mu(s))&=D^-_\mu(\chi^-_\mu(s))=1\,, \\
D^+_\mu(\chi^-_\mu(s))&=D^-_\mu(\chi^+_\mu(s))=0 \,.
\end{aligned}
\end{equation}

Such functions exist since $\{D^+_\mu, D^-_\mu\}$ is a basis for
$\mathcal I^{-s}_U (H_\mu)$ and, by definition of the space $
\Ker({\mathcal  I}^{-s}_U(H_\mu)) $, the functionals $D^\pm_\mu$
induce linearly independent functionals on the $2$-dimensional
quotient space $W^s(H_\mu)/ \Ker({\mathcal  I}^{-s}_U(H_\mu)) $. In
fact, take a pair of functions $\{\chi^+, \chi^-\} \subset
W^s(H_\mu)$ which project onto a dual basis of $\{D^+_\mu,
D^-_\mu\}$ under the projection $W^s(H_\mu) \to
W^s(H_\mu)/\Ker({\mathcal  I}^{-s}_U(H_\mu)) $ and define
$\chi_\mu^{\pm}(s)$ as the orthogonal projections onto
$\Ker({\mathcal  I}^{-s}_U(H_\mu))^\perp \subset W^s(H_\mu)$ of
$\chi^\pm$ respectively. By construction we have
\begin{equation}
\label{pif} P_\mu^s
f=f-D_\mu^+(f)\chi_\mu^+(s)-D_\mu^-(f)\chi_\mu^-(s).
\end{equation}

Indeed, by formula (\ref{eq:dualbasis}) the right hand side of the
formula  clearly belongs to the kernel of both $D_\mu^{\pm}$ and,
besides, if $P_\mu^sf$ is defined by (\ref{pif}), then we clearly
have $f-P_\mu^s f \perp  \Ker({\mathcal  I}^{-s}_U(H_\mu)) $. It
follows that for any $s_\mu < r \leq s$ and for any function $f\in
W^s(H_\mu)$ the following bound holds:
$$
\Vert P_\mu^s f \Vert_r \leq \Vert f \Vert_r + \Vert
\chi^+_\mu(s)\Vert _r  \Vert D^+_\mu \Vert_{-r} \Vert f \Vert_r
+\Vert \chi^-_\mu(s)\Vert_r \Vert D^-_\mu \Vert _{-r}\Vert f
\Vert_r,
$$
and, since $\Vert \chi^{\pm}_\mu(s) \Vert _r<+\infty$ for any $
r\leq s $ and $D^{\pm}_\mu \in W^{-r}(H_\mu)$ for $r>s_{\mu}$
(\cite{FlaFo}, Theorem 1.1 or Theorem 3.2), the proof is complete
for the principal and the complementary series.

The proof for the discrete series is similar, in fact simpler since
in each irreducible component of the discrete series there is only one invariant distribution
${\mathcal  D}_\mu$ (up to constant factors), and we only need one smooth function $\chi_\mu(s)$
such that $D_\mu(\chi_\mu(s))=1$ and $\chi_\mu(s)\perp  \Ker({\mathcal
I}^{-s}_U(H_\mu)) $. We then write $P_\mu^s
f=f-D_\mu(f)\chi_\mu(s)$, and the rest of the proof is identical.

It remains to be proved that the family of projection operators $P^s_\mu$ extends
to a  uniformly bounded family of operators on $W^r(H_\mu)$ for $\mu \geq \mu_0 >1/4$.

As observed in the proof of Lemma 5.1 in \cite{FlaFo}, if
the Casimir parameter $\mu\geq \mu_0 >1/4$, for any $r>1$, the
{\it distorsion} in $W^{-r}(SM)$ of the system of distributions
$\{D^+_\mu, D^-_\mu\}$ stays bounded above (in other terms, the
angle in $W^{-r}(SM)$ between $D^+_\mu$ and $D^-_\mu$ stays bounded
below), and in fact this bound is uniform with respect to $r>1$.
Hence there exists a constant $C_{r,s}(\mu_0)>0$ such that, for all $r>1$,
$$
\Vert \chi^\pm _\mu (s)\Vert _r   \leq   C_{r,s}(\mu_0) / \Vert D^\pm_\mu
\Vert_{-r} \,.
$$
The above argument  then yields the uniform bound on the norm of the operators
$P_\mu^{s}: W^r(H_\mu) \to  \Ker({\mathcal I}^{-r}_U(H_\mu)) $ for $\mu\geq \mu_0>1/4$.
\end{proof}
\medskip

Let us recall that we have defined $\mathcal R^{-s}: \Omega_1^{-s}(SM)
\to \mathcal B^{-s}_U(SM)^\perp$ as the orthogonal projection onto
the orthogonal complement of the subspace of basic currents for the
stable horocycle. For any irreducible component $H_\mu \subset
L^2(SM)$, of Casimir parameter $\mu\in \R\setminus\{0\}$, we have defined
$\Pi^{-s}_\mu: \Omega_1^{-s}(SM) \to \Omega_1^{-s} (H_\mu)$ as the
orthogonal projection onto the corresponding Sobolev space. We have
then defined $\mathcal R^{-s}_\mu = \Pi^{-s}_\mu \circ \mathcal
R^{-s} : \Omega_1^{-s}(SM) \to \mathcal B^{-s}_U (H_\mu)^\perp$ as the
orthogonal projection onto the orthogonal complement of the subspace
of basic currents in $\Omega_1^{-s}(H_\mu)$.

\smallskip

\begin{proof}[Proof of Lemma \ref{lemma:remainderextension}]
The argument is similar to the proof Lemma 5.5 in \cite{FlaFo}.

By Hilbert space theory, any current $\mathcal R^{-s}\in {\mathcal
B}^{-s}_U(SM)^\perp \subset
 \Omega_1^{-s}(SM)$ (orthogonal to the subspace of basic currents) has the following property:
\begin{equation}
\label{eq:Rzero} \mathcal R^{-s} (\lambda) =0 \,,\quad \text{ for
any } \lambda \in [\Ker({\mathcal  B}^{-s}_U(H_\mu))] ^\perp\,.
\end{equation}
In fact, the Hilbert  space $\Omega_1^{-s}(H_\mu)$ is defined as the
dual space $\Omega_1^s(H_\mu)^\ast$, which in turn is isomorphic to
$\Omega_1^s(H_\mu)$. Thus we have
$$
  \Ker ( S ^{\perp})=[\Ker (S)] ^{\perp} \,, \quad \text{ \rm for any subspace }
S \subset \Omega_1^{-s}(H_\mu)\,.
$$
It follows that $\Ker [ \mathcal B^{-s}_U(H_\mu)^\perp] = [\Ker (
\mathcal B^{-s}_U(H_\mu))]^\perp\,, $ hence, in particular,
$$
{\mathcal  R}^{-s} \in \mathcal B^{-s}_U(H_\mu)^\perp = \Ker [
\Ker({\mathcal  B}^{-s}_U(H_\mu)) ]^\perp\,.
$$
By the characterization of basic currents for the stable horocycle
flow given by Lemma \ref{lemma:Isom}, the kernel $\Ker({\mathcal
B}^{-s}_U(H_\mu))$ can be described as follows:
\begin{equation}
 \Ker({\mathcal  B}^{-s}_U(H_\mu))  :=\{ \lambda= \lambda_X \hat X + \lambda_U \hat U +
\lambda_V \hat V \vert \lambda_U \in  \Ker({\mathcal
I}^{-s}_U(H_\mu)) \} \,.
\end{equation}
By the definition of the Hilbert structure of the space
$\Omega_1^s(SM)\equiv W^s(SM)^3$, it follows that the orthogonal
projection $\Pi^s_\mu : \Omega_1^s(H_\mu) \to \Ker({\mathcal
B}^{-s}_U (H_\mu))$ can be written in terms of the orthogonal
projection $P^s_\mu : W^s(H_\mu) \to \Ker({\mathcal
I}^{-s}_U(H_\mu))$:
\begin{equation}
\Pi^s_\mu( \lambda)=\Pi^s_\mu( \lambda_X \hat X + \lambda_U \hat U +
\lambda_V \hat V)= \lambda_X \hat X + P^s_\mu( \lambda_U) \hat U +
\lambda_V \hat V \,.
\end{equation}

For any $\lambda \in \Omega_1^s(H_\mu)$, since by definition $\mathcal
B^{-s}(\gamma)\circ \Pi^s_\mu=0$, by definition of $\mathcal
R^{-s}_\mu(\gamma)$ and by the vanishing established in formula
(\ref{eq:Rzero}),
\begin{equation}
\label{eq:Ridentity} \mathcal R^{-s}_\mu(\gamma)(\lambda) = (\mathcal
R^{-s}(\gamma) \circ \Pi^s_\mu) (\lambda) = \gamma( \Pi^s_\mu
(\lambda)) \,.
\end{equation}

By \cite{FlaFo}, Theorem 4.1, since $P_\mu^s (\lambda_U) \in
\Ker({\mathcal  I}^{-s}_U(H_\mu))$, there exists a unique solution
$f_\lambda \in W^r(H_\mu)$ (for all $r<s-1$) of the cohomological
equation
$$
U f_\lambda =P_\mu^s (\lambda_U).
$$
Moreover, for any $s_\mu< r \leq s$ and any $\rho < r-1$, we have
$f_\lambda \in W^{\rho}(H_\mu)$, and there exists a constant
$C_{\rho,r}$,  depending only on $ \rho, r$ such that  we have
\begin{equation}
\label{eq:CEest} \Vert f_\lambda \Vert _{\rho}\leq C_{\rho,r} \Vert
P_\mu^s( \lambda_U) \Vert _r \,.
\end{equation}
Thus by Lemma  \ref{lemma:projbound} we can conclude that  for all
Casimir parameters $\mu\not =0$,  for any $s_\mu< r \leq s$ and any
$\rho < r-1$,  there exists $C_{s,\rho,r}(\mu)>0$ such that
\begin{equation}
\label{eq:keyest} \Vert f_\lambda \Vert _{\rho}\leq
C_{s,\rho,r}(\mu)
 \Vert  \lambda_U \Vert _r \leq  C_{s,\rho,r}(\mu)  \Vert \lambda \Vert _r
\end{equation}
and that for any $\mu_0>1/4$ and for $1/2 < r \leq s$ and any
$\rho < r-1$  there exists a constant  $C_{s,\rho,r}(\mu_0)>0$ such that, for all
$\mu\geq \mu_0$,
\begin{equation}
\label{eq:keyestunif} \Vert f_\lambda \Vert _{\rho}\leq
C_{s,\rho,r}
 \Vert  \lambda_U \Vert _r \leq  C_{s,\rho,r} (\mu_0) \Vert \lambda \Vert _r\,.
\end{equation}

Let $x$, $y\in SM$ be the endpoints of the arc $\gamma$. By the
formula
$$
df_\lambda = Xf_\lambda \hat X + Uf_\lambda \hat U + V f_\lambda
\hat V \,,
$$
it follows that the following identity holds:
\begin{equation}
\label{eq:gammaprojid} \gamma( \Pi^s_\mu (\lambda)) = \gamma [
(\lambda_X-Xf_\lambda) \hat X +
  (\lambda_V-Vf_\lambda) \hat V] + f_\lambda(y)-f_\lambda(x)\,.
\end{equation}
The above identity yields the following estimate:
$$
\vert \gamma( \Pi^s_\mu (\lambda)) \vert  \leq ( \Vert \lambda_X
\Vert_{\infty} + \Vert \lambda_V \Vert_{\infty}  + \Vert df_\lambda
\Vert_\infty) \left( \int_\gamma \vert \hat X \vert +  \int_\gamma
\vert \hat V \vert\right) + 2 \Vert f_\lambda\Vert_\infty\,.
$$
By the Sobolev Embedding Theorem, for any $\rho \in (5/2,s-1)$,
there exists a positive constant $C_{\rho}(M)$ depending only on $M$
such that
$$
\Vert f_\lambda \Vert _\infty+  \Vert d f_\lambda \Vert _\infty \leq
C_{\rho}(M) \Vert f_\lambda \Vert_{\rho}.
$$
We thus obtain that for any $s\geq r >7/2$ and for all Casimir parameters $\mu>0$
 there esists a constant $C_{r,s}(\mu)>0$ such that,  for all $\lambda\in \Omega_1^s(H_\mu)$,
$$
\vert {\mathcal  R}_\mu^{-s}(\gamma)(\lambda) \vert =\vert \gamma(
\Pi^s_\mu (\lambda)) \vert
 \leq C_{r,s}(\mu) \left(1 + \int_\gamma \vert \hat X \vert +  \int_\gamma \vert \hat V \vert\right)
\Vert \lambda \Vert_r  \,;
$$
for any $s\geq r >7/2$ and for any $\mu_0>1/4$, there exists a constant $C_{r,s}(\mu_0)>0$ such that,
for all $\mu \geq \mu_0$ and for all $\lambda \in \Omega_1^s(SM)$,
$$
\vert {\mathcal  R}_\mu^{-s}(\gamma)(\lambda) \vert =\vert \gamma(
\Pi^s_\mu (\lambda)) \vert
 \leq C_{r,s}(\mu_0) \left(1 + \int_\gamma \vert \hat X \vert +  \int_\gamma \vert \hat V \vert\right)
\Vert \lambda \Vert_r  \,;
$$

Since $\Omega_1^s(H_\mu)$ is dense in $\Omega_1^r(H_\mu)$  for $s\geq
r$, it follows that the distribution ${\mathcal  R}_\mu^{-s}(\gamma)
\in \Omega_1^{-s}(H_\mu)$ has a unique continuous extension ${\mathcal
R}_\mu^{-s,-r}(\gamma)$ to the space $\Omega_1^r(H_\mu)$ such that
$$
\Vert {\mathcal  R}_\mu^{-s,-r}(\gamma)(\lambda) \Vert_{-r}  \leq
C_{r,s}(\mu) \left(1 + \int_\gamma \vert \hat X \vert +  \int_\gamma
\vert \hat V \vert \right)\,;
$$
for all $\mu_0 >1/4$, there exists a constant $C_{r,s}(\mu_0)>0$ such that for
all $\mu \geq \mu_0$,
\begin{equation}
\label{eq:Runifboundone}
\Vert {\mathcal  R}_\mu^{-s,-r}(\gamma)(\lambda) \Vert_{-r}  \leq
C_{r,s}(\mu_0) \left(1 + \int_\gamma \vert \hat X \vert +  \int_\gamma
\vert \hat V \vert \right)\,.
\end{equation}

Since $M$ is a compact hyperbolic surface, the Casimir operator  of the standard
unitary representation of $SL(2,\R)$ on $L^2(SM)$ has discrete spectrum, hence
 for any $s\geq r>7/2$ there is a constant $C_{r,s}>0$ (depending  only on $r,s$)
 such that,  for all Casimir parameters $\mu>0$ and for every rectifiable arc $\gamma$ in $SM$,
 we have
\begin{equation}
\label{eq:Runifboundtwo}
\Vert  {\mathcal  R}^{-s,-r}_\mu(\gamma) \Vert _{-r} \leq C_{r,s} \left(
1 + \int_\gamma \vert \hat X\vert+ \int_\gamma \vert \hat
V\vert\right)\,.
\end{equation}
Thus Lemma \ref{lemma:remainderextension} is completely proved.
\end{proof}

\subsection{Inner products of cocycles.}
\label{innerproducts}

Cocycles which belong to different irreducible components are orthogonal. We compute
below the inner product of the two complex-conjugate coycles which belong to a given
component principal series type and prove that it is non-zero.

\begin{lemma}
\label{lemma:innerproducts}
For any $\mu>1/4$ and any $T\in \R\setminus\{0\}$, the $L^2$ inner product
$$
<\beta^+_\mu(\cdot,T), \beta^-_\mu(\cdot,T)> \not = 0\,.
$$
\end{lemma}
\begin{proof} The  computation proceeds as follows. Let $\{f^+,f^-\}\in C^\infty(H_\mu)$ be a pair of functions dual to the basis $\{D^+_\mu, D^-_\mu\}$ of invariant distributions, that is,
$$
D^+_\mu( f^+) = D^-_\mu(f^-) \not=0  \,,\quad  D^+_\mu( f^-) = D^-_\mu(f^+) =0\,.
$$
It follows from Corollary~\ref{cor:mainprincip}  that
\begin{equation}
\begin{aligned}
e^{-(1-\nu)t}  &< \int_0^{Te^t}  f^+ \circ h^U_\sigma d\sigma,  \int_0^{Te^t}
 f^- \circ h^U_\tau d\tau>   \\ -  &D^+_\mu( f^+) \overline{D^-_\mu(f^-)}
 <\beta^+_\mu(\cdot,T), \beta^-_\mu(\cdot,T)> =O (\exp(-\epsilon_\mu t))\,.
\end{aligned}
\end{equation}
Our goal is therefore to compute the asymptotics of the normalized inner product
$$
 \mathcal E_{T,t}(f^+,f^-) := e^{-(1-\nu)t}
 \int_0^{Te^t}   \int_0^{Te^t}    <f^+ \circ h^U_\sigma, f^- \circ h^U_\tau>  d\sigma d\tau\,.
$$
The above computation can be performed explicitly in the standard model for representations
of the principal series on the Hilbert space $L^2(\R, dx)$.

In this model, the horocycle flow is represented by the group of translations, hence its infinitesimal
generator is the operator $d/dx$. In general, the action of $SL(2,\R)$ on $L^2(\R,dx)$ for the irreducible
representation $\pi_\mu$ of the principal series is given by the formula:
$$
\pi_\mu[\begin{pmatrix} a & b \\ c  & d \end{pmatrix}]  f (x) :=  \vert cx+d\vert^{-(1+\nu)}
f( \frac{ax+b}{cx+d}) \,.
$$
Hence the derived representation $d\pi_\mu$ of the Lie algebra $\mathfrak{sl}(2,\R)$
of $SL(2,\R)$ is described by the following formulas:
\begin{equation}
\begin{aligned}
d\pi_\mu (U) &= \frac{d}{dx} \,, \\  d\pi_\mu (V) &=-x^2 \frac{d}{dx} - (1+\nu)x \,, \\
d\pi_\mu (X)&=  x \frac{d}{dx} + \frac{1+\nu}{2}\,.
\end{aligned}
\end{equation}
>From the above formulas one can also deduce the formula for the representation
of the generator $\Theta$ of the circle action on the unit tangent bundle. In fact,
since $\Theta= U-V$, we have
$$
d\pi_\mu (\Theta) =  (1+x^2) \frac{d}{dx} + (1+\nu)x \,.
$$
A calculus exercise shows that, up to normalization, the unique $\Theta$-invariant
function $u_0 \in C^\infty(H_\mu)$ is given in the representation model by the formula:
$$
u_0 (x) :=   \frac{1}{ (1+x^2)^{  \frac{1+\nu}{2} }}    \,, \quad x\in \R\,.
$$
It follows from the construction of the $U$-invariant distributions in \cite{FlaFo}, \S 3,
that the following holds:
$$
D^+_\mu ( u_0) \not =0\,.
$$
Given any function $f^-$ such that $D^+_\mu(f^-)=0$ and $D^-_\mu(f^-)\not=0$,
we can therefore choose
$$
f^+ := u_0 -  \frac{D^-_\mu(u_0)}{D^-_\mu(f^-)} f^- \,.
$$

We will choose $f^-$ to be represented by any function in $C_0^\infty(\R) \subset L^2(\R,dx)$
with non-zero integral over the real line. In order to justify this choice we remark that the
$U$-invariant distribution $D^-_\mu$ is an extension to $C^\infty(H_\mu)$ of the distribution
$\mathcal A \in {\mathcal D}' (\R)$ given by the average over the real line, while the $U$-invariant
distribution $D^+_\mu$ vanishes identically on the (non-dense!) subspace $C^\infty_0(\R) \subset
C^\infty(H_\mu)$. In fact, since $D^\pm_\mu$ are $U$-invariant and the average
$\mathcal A$ is $d\pi_\mu(U)$-invariant, there exists constants $c^\pm \in \R$ such that
$$
D^\pm_\mu = c^\pm \mathcal A   \quad \text{ \rm on } \, C_0^\infty(\R) \subset C^\infty(H_\mu)\,.
$$
However,   the following formulas hold in ${\mathcal D}' (\R)$:
$$
  d\pi_\mu (X) \mathcal A=  -\frac{1-\nu}{2} \mathcal A \,.
$$
It follows that  $c^+=0$, since $d\pi_\mu (X) D^+_\mu = -\frac{1+\nu}{2} D^+_\mu$,
and that $c^-\not=0$, that is, $D^-_\mu$ is, up to a non-zero multiplicative constant,
an extension of the distribution $\mathcal A$ to $C^\infty(H_\mu)$. Hence for any function
$f^-$ represented by a smooth real-valued function with compact support and non-zero integral,
$$
D^+_\mu (f^-) =0 \quad \text{ \rm and} \quad D^-_\mu(f^-) =  c^-\mathcal A(f^-) \not=0\,.
$$

We remark that by Corollary~\ref{cor:mainprincip} the normalized inner product
$$
 \mathcal E_{T,t}(f^-,f^-) := e^{-(1-\nu)t}  \Vert  \int_0^{Te^t}  f^- \circ h^U_\sigma d\sigma \Vert^2
$$
is given asymptotically by the following formula:
\begin{equation}
\label{eq:ff}
 \mathcal E_{T,t}(f^-,f^-)   - e^{\nu t} \vert D^-_\mu(f^-)\vert^2
 \Vert \beta^-_\mu(\cdot,T)\Vert^2  =O (\exp(-\epsilon_\mu t))\,.
\end{equation}

Our task is therefore reduced to estimate the integral
$$
 \int_0^{Te^t}   \int_0^{Te^t}   \int_\R  \frac{f^-(x+ \tau)}{ (1+(x+\sigma)^2)^{ \frac{1+\nu}{2}}}  dx
 d\sigma d\tau\,.
$$
By Fubini's and change of variables theorems, it is enough to compute the integral
\begin{equation}
\label{eq:ITx}
I_T(x) :=  \int_0^T \int_0^T   \frac{1}{(1+(x+\sigma-\tau)^2)^{ \frac{1+\nu}{2}}}  dx   \,.
\end{equation}
By explicit integration we obtain that as $T\to +\infty$ the function
$$
 T ^{-(1-\nu)} I_T(x) -  T^\nu  \int_{-T}^T \frac{du}{ (1+u^2)^{\frac{1+\nu}{2}} }- \frac{2}{1-\nu}   \to 0
$$
uniformly on compact intervals. By integration by parts the integral
$$
J_T := \int_{-T}^T \frac{du}{ (1+u^2)^{\frac{1+\nu}{2}}}
$$
satisfies the following formula:
$$
J_T=  \frac{2T}{(1+T^2)^{ \frac{1+\nu}{2}}} +
(1+\nu) J_T - (1+\nu) \int_{-T}^T \frac{du}{ (1+u^2)^{\frac{3+\nu}{2}}}\,.
$$
It follows immediately that
$$
J_T =   -\frac{1}{\nu}   \frac{2T}{(1+T^2)^{ \frac{1+\nu}{2}}}  + \frac{1+\nu}{\nu}
 \int_{-T}^T \frac{du}{ (1+u^2)^{\frac{3+\nu}{2}}}\,.
$$
Since the convergence in formula (\ref{eq:ITx}) is uniform on compact sets and
the improper integral
$$
I_\nu :=  \int_{-\infty}^\infty \frac{du}{ (1+u^2)^{\frac{3+\nu}{2}}} = \lim_{T\to +\infty}
 \int_{-T}^T \frac{du}{ (1+u^2)^{\frac{3+\nu}{2}}}
$$
is absolutely convergent, we obtain that the normalized integral
\begin{equation}
\mathcal E_{T,t}(u_0, f^-) := e^{-(1-\nu)t}  < \int_0^{Te^t}  u_0 \circ h^U_\sigma d\sigma,  \int_0^{Te^t}
 f^- \circ h^U_\tau d\tau>
\end{equation}
is given asymptotically (as $t\to +\infty$) by the following formula:
\begin{equation}
\label{eq:uf}
\mathcal E_{T,t}(u_0, f^-) + \frac{2(1-2\nu)}{\nu(1-\nu)} T^{1-\nu} D^-_\mu(f^-) -
 \frac{1+\nu}{\nu}  I_\nu T D^-_\mu(f^-) e^{\nu t} \to 0 \,.
 \end{equation}
 By definition we have
 $$
 \mathcal E_{T,t}(f^+, f^-) = \mathcal E_{T,t}(u_0, f^-) -  \frac{D^-_\mu(u_0)}{D^-_\mu(f^-)} \mathcal E_{T,t}(f^-, f^-)
 $$
Since the normalized inner product   $\mathcal E_{T,t}(f^+, f^-)$ converges, by formulas
(\ref{eq:ff}) and (\ref{eq:uf}) it follows that
\begin{equation}
\begin{aligned}
 \mathcal E_{T,t}(f^+, f^-)  &\to - \frac{2(1-2\nu)}{\nu(1-\nu)} T^{1-\nu} D^-_\mu(f^-)  \,; \\
 e^{-\nu t} \mathcal E_{T,t}(f^-, f^-)  & \to   \frac{1+\nu}{\nu}  I_\nu T
 \frac{D^-_\mu(f^-)^2}{D^-_\mu(u_0) }\,.
\end{aligned}
\end{equation}
We have thus proved the lemma and in addition we obtain the formulas:
\begin{equation}
\begin{aligned}
<\beta^+_\mu(\cdot,T),\beta^-_\mu(\cdot,T)>  &=
 - \frac{2(1-2\nu)}{\nu(1-\nu)} \frac{T^{1-\nu}}{D^-_\mu(u_0)}\,; \\
\Vert \beta^\pm_\mu(\cdot, T)\Vert^2 &=  \frac{1+\nu}{\nu}  \frac{I_\nu T}{D^-_\mu(u_0)} \,.
\end{aligned}
\end{equation}
\end{proof}

\subsection{Rotationally symmetric measures.}
\label{rotsym}

We prove below some elementary results about compactly supported measures on the complex plane
and on higher dimensional cartesian products of the complex plane. These results characterize
probability distributions with rotational symmetries up to an affine change of coordinates.

\begin{lemma}
\label{lemma:rotinv}
Let $\mu$ be a compactly supported Borel measure on the complex plane.
If for any $a<b \in \R$ there exists a constant $m(a,b)$ such that, for all $v \in \C\setminus\{0\}$,
$$
\mu \{  z \in \C \vert   a\leq  \frac{\re (z v)}{\vert v\vert}  <  b  \} = m(a,b)\,,
$$
then the measure $\mu$ on $\C$ is rotationally invariant.
\end{lemma}
\begin{proof}
It follows from the assumption that the probability distribution with respect to the measure $\mu$ on
$\C$ of the function
$$
\re (e^{i\theta} z) =   \frac{ e^{i\theta} z + e^{-i\theta} \bar z }{2}\,,
$$
does not depend  on $\theta \in \R$, hence for every $k\in \Z$ the integral
$$
\int_{\C}   (e^{i\theta} z + e^{-i\theta} \bar z)^k   d\mu\,,
$$
does not depend  on $\theta \in \R$ as well.  It follows that
$$
\int_{\C}  z^r {\bar z}^s d\mu = 0\, \quad \text{ \rm for all } \, r, s \in \N\,, \,\, r-s\not=0\,.
$$
Let $R_\theta$ be the rotation of angle $\theta\in \R$ on $\C$. For any real analytic
function $f$ on $\C$, we have by power series expansion
$$
\int_{\C} f \circ R_\theta\, d\mu =   \sum_{r \in \N}  \frac{1}{(r!)^2} \frac{\partial^2 f}{\partial z^r
\partial{\bar z}^r}(0)  \int_\C \vert z\vert^{2r}  = \int_{\C} f \, d\mu   \,.
$$
Since $\mu$ has compact support and polynomials are dense in the uniform topology on any
compact subset of the complex plane, the result follows.
\end{proof}

For any $A$, $B\in \R^+$ , let  $T_{A,B}: \C \to \C$ the affine map defined as follows:
$$
T_{A,B}(x,y) = (x/A, y/B) \,, \quad \text{ \rm for all }\, (x,y)\in \C\equiv\R^2 \,.
$$

\begin{lemma}
\label{lemma:genconstdist}
Let $\mu$ be a compactly supported Borel probability measure on the complex plane.
Assume that there exist $(A,B) \in (\R^+)^2$ such that for any $a$, $b \in \R$ there exists a constant
$m(a,b)$ such that, for all $\theta \in \R$,
$$
\mu \{  z \in \C \vert   a\leq  \frac{\re (e^{i\theta} z )}{ (A^2\cos^2\theta +B^2\sin^2\theta)^{1/2}}  <  b  \}
= m_{A,B}(a,b)\,.
$$
Then the measure $\mu$ has the following form: there exists a rotationally invariant measure
$\rho$ on $\C$ such that
$$
\mu =  (T_{A,B}^{-1})_\ast (\rho)\,.
$$
\end{lemma}
\begin{proof}
We claim that, by Lemma~\ref{lemma:rotinv}, the measure $\rho := (T_{A,B})_\ast \mu$
is rotationally invariant. In fact, for any $\theta\in \R$, let $v_{A,B}(\theta)\in \C$ be the unit complex
number defined as follows:
$$
v_{A,B}(\theta):=  \frac{(A\cos\theta, B\sin\theta)}{(A^2\sin\theta +B^2\cos^2\theta)^{1/2}}\,.
$$
A straightforward calculation yields that, for any $a<b$ and for all $\theta\in \R$,
\begin{equation}
\begin{aligned}
&\rho \{  z \in \C \vert  a\leq  \frac{\re [z v_{A,B}(\theta)]}{\vert v_{A,B}(\theta)\vert}  <  b  \}  \\
&=\mu \{  z \in \C \vert   a\leq  \frac{\re (e^{i\theta} z )}{ (A^2\cos^2\theta +B^2\sin^2\theta)^{1/2}}  <  b  \}  \,.
\end{aligned}
\end{equation}
Since the family $\{v_{A,B}(\theta) \vert \theta\in \R\}$ coincides with $\{ v\in \C\vert \vert v\vert=1\}$,
it follows that
$$
\rho \{  z \in \C \vert  a\leq  \frac{\re (z v)}{\vert v\vert}  <  b  \} = m_{A,B}(a,b),
$$
hence $\rho$ is rotationally invariant by Lemma~\ref{lemma:rotinv}, as claimed.
\end{proof}

\begin{lemma} Let $\beta: (X,\mu) \to \C$ be a bounded measurable function on the probability
space $(X,\mu)$. The family of  real-valued functions
$$
\left\{ \frac{ \re (e^{i\theta} \beta)}{\Vert \re (e^{i\theta} \beta) \Vert } \vert \theta \in \R\right\}
$$
has a constant probability distribution if and only if there exists an affine map $T:\C \to \C$
such that the function $T\circ \beta:X \to \C$ has a rotationally invariant probability distribution.
\end{lemma}
\begin{proof} Up to composition with a rotation of the complex plane it is possible to assume
that
$$
\int_ X (\re\beta) (\im \beta) d\mu = \frac{1}{2} \im (\int_X \beta^2 d\mu) =0 \,.
$$
Under that assumption, it follows that
$$
\Vert \re( e^{i\theta} \beta) \Vert^2 = \Vert \re \beta\Vert^2  \cos^2\theta  +
\Vert \im \beta\Vert^2 \sin^2\theta  \,.
$$
Thus the statement follows from Lemma~\ref{lemma:genconstdist}.
\end{proof}

The above results generalize to functions with values in higher dimensional complex
spaces. In fact, the following holds:

\begin{lemma}
\label{lemma:rotinvhighD}
Let $\mu$ be a compactly supported Borel probability measure on $\C^n$.
Assume that for any $a$, $b \in \R$ there exists a constant
$m(a,b)$ such that, for all $v\in \C^n\setminus\{0\}$,
$$
\mu \{  z \in \C \vert   a\leq  \frac{\re( z\cdot v)}{\vert v\vert }  <  b  \} = m(a,b)\,,
$$
then  $\mu$ is rotationally invariant, that is, it is invariant under the action of the
orthogonal group $SO(2n,\R)$ on $\C^n\equiv \R^{2n}$.
\end{lemma}
\begin{proof}
We claim that there exists a compactly supported probability measure $m$ on the real line such
that the constants $m(a,b) = m\{x\in \R\vert a\leq x<b\}$ and that measure $\mu$ is uniquely
determined by the measure $m$ on the real line. In fact, the measure $m$ is just the probability
distribution of the function $ \re (z\cdot v)/\vert v\vert$ for any given $v\in \C^n\setminus\{0\}$.
The uniqueness of the measure $\mu$ follows from the fact the computations of the moments
of the probability distribution $m$ on $\R$ yields the values of all integrals of the form
$$
\int_{\C^n}  z^\alpha \bar z^\beta \,d\mu
$$
(in terms of binomial coefficients and of the moments of the probability measure $m$).
Since $\mu$ has compact support and polynomials are dense in the space of continuous
functions on compact sets, the uniqueness follows.

Finally, the measure $\mu$ is rotationally invariant as it is the unique measure which
satisfies a rotationally invariant condition.

\end{proof}

For any $(A,B)\in (\R^+)^n$ , let  $T_{A,B}: \C^n \to \C^n$ the invertible affine map defined as follows:
for all $(x_1,y_1, \dots, x_n,y_n) \in \C^n\equiv\R^{2n}$,
$$
T_{A,B}(x_1,y_1, \dots, x_n,y_n) = (x_1/A_1,y_1/ B_1, \dots, x_n/A_n, y_n/B_n) \,.
$$

\begin{lemma}
\label{lemma:genconstdisthighD}
Let $\mu$ be a compactly supported Borel probability measure on $\C^n$.
Assume that there exists $(A, B)\in (\R^+)^n \times (\R^+)^n $ such that  for any $a$, $b \in \R$ there exists a constant $m_{A,B}(a,b)$ such that, for all $(r,\theta) \in (\R^+)^n \times \R^n$,
$$
\mu \{  z \in \C^n \vert   a\leq  \frac{\re (\sum_{s=1}^n r_se^{i\theta_s} z_s )}{ \left[ \sum_{s=1}^n r_s^2
(A_s^2\cos^2\theta_s+B_s^2\sin^2\theta_s) \right]^{1/2}}  <  b  \}
= m_{A,B}(a,b)\,.
$$
Then the measure $\mu$ has the following form: there exists a rotationally invariant measure
$\rho$ on $\C^n$ such that
$$
\mu =  (T_{A,B}^{-1})_\ast (\rho)\,.
$$
\end{lemma}

In fact, a stronger result holds. The key step is given by the following result.

\begin{lemma}
\label{lemma:uniqueness}

Let $\mu$ be a compactly supported Borel probability measure on $\C^n$. Let  $(A,B)\in (\R^+)^n \times(\R^+)^n$ be a pair of vectors such that
$$
A_1\not =B_1, \quad \dots, \quad A_n\not =B_n \,.
$$
If there exist distinct integral vector $v^{(1)}, \dots, v^{(n)} \in \Z^d$ ($d\geq 1$) such that  the probability distribution $m_{A,B,r,\theta}$ of the function
$$
  \frac{\re (\sum_{s=1}^n r_se^{i<v^{(s)},\theta>} z_s )}{ \left[ \sum_{s=1}^n r_s^2
(A_s^2\cos^2 <v^{(s)},\theta>+B_s^2\sin^2 <v^{(s)},\theta>) \right]^{1/2}} \,,
$$
defined on the probability space $(\C^n,\mu)$, is independent of $\theta\in \T^d$, for any given
$r\in (\R^+)^n$, then it is also independent of $r\in \R^n\setminus\{0\}$, hence there exists a probability distribution $m_{A,B}$ on the real line such that
$$
m_{A,B,r,\theta} = m_{A,B} \,, \quad \text{ for all } (r,\theta) \in \R^n\setminus\{0\} \times \T^d\,.
$$
\end{lemma}
\begin{proof}

By assumption there exists a compactly supported measure $m_{A,B,r}$ on the real line such that
$m_{A,B,r,\theta}=m_{A,B,r}$ for all $\theta\in \T^n$. A computation of the moments $\mathcal M^{(k)}_{A,B,r}$ of the measure $m_{A,B,r}$ yields that all the odd moments $\mathcal M^{(2k+1)}_{A,B,r}$ vanish while from the the computation of even moments $\mathcal M^{(2k)}_{A,B,r}$ we can derive  the identities below. Let  $P_{A,B,r}(\theta)$ be the trigonometric polynomial
$$
P_{A,B,r}(\theta) :=  \sum_{s=1}^n r_s^2
(A_s^2\cos^2 <v^{(s)},\theta>+B_s^2\sin^2 <v^{(s)},\theta>) \,,
$$
which after a simple calculation can be written as follows:
\begin{equation*}
P_{A,B,r}(\theta) =  \sum_{s=1}^n r_s^2\bigl( \frac{A_s^2-B_s^2}{4}(e^{2i <v^{(s)},\theta>} +
e^{-2i<v^{(s)},\theta>}) + \frac{A_s^2+B_s^2}{2}\bigr) \,.
\end{equation*}
The calculation of even moments yields:
\begin{equation}
\label{eq:evenmoments}
 \int_{\C^n} \re (\sum_{s=1}^n r_s e^{i<v^{(s)},\theta>} z_s )^{2k} = \mathcal M^{(2k)}_{A,B,r}
  P_{A,B,r}(\theta)^k\,.
\end{equation}
Since by assumption the set of integral vectors $\{v^{(1)}, \dots, v^{(n)}\}$ has distinct elements,
it has a unique maximal element $v^{(l)}$ with respect to the lexicographic order on $\Z^d$.
Thus by comparing the coefficients of the exponential  $\exp( 2 i k <v^{(l)}, \theta>)$ on the left
and right sides of  formula~(\ref{eq:evenmoments}), it follows that
$$
 r_l^{2k}  \int_{\C^n}    z_l^{2k}  d\mu =   \mathcal M^{(2k)}_{A,B,r}( \frac{A_l^2-B_l^2}{2})^k r_l^{2k}\,.
$$
By assumption $A_l\not=B_l$, hence on the set $\R^n\setminus \{r_l=0\}$  the moment
$\mathcal M^{(2k)}_{A,B,r}$ is given by the formula
$$
\mathcal M^{(2k)}_{A,B,r}= \frac{ 2^k \int_{\C^n}    z_l^{2k}  d\mu}{
(A_l^2-B_l^2)^k}\,.
$$
It follows that the function $\mathcal M^{(2k)}_{A,B,r}$ is constant on $\R^n\setminus \{r_l=0\}$.
Since, by formula~(\ref{eq:evenmoments}),  it is continuous  on $\R^n\setminus\{0\}$, it follows that
$\mathcal M^{(2k)}_{A,B,r}$ is equal to a constant $\mathcal M^{(2k)}_{A,B}$ on $\R^n\setminus\{0\}$.
Thus all the even moments $ \mathcal M^{(2k)}_{A,B,r}$ as well as all the odd moments  $\mathcal M^{(2k+1)}_{A,B,r}$ of the compactly supported probability measure $m_{A,B,r}$  do not depend on $r\in\R^n\setminus\{0\}$. It follows that there exists a probability distribution $m_{A,B}$ on the real line, with
zero odd moments and even moments equal to $\mathcal M^{(2k)}_{A,B}$,  such that $m_{A,B,r,\theta}= m_{A,B}$ for all $(r,\theta)\in (\R^n\setminus\{0\})\times\T^d$, as stated.
\end{proof}

By Lemma~\ref{lemma:genconstdisthighD} and Lemma~\ref{lemma:uniqueness}, we can then
derive the following characterization:

\begin{lemma}
\label{lemma:genconstdisthighDstrong}
Let $\mu$ be a compactly supported Borel probability measure on $\C^n$.
Assume that there exists $(A, B,r)\in (\R^+)^n \times (\R^+)^n \times  (\R^+)^n$ such that  for any $a$, $b \in \R$ there exists a constant $m_{A,B,r}(a,b)$ such that, for all $\theta \in \R^n$,
$$
\mu \{  z \in \C^n \vert   a\leq  \frac{\re (\sum_{s=1}^n r_se^{ i\theta_s} z_s )}{ \left[ \sum_{s=1}^n r_s^2
(A_s^2\cos^2\theta_s+B_s^2\sin^2\theta_s) \right]^{1/2}}  <  b  \}
= m_{A,B,r}(a,b)\,.
$$
If $A_1\not=B_1, \dots, A_n\not=B_n$, then the measure $\mu$ has the following form: there exists a rotationally invariant measure
$\rho$ on $\C^n$ such that
$$
\mu =  (T_{A,B}^{-1})_\ast (\rho)\,.
$$
\end{lemma}

\end{document}